\documentclass[11pt,reqno]{amsart}
\usepackage{amsthm,amsfonts,amssymb,euscript,mathrsfs,graphics, color, amsmath, amssymb, latexsym, marginnote,color}
\usepackage[dvips]{graphicx}
\usepackage[margin=1.05in]{geometry}

\usepackage{hyperref}

\def\Z{{\mathbb Z}}

\def\R{{\mathbb R}}

\def\e{{\varepsilon}}
\def\g{{\gamma}}

\def\r{{\rho}}

\def\b{{\beta}}

\def\what{\widehat}

\def\Im{\mbox{Im}}
\def\Re{\mbox{Re}}

\def\supp{\,\mbox{supp}\,}

\def\+R{+_{_{ \!\! \R}}}

\def\ve{{\varepsilon}}

\def\bar{\overline}

\def\wt{\widetilde}

\def\tofill{\vskip30pt $\cdots$ To fill in $\cdots$ \vskip30pt}

\newcommand{\comment}[1]{\vskip.3cm
\fbox{%
\parbox{0.93\linewidth}{\footnotesize #1}}
\vskip.3cm}
\numberwithin{equation}{section}

\begin{document}

%\title{On the blow-up at infinity for quasilinear wave equations}
\title[Weak null quasilinear waves]{On the global behavior of weak null \\ quasilinear wave equations}

%Other possible alternative titles:
%- On the global behavior of solutions of (Alinhac's-Lindblad) quasilinear wave equations.
%- On the global behavior of weak null quasilinear wave equations.

\author{Yu Deng}
\address{Courant Institute of Mathematical Sciences}
\email{yudeng@cims.nyu.edu}

\author{Fabio Pusateri}
\address{Princeton University}
\email{fabiop@math.princeton.edu}

%\date{\today}

\newtheorem{theorem}{Theorem}[section]
\newtheorem{lemma}[theorem]{Lemma}
\newtheorem{proposition}[theorem]{Proposition}
\newtheorem{corollary}[theorem]{Corollary}
\newtheorem{definition}[theorem]{Definition}
\newtheorem{remark}[theorem]{Remark}
\newtheorem{conjecture}[theorem]{Conjecture}
\newtheorem{convention}{Convention}
\numberwithin{equation}{section}

%\thanks{Y. Deng was supported in part by xxx. F. Pusateri was supported in part by xxx.}

\begin{abstract}
We consider a class of quasilinear wave equations in $3+1$ space-time dimensions
that satisfy the ``weak null condition'' as defined by Lindblad and Rodnianski \cite{LR1},
and study the large time behavior of solutions to the Cauchy problem.
The prototype for the class of equations considered is $-\partial_t^2 u + (1+u) \Delta u = 0$.
Global solutions for such equations have been constructed by Lindblad \cite{Lindblad1,Lindblad2} and Alinhac \cite{Alinhac1}.
Our main results are the derivation of a precise asymptotic system with good error bounds,
and a detailed description of the behavior of solutions close to the light cone, including the blow-up at infinity.
%In particular, we obtain an asymptotic expansion describing how solutions blow-up in infinite time.
%The weak null condition is conjectured to be a sufficient condition for global regularity.
%Resonances, modified scattering, asymptotic system etc...
\end{abstract}

\maketitle

\setcounter{tocdepth}{1}

\begin{quote}%\footnotesize
\tableofcontents
\end{quote}

\section{Introduction}

%\subsection{The problem and background}

In this paper we are interested in the global behavior of nonlinear wave equations in $3+1$ space-time dimensions.
%The question of global-in-time existence and asymptotic behavior of small solutions to nonlinear wave equations
%has been a subject under active investigation for over fifty years, due to its importance in various branches of physics.
%such as sound propagation, elasticity, general relativity.
Our main focus is the description of the asymptotic behavior for small solutions of the Cauchy problem
for a class of nonlinear wave equations satisfying the so-called ``weak null condition'' \cite{LR1,LR2}.
The prototype that we are going to consider is the equation
\begin{align}
\label{WNW00}
-\partial_t^2 u + (1+u) \Delta u = 0.
\end{align}

\smallskip
\subsubsection*{Background and the Weak Null Condition Conjecture}
One area of research on nonlinear wave equations where major progress has been made %over the past 30 years
focuses on identifying nonlinearities that lead to global solutions for small initial data.
Of particular interest is the case of quadratic nonlinearities, that is, systems of the form
\begin{equation}
\label{quadraticsystem0}
\Box u_i = \sum a_{i, \alpha \beta}^{jk} \partial^\alpha u_j \partial^\beta u_k, %\,\, + \,\, \mbox{cubic terms} %\qquad i = 1,\dots,N
\end{equation}
where $\Box = -\partial_t^2 + \Delta$, $i = 1,\dots,N$ for some positive integer $N$, and the sum runs over $j,k=1,\dots,N$,
and all multi-indices $\alpha,\beta \in \Z_+^4$ with $|\alpha|,|\beta| \leq 2$, $|\alpha| + |\beta| \leq 3$,
with the usual convention that $\partial_0 = \partial_t$.
Indeed, in $3$ spatial dimensions general quadratic nonlinearities can have long range effects:
Since solutions of the linear wave equation decay uniformly in space at best at the rate of $t^{-1}$,
the $L^2$ norm of the nonlinearity computed on a linear solution also decays at the borderline non-integrable rate of $t^{-1}$.
In particular, quadratic nonlinearities can contribute to the long time behavior of solutions and even cause finite-time blowup.
It is in fact known since the pioneering works of John \cite{John0,John2} that this latter scenario can occur even for solutions with small, smooth and localized data.
At the same time, for some very general classes of quadratic nonlinearities, solutions
were shown to exist almost globally, that is, for times of the order $\exp(c/\e)$, where $\e$ is the size of the Cauchy data,
by John and Klainerman \cite{JK} and Klainerman \cite{K0}.

The main breakthrough in identifying classes of nonlinear wave equations
where solutions with small data exist globally and scatter was in the works of Klainerman \cite{K1,K00},  %Choquet-Bruhat and Christodoulou \cite{C2},
see also Christodoulou \cite{C1}, on the {\it null condition}, which we will refer to as (NC).
The class of nonlinearities that satisfy (NC) %, as introduced by Klainerman \cite{K1},
have the form
\begin{equation}
\label{quadraticsystem}
\Box u_i = \sum_{ \substack{1 \leq |\alpha|,|\beta| \leq 2 \\ |\alpha|+|\beta| \leq 3}}
      a_{i, \alpha \beta}^{jk} \partial^\alpha u_j \partial^\beta u_k  % +  \,\, \mbox{cubic terms}  %\qquad i = 1,\dots,N
\end{equation}
where the following condition holds:
\begin{equation}
\label{nullcondition}
\sum%_{2 \leq |\alpha|,|\beta| \leq 3}
  a_{i, \alpha \beta}^{jk} \xi_\alpha \xi_\beta = 0
  \qquad \mbox{for any $\xi \in \R^4$ such that $-\xi_0^2 + \xi_1^2 + \xi_2^2 + \xi_3^2 = 0$}.
\end{equation}
For such systems it was shown by Klainerman \cite{K1} that in $3+1$ dimension solutions with small and a localized data exist globally.
%This seminal work of  Klainerman is based on the invariance of Minkowski space under the
%Lorentz group and on energy estimates using the vector fields that generate the Lorentz group \cite{K0}.
Moreover, (NC) leads to a linear asymptotic behavior of solutions as $t\rightarrow \infty$. 
%See also the work by Shatah and the author \cite{nullcondition}.

On the other hand, while (NC) is sufficient for global existence, it can be easily seen to not be necessary as the simple example
\begin{equation}
\label{weaknullex}
\left\{
\begin{array}{l}
\Box u_1 = (\partial_tu_2)^2
\\
\Box u_2 = 0
\end{array}
\right.
\end{equation}
shows. Since \eqref{weaknullex} is a decoupled system, one can trivially solve globally-in-time the equation for $u_1$.
One should note however that the behavior of $u_1$ is different from that of a linear solution.

In \cite{LR1,LR2} Lindblad and Rodnianski introduced a weaker notion than (NC), that of the {\it weak null condition}, which we will refer to as (WNC).
This concept was later instrumental to showing the stability of Minkowski space for the Einstein's vacuum equations in harmonic gauge \cite{LR2,LR3}.
To describe (WNC) consider a system of the form \eqref{quadraticsystem}, which we rewrite for simplicity as
\begin{align}
\label{Wsys}
\partial_t^2 u - \Delta u = Q(u,\partial u, \partial^2 u).
\end{align}
Making the ansatz
\begin{align}
\label{ansatz}
u(t,x) \approx \frac{\varepsilon}{|x|} U(q, s, \omega), \qquad q = t-|x|, \quad s = \e \log t, \quad \omega = x/|x|,
\end{align}
one can derive, at least formally, up to faster decaying remainders, an {\it asymptotic PDE} for $U$.
We provide some examples of such asymptotic PDEs in the next paragraph.
This type of asymptotic PDE was introduced by H\"ormander \cite{Hor1,Hor2} to study the blow-up time for scalar wave equations 
violating the null condition. %See Alinhac \cite{Alinhac00} for the completion of .

\iffalse
For example:
\begin{align*}
\begin{array}{ll}
\Box u = u_t^2 - |\nabla u|^2 &  \leadsto \qquad \partial_s U_q = 0,
%\\
%\\
%\Box u = u_t^2 \qquad &  \rightarrow \qquad \partial_s U_q + U_q^2 = 0,
\\
\Box u = u_t \, \Delta u \qquad &  \leadsto \qquad \partial_s U_q + U_q U_{qq} = 0,
\\
\Box u = u \, \Delta u &  \leadsto \qquad \partial_s U_q + U \, U_{qq} = 0.
\end{array}
\end{align*}
\fi

The system \eqref{Wsys} is said to satisfy (WNC) if the asymptotic PDE for the ansatz \eqref{ansatz} admits a global solution defined for all $s$,
which, together with its derivatives, grows at most exponentially in $s$ (algebraically in $t$).
An important conjecture in the field is the following\footnote{It is natural to attribute this conjecture to Lindblad and Rodnianski,
following \cite{LR1,LR2}. However, these same authors commented in \cite[p. 1405]{LR3}
``In our previous work \cite{LR1} we identified criteria under which it is more likely that a quasilinear system of the form 
\eqref{quadraticsystem} has global solutions [\dots]
At this point, it is unclear whether this criterion is sufficient for establishing a ``small data global
existence'' result for a {\it general} system of quasilinear hyperbolic equations.''}

\begin{conjecture}\label{WNCconj}
(WNC) is a sufficient condition for the global regularity of the Cauchy problem with small and localized data.
\end{conjecture}

To give an idea of the definition of (WNC) consider the following three examples of scalar nonlinear wave equations, and their associated asymptotic PDEs:
\begin{align}
\label{exs}
\begin{array}{lcl}
\Box u = u_t^2 - |\nabla u|^2  \qquad & \leadsto & \qquad \partial_s \partial_q U = 0,
\\
\Box u = u_t \, \Delta u & \leadsto & \qquad \partial_s \partial_q U + \partial_q U \partial_q^2 U = 0,
\\
\Box u = u \, \Delta u & \leadsto & \qquad \partial_s \partial_q U + U \, \partial_q^2 U = 0.
\end{array}
\end{align}
The first equation is an example (the only scalar one) of the classical (NC); when (NC) is satisfied
one always obtains the trivial asymptotic PDE $\partial_s \partial_q U = 0$, which obviously has global and uniformly bounded solutions.
The second example is Burgers' equation $\partial_s v + v \partial_q v = 0$, for $v = \partial_q U$;
solutions to this PDE blow-up in finite time, and so do solutions of the corresponding wave equation.
See the works \cite{Hor2,Alibook,Alinhac00,Cshocks,HKSWshocks,speck} and references therein
for extensive studies of blow-up for geometric classes of hyperbolic wave equations %the form $-\partial_t^2 + g(\partial_u) \partial^2 u$
violating (WNC).

The last of the three examples in \eqref{exs} is the prototypical equation \eqref{WNW00}.
For this equation the construction of global solutions for small $C_0^\infty$ data was done by Lindblad \cite{Lindblad1} in the radial case.
Global solutions where obtained by Alinhac in \cite{Alinhac1} in the non-radial case.
In \cite{Lindblad2} Lindblad treated the more general case of \eqref{WNW1} below.
%To obtain our results on asymptotics and nonlinear scattering we are going to use some of the a priori bounds on weighted Sobolev norms in \cite{Lindblad2}.
%However, we are only going to use some of the bounds in \cite{Lindblad2}, namely those involving one single scaling vectorfield
%and several spatial rotation vectorfields.
One way to see the difference with respect to the previous asymptotic PDEs is to set again $v = \partial_q U$, obtaining
$\partial_s v + \partial_q^{-1} v \partial_q v = 0$; in this equation $v$  is transported by $\partial_q^{-1} v$
which smooths out the contribution of high frequencies, and prevents the formation of shocks by the intersection of characteristic.

The explicit global solution of the asymptotic PDE
\begin{equation}
\label{asyPDE0}
\partial_s \partial_q U + U \, \partial_q^2 U = 0,
\end{equation}
with initial data $U(0) = \varphi$, $\psi:=\partial_r\varphi$, is given by
\begin{equation}
\label{asyPDEsol}
U(s;z(s,\alpha)) = \frac{1}{2} \left[ \int_{-\infty}^\alpha - \int_\alpha^\infty \right] e^{- \psi(\beta)s} \psi(\beta) \,\mathrm{d}\beta
\end{equation}
with
\begin{equation}
\label{asyPDEsol2}
z(s,\alpha) = \frac{1}{2} \left[ \int_\alpha^\infty - \int_{-\infty}^\alpha \right] e^{- \psi(\beta)s} \,\mathrm{d}\beta,
\end{equation}
from which one sees exponential growth in (logarithmic) time $s$ for derivatives of $\partial_q U$.

\smallskip
\subsubsection*{Motivation and results}
Our interest in the asymptotic behavior of solutions to \eqref{WNW00}, as well as the more general equation \eqref{WNW1}, is twofold.
First, understanding the global behavior of solutions of nonlinear hyperbolic and dispersive equations is one of the main objectives in the field.
For \eqref{WNW00} and \eqref{WNW1} this question has been left open since the cited works of Lindblad and Alinhac. See also H\"ormander \cite{Hor1} % Best reference?
and John \cite{John3}.
Recently, several results on asymptotics for hyperbolic PDEs have been obtained.
In \cite{LindbladAsy} Lindblad has proven asymptotics for small solutions of Einstein's vacuum equations in General Relativity.
Semilinear models for Einstein's equations have been studied in \cite{LinSch}, and the Maxwell--Klein-Gordon system was considered in \cite{CKL}.
Other works on asymptotics for semilinear wave equation include \cite{Kata,KK} and references therein.
See also the recent work \cite{IPMKG1} %for the study of global solutions and asymptotics for
on modified scattering for coupled Wave--Klein-Gordon systems modeling the Einstein--Klein-Gordon system.
We remark that, in the context of quasilinear waves, \eqref{WNW1} is conceptually harder to treat than Einstein's equations 
%(and its coupling with matter fields) 
despite its much simpler look;
indeed, Einstein's equations satisfy (WNC) essentially by having a nilpotent structure as \eqref{weaknullex} (with additional null forms and cubic terms).

Second, it is well established that, especially for quasilinear evolution equations, the question of global regularity is intimately tied to
a precise understanding of the asymptotic behavior of solutions.
Some important examples in this respect are the global stability of Minkowski space in General Relativity \cite{CK,LR3}, % argue better ?
the stability of one dimensional interfaces for gravity waves \cite{IoPu2,ADa,IT} and capillary waves \cite{IoPu4,IT2},
and the stability of two dimensional interfaces for the full gravity-capillary water waves system \cite{DIPP}.
Thus, one of the main reasons for our work is a contribution to the understanding of the Weak Null Condition Conjecture \ref{WNCconj}
through the study of the asymptotic behavior of solutions of \eqref{WNW00} and \eqref{WNW1}.
To this end, we should also remark that since solutions of the asymptotic system associated to \eqref{WNW00} grow exponentially in $s$,
this equation is supposed to represent the hardest example within the class of equations satisfying (WNC).
In fact, our techniques and results can be adapted to general systems satisfying (WNC).
%Moreover, the use of the recently developed analytical techniques that we employ and extend in this work, appears to be new in the context of nonlinear
%wave equations.

\smallskip
To end this introduction, let us give a brief and informal description of our results.
Precise statements are given in Theorem \ref{Mainth1} (the derivation of an asymptotic system)
and Theorem \ref{Mainth2} (exact asymptotics along the light cone). Some ideas of the proof are given in Subsection \ref{ssecStra}.
We start by considering a global small solution $u$ of \eqref{WNW00} as constructed by Lindblad \cite{Lindblad2},
and assume a priori some of the bounds obtained in \cite{Lindblad2} on weighted $L^2$ norms of the solution.
We then improve these bounds by iterating Duhamel's formula in Fourier space
and obtain almost sharp bounds for the Fourier transform of the profile $f$ associated to the solution $u$, 
see \eqref{profile}.
%More precisely we assume slow growth in time at the rate $(1+|t|)^{C\e}$
%for the $L^2$ norm of one scaling vectorfield $S=x \cdot \nabla_x + t \partial_t$ and several spatial rotation vectorfields $x \wedge \nabla_x$ applied to
%(regular derivatives) of $u$. Here $\e$ is the size of the initial data.
A detailed analysis of the bilinear interactions, depending on the size and direction of the frequency,
leads to the derivation of an ``asymptotic ODE'' for $\what{f}$ with good control on the error terms.
Thanks to this we can then construct
(a) an {\it asymptotic profile } $\mathcal{U}: (s,\rho,\theta) \in \R \times \R \times \mathbb{S}^2 \mapsto \R$
which satisfies the asymptotic PDE \eqref{asyPDE0} with suitable initial conditions, %$u(0,x)$,
and (b) a {\it shift} function $\mathcal{A}: (s,\theta) \in \R \times \mathbb{S}^2 \mapsto \R$,
which depends nonlinearly on $f$, such that, for $|x| \approx |t| \gg 1$ we have
\begin{align}
\label{infmain}
v(t,x) \approx \frac{1}{|x|} \mathcal{U} \big( \log|t|, t-|x| + \mathcal{A}(t,x/|x|), x/|x| \big).
\end{align}
%\comment{Fabio: Adjust this at the end if $\mathcal{A}$ is not given in term of $\mathcal{U}$\dots}
We remark that our results also apply to the general case
\begin{align}
\label{WNW1}
-\partial_t^2 u + g_{\alpha\beta}(u) \partial^\alpha\partial^\beta u = 0
\end{align}
with $g_{\alpha\beta}$ smooth and $g_{\alpha\beta}(0) = 1$.
For simplicity we will work with \eqref{WNW00} but it will be apparent that all our arguments can be carried out for \eqref{WNW1} as well.
%See Remark \ref{RemModels} for more about this.

%\comment{Fabio: Point out in a few places with some remarks where the small changes are}

\smallskip
\subsubsection*{Organization of the paper}
The set up of the equation, Duhamel's formula in Fourier space, and basic definition are given in the first part of Section \ref{secsetup}.
In Subsection \ref{secStat} we give the statements of our main results, and in \ref{ssecStra} we sketch some of the ideas in the proofs.
Section \ref{secPre} contains several auxiliary lemmas. 
In \ref{secBB} we first recall the a priori bounds on weighted energies of \cite{Lindblad2} and derive some basic consequences. 
Lemma \ref{LemAIBP} contains a result about angular integration by parts in oscillatory integrals 
that will be repeatedly used to restrict bilinear interactions to nearly parallel frequencies.
In \ref{secImp} we obtain some key improved bounds on the spatial Fourier transform of the profile and localized version of it.
Section \ref{SecProof1} is dedicated to the proof of Theorem \ref{Mainth1};
this is subdivided into five lemmas which are stated in Subsection \ref{secMainLemmas}.
Finally, exact asymptotics as in the statement of Theorem \ref{Mainth2} are derived in Section \ref{SecProof2}.

\bigskip
\section{Set up and statements}\label{secsetup}

\subsection{Duhamel's formula}\label{setup}
Throughout this paper the Fourier transform is defined as
\begin{align}
\label{F}
\widehat{f}(\xi)=(2\pi)^{-d/2}\int_{\mathbb{R}^d}e^{-ix\cdot\xi}f(x)\,\mathrm{d}x
\end{align}
in dimension $d$.
We fix a number of parameters
\begin{align}
\label{param}
\begin{split}
& N>N_1 \gg 1, \quad \gamma\gg\delta:=N^{-1}, \qquad (N=1350, \gamma=1/90, N_1=N/3=450),
\end{split}
\end{align}
and a solution $u \in C([0,\infty),H^{N+1}(\R^3))$ to
\begin{equation}
\label{Wave}
\Box u = u\Delta u,
\end{equation}
with initial data
\begin{align}
\label{Wavedata}
(u(0),\partial_tu(0)) = (\e \varphi^0,\e \psi^0),
\end{align}
where $0<\varepsilon \ll 1$ and $\varphi^0$ and $\psi^0$ are fixed Schwartz functions.
In fact, we only need $\varphi_0$ and $\psi_0$ to belong to a suitable weighted Sobolev space
for which the energy estimate in Lindblad \cite{Lindblad2} hold, see \eqref{Ali0}-\eqref{Aliinfty}.  %\comment{Hypotheses from Lindblad}

We define
\begin{equation}
\label{profile}
v :=(\partial_t-i|\nabla|)u,\qquad f(t) :=e^{it|\nabla|}v(t).
\end{equation}
Since
\begin{equation}
\label{profileeq}
(\partial_t+i|\nabla|)v=u\Delta u, \qquad \partial_t f := e^{it|\nabla|} (u\Delta u),
  \qquad u = \frac{i(v-\bar{v})}{2|\nabla|} = -\frac{1}{|\nabla|}\Im (v),
\end{equation}
we have the Duhamel's formula
\begin{equation}\label{duhamel}
\widehat{f}(t,\xi)-\widehat{f}(0,\xi)
= \sum_{\kappa_1,\kappa_2\in\{+,-\}}\kappa_1\kappa_2J_{\kappa_1\kappa_2}[f_{\kappa_1},f_{\kappa_2}](t,\xi),
\end{equation}
where we denote
$$f_{+}=f, \qquad f_{-}=\bar{f},$$
and
\begin{align}
\label{duhamel2}
\begin{split}
J_{\kappa_1\kappa_2}[f,g](t,\xi) &:= \int_0^t I_{\kappa_1\kappa_2}[f,g](s,\xi) \, \mathrm{d}s,
\\
I_{\kappa_1\kappa_2}[f,g](t,\xi) &:= \frac{1}{4}(2\pi)^{-3/2}
\int_{\R^3} e^{it\Phi_{\kappa_1\kappa_2}(\xi,\eta)}\frac{|\xi-\eta|}{|\eta|}
\widehat{f}(t,\xi-\eta)\widehat{g}(t,\eta)\,\mathrm{d}\eta,
\\
\Phi_{\kappa_1\kappa_2}(\xi,\eta) &:= |\xi|-\kappa_1|\xi-\eta|-\kappa_2|\eta|.
\end{split}
\end{align}
The functions $\Phi_{\kappa_1\kappa_2}$, are usually called {\it phases} and  measure the quadratic interactions between waves.

%\begin{remark}\label{RemModels}
%When considering the general equation \eqref{WNW1} instead of \eqref{Wave}, one obtains a formula similar to \eqref{duhamel2}.
%More precisely, if $g_{\alpha\beta}(u) = 1 + g_{\alpha\beta}u + O(u^2)$, $g_{\alpha\beta}\in\R$, and disregarding the cubic terms in the equation, the multiplier $|\xi-\eta|/|\eta|$ in \eqref{duhamel2} is replaced by $g_{\alpha\beta}(\xi-\eta)^\alpha(\xi-\eta)^\beta/(|\xi-\eta||\eta|)$.
%This multiplier obviously has the same homogeneity and degree of vanishing as the on in \eqref{duhamel2}.
%See Remark \eqref{RemModels2} to see how this symbols does not affect the main calculation in\dots
%\end{remark}

\subsection{Notation}
We choose a suitable decomposition of the indicator function $\mathbf{1}_{[0,\infty)}$ by fixing functions $\tau_0,\tau_1,\cdots : \R \to [0,1]$, with the properties
\begin{align}
\label{timedecomp}
 \mathrm{supp}\,\tau_m\subseteq [2^{m-1},2^{m+1}],\quad \sum_{m=0}^{\infty}\tau_m(t) = 1,  \quad |\tau_m^\prime(t)|\lesssim 2^{-m}.
\end{align}
We use this to decompose (omitting the dependence on the inputs for a lighter notation)
\begin{align}
\label{duhamel3}
\begin{split}
J_{\kappa_1\kappa_2}(t,\xi) = \sum_{m=0}^{\infty} J^{m}_{\kappa_1\kappa_2}(t,\xi), \qquad
J^{m}_{\kappa_1\kappa_2}(t,\xi) := \int_0^t \tau_m(s)I_{\kappa_1\kappa_2}(s,\xi) \, \mathrm{d}s
=: \int_0^t I^{m}_{\kappa_1\kappa_2}(s,\xi) \, \mathrm{d}s.
\end{split}
\end{align}

$C$ will generally denote an absolute constant that may vary from line to line.
The notation $A \lesssim B$ means that $A\leq CB$ for some absolute constant $C>0$; we will use $\gtrsim$ and $\approx$ in a similar standard way.

We denote the space-time gradient $\partial=(\partial_t,\nabla)$.
The rotation vector fields are $\Omega=(x_i\partial_j-x_j\partial_i)_{1\leq i<j\leq 3}$,
and the scaling vector field $S=t\partial_t+x\cdot\partial_x$.

We fix a smooth radial cutoff function $\varphi$ that equals $1$ for $|\xi|\leq 1$ and vanishes for $|\xi|\geq 2$, and define
\begin{align}
\label{LPdef}
\begin{split}
& \varphi_k(z)=\varphi(2^{-k}z)-\varphi(2^{-k+1}z), \quad \varphi_I(z) := \sum_{j \in I\cap \Z}\varphi_j(z) \quad \textrm{for any $I\subset\R$}
\\
& \varphi_{\leq A}(z) := \varphi_{(-\infty, A]}(z), \quad \varphi_{>A}(z) = \varphi_{(A,\infty)}(z), \quad \textrm{ etc.}
\end{split}
\end{align}
We let $P_k$ denote the standard Littlewood-Paley projections, $\widehat{P_kf}(\xi):=\varphi_k(\xi)\widehat{f}(\xi)$, and let $Q_{jk}$ be defined,
for $(j,k)\in\mathcal{J}:=\{(j,k):j\geq \max(-k,0)\}$, as
\[
Q_{jk}f(x)=P_{[k-2,k+2]}(\varphi_j^{(k)}\cdot P_kf)(x),\quad \varphi_j^{(k)}=
\left\{
\begin{array}{ll}
\varphi_j(x),&j>\max(-k,0);
\\
\varphi_{\leq j}(x),&j=\max(-k,0).
\end{array}
\right.
\]
For any $k\in\mathbb{Z}$ we will define $k^+=\max(k,0)$ and $k^-=\max(-k,0)$.

\subsection{Definition of Acceptable Error}
Our main objective is to find the leading order asymptotics in \eqref{duhamel2}.
As it will turn out this will require the derivation of a nonlinear PDE that drives the asymptotic dynamics, as well as a phase correction/shift.
We then need a proper definition of acceptable remainder terms. The norm in which we will measure our remainders is:
\begin{align}
\label{Xnorm}
{\| g \|}_{X} := \sup_{|\ell|\leq N_1}\sup_{k\in\mathbb{Z}}2^{k+15k^+}\big\|\varphi_k(\xi)\Omega^{\ell}g(\xi)\big\|_{L^\infty_\xi}.
\end{align}

\begin{definition}[Acceptable remainder]\label{def}
We will deem a function $R(t,\xi)$ to be an acceptable remainder if we can write is as
\begin{align}
\label{defRem}
\begin{split}
& R(t,\xi) = R_1(t,\xi) + R_2(t,\xi),
\\ & {\| R_1(t) \|}_X \lesssim \e^2 \langle t\rangle^{-1-\g},
\qquad  R_2 = \partial_t R' \quad \mbox{with} \quad {\|R'(t)\|}_X \lesssim \e^2 \langle t\rangle^{-\g}.
\end{split}
\end{align}
%Hope in stronger estimate for the $\mathcal{S}$ type terms.
We will write $R=R(t,\xi) \in \mathcal{R}$ if it is an acceptable remainder in the above sense.
%Similarly, we will write $R \in \mathcal{R}_1$ or $\mathcal{R}_2$ if it satisfies the first or second property in \eqref{defRem}.
%Note that for $S$ as above $t^{-1}S \in \mathcal{R}_1$.
\end{definition}

\medskip
\subsection{Statements of the main results}\label{secStat}

Our first main Theorem gives convergence of solutions of \eqref{Wave} to an asymptotic system.
%with control of the remainder in a critical norm.

\begin{theorem}[Approximation by the asymptotic system]\label{Mainth1}
Let $u$ be a solution of the equation \eqref{Wave}-\eqref{Wavedata}, and let $f$ be the associated profile in \eqref{profile}.
Denote spherical coordinates by $\xi = \rho \theta$, $\rho>0$,$\phi \in\mathbb{S}^2$.
Define the ``radial profile''
\begin{align}
\label{Mainth11}
\begin{split}
F_\theta(t,\rho) &:= \widehat{f}(t,\rho\theta),\quad \rho\geq 0;\qquad  F_\theta(t,\rho) := \overline{F_\theta(t,-\rho)},\quad \rho<0,
\end{split}
\end{align}
and the quantities
\begin{align}
\label{Mainth12}
\begin{split}
h_\theta(t,\rho) &:=  \int_\R \int_{\mathbb{S}^2_\phi} e^{it \rho[1-\theta \cdot \phi]} \big| F_\phi(t,r) \big|^2 r^2
  \,\mathrm{d}\phi \mathrm{d}r,
\\
H_\theta(t,\rho) &:= \varphi_{\leq-10}(\rho\langle t \rangle^{7/8}) \int_0^t h_\theta(s,\rho)\,\mathrm{d}s,
\end{split}
\end{align}
and
\begin{align}
\label{Mainth13}
\begin{split}
B_\theta(t) &:= \frac{1}{32\pi^2}
  \mathrm{Re} \Big[ \int_\R \int_{\mathbb{S}^2_\phi} %\varphi_{\leq 0}(r\langle t \rangle^{1-\delta})
  e^{itr[1-\theta \cdot \phi]} H_\phi(t,r) \,r \,\mathrm{d}\phi \mathrm{d}r \Big],
\\
C_\theta(t) &:= B_\theta(t)-\frac{1}{32\pi^2}\frac{1}{t}\int_{\mathbb{R}}\varphi_{\leq -10}(\rho\langle t\rangle^{7/8})H_\theta(t,\rho)\,\mathrm{d}\rho.
%cutoff redundant, just to keep it in mind
\end{split}
\end{align}

Then $F_\theta$ satisfies the ``asymptotic PDE''
\begin{align}\label{asymptpde}
\begin{split}
\partial_t F_\theta(t,\rho) &=\frac{1}{2(2\pi)^{3/2}}\varphi_{\leq0}(\rho\langle t \rangle^{7/8})\cdot h_{\theta}(t,\rho) -i\rho C_\theta(t) F_\theta(t,\rho) \\&+ \frac{1}{it}\frac{1}{4(2\pi)^{1/2}}
  \int_\R \frac{(\rho-r)^2}{\rho} F_\theta(t,\rho-r)F_\theta(t,r) \, \mathrm{d}r+ \mathcal{R}(t,\xi),
\end{split}
\end{align}
where %$\xi=\rho\theta$, and
$\mathcal{R}$ is an acceptable remainder in the sense of Definition \ref{def}.
\end{theorem}

The proof of Theorem \ref{Mainth1} will follow from the approximation Lemmas \ref{Lem--}-\ref{LemAux},
which are stated in \ref{secMainLemmas} and proven in the rest of Section \ref{SecProof1}.
Combining Theorem \ref{Mainth1} with additional linear and nonlinear estimates,
%the linear estimates in Lemma \ref{LemLinear}, and the improved bounds of Lemma \ref{Lemj=m}, %?
we will obtain the following result:

\begin{theorem}[Nonlinear Scattering and Asymptotics]\label{Mainth2}

Consider a solution $u$ of \eqref{Wave}. For $||x|-t|\lesssim t^{\gamma/10}$ the asymptotic behavior of $u(t,x)$ as $t\to \infty$, is as follows:
there exists a function (``the asymptotic profile'')
$\widetilde{U}=\widetilde{U}_\theta(s,q)$, satisfying
\begin{align}
\label{Mainth21}
\sup_{|\alpha|\leq 14,|\beta|\leq N_1}\|\nabla^{\alpha}\Omega^{\beta}\partial_q\widetilde{U}_\theta(s)\|_{L_q^2\cap L_q^\infty}
  \lesssim \e e^{C\e s},
\end{align}
and the equation
\begin{align}
\label{Mainth22}
\partial_s\partial_q\widetilde{U}_\theta+\widetilde{U}_\theta\cdot\partial_q^2\widetilde{U}_\theta=0,
\end{align}
and a function $D_\theta(t)$ (the ``shift'') satisfying
\begin{align}
\label{Mainth23}
\sup_{|\alpha|\leq N_1}|\partial_\theta^\alpha D_\theta(t)|\lesssim \varepsilon(1+|t|)^{C\varepsilon},
\end{align}
such that
\begin{align}
\label{Mainth24}
\sup_{|\alpha|\leq 10,|\beta|\leq N_1} \Big|\nabla^{\alpha}\Omega^{\beta}\partial \big[ u(t,x)-
  2(2\pi)^{1/2}\frac{1}{|x|} \wt{U}_\theta(\log t,|x|-t+D_\theta(t)) \big] \Big| \lesssim \e t^{-1-\gamma/30},
\end{align}
where $\theta = x/|x|$. %$\partial$ includes $\partial_t$
\end{theorem}

\smallskip
Let us make a few comments.
\setlength{\leftmargini}{2em}
\begin{itemize}

\smallskip
\item[-] Our result applies also to the general case of the equation \eqref{WNW1}
%\begin{align}
%-\partial_t^2 u + g_{\alpha\beta}(u) \partial^\alpha\partial^\beta u = 0
%\end{align*}
with $g_{\alpha\beta}$ smooth and $g_{\alpha\beta}(0) = 1$.
For simplicity we deal with \eqref{Wave} but it will be apparent that all our arguments apply verbatim to \eqref{WNW1}.
The only changes involve the dependence on the ``parameter'' $\theta$ in the formulas of Theorems \ref{Mainth1} and \ref{Mainth2}. 

\smallskip
\item[-] The asymptotic formula \eqref{Mainth22}--\eqref{Mainth24} follows from the asymptotic PDE \eqref{asymptpde}, whose derivation occupies most of our paper.
It is important to observe that the PDE \eqref{asymptpde}, with the explicit formulas \eqref{Mainth11}--\eqref{Mainth12},
actually contains more precise information about the asymptotic behavior of solutions of \eqref{Wave}.
%In particular we see a low frequency drifting term $h_\theta$, and a correction term $C_\theta$ coming from the accumulation of low frequencies effects.
However, in order to be able to fully exploit this, and obtain asymptotics inside the light cone for the solution $u$ itself,
it seems that one would need to strengthen the notion of acceptable remainder, and prove even stronger bounds than the ones in this paper.
This we plan to do in future work.

\smallskip
\item[-] Away from the light cone, when $||x|-t|\gtrsim (1+|t|)^{\gamma/10}$, one has
\[
\sup_{|\alpha|\leq 10,|\beta|\leq N_1}|\nabla^{\alpha}\Omega^{\beta}\partial u(t,x)|\lesssim \varepsilon(1+|t|)^{-1-\gamma/30},
\]
as a consequence of the standard Klainerman-Sobolev embedding, and the weighted bounds \eqref{Ali0}.

\smallskip
\item[-] Note that \eqref{Mainth24} with $\alpha=\beta=0$ implies the sharp $t^{-1}$ decay for $\partial u$ close to the light cone.
In particular, it follows that we can close energy estimate (with small-in-time growth) in standard Sobolev spaces.
Weighted energy estimates would instead require additional arguments, see also the comment below.

\smallskip
\item[-] As already mentioned, the construction of global solutions for \eqref{Wave} for small $C_0^\infty$ data was done by Lindblad \cite{Lindblad1}
in the radial case. Global solutions where obtained by Alinhac in \cite{Alinhac1}, and by Lindblad \cite{Lindblad2}
for the more general case of \eqref{WNW1}.
To obtain our results on asymptotics and nonlinear scattering we are going to use
some of the a priori bounds on weighted Sobolev norms from \cite{Lindblad2},
%However, note that we are only going to use some of the bounds in \cite{Lindblad2},
more precisely only those involving one single scaling vectorfield and several rotation vectorfields.

%In particular, by extending our methods to perform energy estimates for the weighted quantities that we use, it would be possible to also recover the result of \cite{Lindblad2} under weaker assumptions.

\smallskip
\item[-] We did not optimize the choice of the parameters in \eqref{param},
nor the fraction of $\gamma$ in \eqref{Mainth24}, but prioritized convenience instead.
For similar reasons, the exponent $7/8$ appearing in the definitions \eqref{Mainth12}--\eqref{Mainth13} 
is a convenient number slightly smaller than $1$.

%\item {Other comments?}

\end{itemize}

\medskip
\subsection{Ideas and structure of the proof}\label{ssecStra} %\label{ssecMt}
We now discuss the strategy for our proofs and some of the main features of our arguments.
We focus on describing the main intuition behind the derivation of the asymptotic system in Theorem \ref{Mainth1}
and the construction of the asymptotic profile in Theorem \ref{Mainth2}.
For the sake of explanation we will disregard many technical aspects of the proofs. 
The initial approach follows the general scheme of recent works on global solutions of (quasilinear) evolution systems, see \cite{GNT1,GMS2,IP2,IoPu2}.
In the context of wave equation this general framework was employed by Shatah and the second author \cite{nullcondition,nullconditionbumi}.
%cite{GMS2,nullcondition}, modified scattering \cite{KP,IoPu2,IoPu4} and works on quasilinear dispersive \cite{IP2,GIP,DIP,DIPP}.
%The starting point is the analysis of Duhamels' formula in Fourier space.
%One then applies several multilinear harmonic analysis techniques, atomic decompositions, and stationary phase type arguments.
Let $v$ and $f$ be defined as in \eqref{profile}.
For sufficiently regular and localized solutions, one has the following linear asymptotic for $|x| \approx t$:
\begin{align}
\label{vlin}
v(t,x) \approx \frac{c}{|x|} \int_{0}^\infty e^{i\rho(t-|x|)} \, \rho \, \what{f}\big(t,\rho x/|x|\big) \, \mathrm{d}\rho
\end{align}
up to faster decaying terms; see Proposition \ref{transfer0} for a precise formula.
Motivated by this, we look at the evolution of $\what{f}$, which reads, see \eqref{duhamel2},
\begin{align}
\label{Stra1}
& \widehat{f}(t,\xi) = \widehat{f}(0,\xi) + \sum_{\kappa_1,\kappa_2\in\{+,-\}} \frac{\kappa_1\kappa_2}{4 (2\pi)^{3/2}}
  \int_0^t \int_{\R^3} e^{it\Phi_{\kappa_1\kappa_2}(\xi,\eta)}\frac{|\xi-\eta|}{|\eta|}
  \widehat{f_{\kappa_1}}(t,\xi-\eta) \widehat{f_{\kappa_2}}(t,\eta)\, \mathrm{d}\eta \mathrm{d}s,
\end{align}
where we denote $f_+ = f$, $f_- =\bar{f}$, and
\begin{align}
\label{Stra2}
\Phi_{\kappa_1\kappa_2}(\xi,\eta) := -|\xi|+\kappa_1|\xi-\eta|+\kappa_2|\eta|.
\end{align}
%Equation \eqref{Stra1} is Duhamel's formula in Fourier space for $f$,
%and the functions $\Phi_{\kappa_1\kappa_2}$, usually called {\it phases},  measure the quadratic interactions between waves.

For simplicity, let us focus on one of the terms in the nonlinear expression \eqref{Stra1}-\eqref{Stra2},
the term with $(\kappa_1\kappa_2) = (+-)$. In other words, let us look at the equation
\begin{align}
\label{Stra3}
\partial_t \what{f}(t,\xi) \approx \int_{\R^3} e^{it(-|\xi|+|\xi-\eta|-|\eta|)} \frac{|\xi-\eta|}{|\eta|} \widehat{f}(t,\xi-\eta) \widehat{\bar{f}}(t,\eta)\, \mathrm{d}\eta
   \, + \, \mbox{similar terms}.
\end{align}
One can verify that the interacting waves have the same velocity, and therefore travel in the same direction, for frequencies in the set
\begin{align}
\label{Stra4}
%\mathcal{S} =
\{ (\xi,\eta) \in \R^3 \times \R^3: \, \xi = \lambda \eta, \quad \lambda \leq 1 \}.
\end{align}
For such frequencies the integral \eqref{Stra3} presents no oscillations in $\eta$, $\nabla_\eta \Phi_{+-} = 0$.
%\eqref{Stra4} is the so-called space-resonant set in the language of \cite{GMS2}.
In addition, on the subset
\begin{align}
\label{Stra5}
%\mathcal{T} =
\{ (\xi,\eta) \in \R^3 \times \R^3: \, \xi = \lambda \eta, \quad \lambda \leq 0 \},
\end{align}
there are no oscillations over time, as $\Phi_{+-} = 0$ as well.
%\eqref{Stra5} is the so-called time-resonant set \cite{GMS2} and in this case it coincides with
%the full resonant set $\mathcal{R} = \mathcal{S} \cap \mathcal{T}$.
Interactions between waves with parallel frequencies in the set \eqref{Stra5} are call ``resonant'' and are the most problematic to handle.
We refer the reader to the introduction of \cite{nullcondition} for a longer discussion along these lines.

Our general approach to obtaining bounds on solutions will be to first restrict the integral in \eqref{Stra3}
to the parallel frequencies in the set \eqref{Stra5} and then analyze the resulting contribution.
We remark that such contribution would vanish if (NC) was satisfied.
One of the things we are going to show is that, when the (WNC) is satisfied, the asymptotic PDE \eqref{asyPDE0}  %of Lindblad-Rodnianski
or, more precisely, the equation \eqref{asymptpde},
emerges from the reduction of \eqref{Stra1}-\eqref{Stra2} to resonant interactions. %, as it is natural to expect.
However, in order to be able to see the full asymptotic dynamics there are several important aspects one needs to deal with.
In the remaining of this section we give a schematic description of our main steps.

\smallskip
{\it Step 1: Improved bounds by iteration}.
We begin by translating the energy bounds from \cite{Lindblad2} on weighted $L^2$-norms of solutions of \eqref{WNW00} in terms of the profile $f$.
Roughly speaking these give us a bound of the form
\begin{align}
\label{Stra10}
{\| (1 + |x|) |\nabla| f \|}_{L^2} + {\| \Omega^N  |\nabla| f \|}_{L^2} \lesssim \e (1+|t|)^{C\e},
\end{align}
where $\Omega$ is any rotation vectorfield of the form $x_i \partial_{x_j} - x_j \partial_{x_i} $, $N$ is a sufficiently large number,
and $\e \ll 1$ is the size of the initial data.
Then, we obtain a number of stronger bounds by iterating the basic bound \eqref{Stra10} through Duhamel's formula \eqref{Stra2}-\eqref{Stra3}.
These bounds give good control on $\what{f}$, as well as improvements for localized versions of it.
See Lemmas \ref{LemBB}, \ref{lemImp} and \ref{Lemspe}. %and Lemma \ref{Lemj=m}.
A key tool that we often use is the angular integration by parts Lemma \ref{LemAIBP}.

%\comment{After we are done we can specify a bit more what type of bounds these are}

\smallskip
{\it Step 2: Low frequency outputs}.
It is important to observe that one cannot restrict the interactions to parallel frequencies when these frequencies are too small.
Therefore, we first look at the contribution of very small frequencies outputs $|\xi| \lesssim (1+|t|)^{-1+}$.
In this case, using the weighted and pointwise bounds obtained in the previous step, we show that
\begin{align}
\label{Stra11}
\partial_t \what{f}(t,\xi) \approx \frac{1}{t\,|\xi|}\int_{\R^3} |\widehat{f}(t,\eta)|^2 \, \mathrm{d}\eta + R := h(t,|\xi|) + R
\end{align}
where $R$  denotes here, and in what follows, a remainder which decays fast enough in time in a suitably strong norm, see Definition \ref{def}.
In other words, we see that low frequency outputs contribute a bulk term $h$ such that $|\xi| h(t,|\xi|)$ decays at a critical $t^{-1}$ rate.

\smallskip
{\it Step 3: High-low interactions and phase shift}.
The next step is to measure the feedback contribution of the low frequency bulk $h$ to the nonlinear interactions.
We then look at \eqref{Stra3} when $|\eta| \lesssim (1+|t|)^{-1+} \ll |\xi|$.
Formally replacing $\what{f}(\eta)$ by the contribution coming from $\int_0^t h(s,|\eta|) \, \mathrm{d}s =: H(t,|\eta|)$,
and approximating $\what{f}(\xi-\eta)$ by $\what{f}(\xi)$, we see that:
\begin{align}
\label{Stra12}
\partial_t \what{f}(t,\xi) \approx |\xi| \what{f}(\xi) \cdot i \int_{0}^\infty H(t,r) \, r \, \mathrm{d}r + R.
\end{align}
This shows that high frequencies are essentially transported by the low frequencies bulk.
Eventually we will renormalize $f$ to incorporate this contribution via a phase correction/shift, using the reality of $H$.
This is a similar phenomenon to the one occurring in scattering-critical equations such as 1d nonlinear Schr\"odinger equations \cite{HN,KP,IoPu1},
and $2$d water waves \cite{IoPu2,IoPu4}.
%See also the recent work \cite{IPMKG1} on asymptotics for coupled Wave--Klein-Gordon systems. %and the Maxwell--Klein-Gordon system \cite{IPMKG2}.

\smallskip
{\it Step 4: High-high interactions: Reduction to parallel frequencies}.
After treating very low frequency inputs and outputs we can isolate the contribution of parallel frequencies,
using the improved bounds obtained earlier, and angular integration by parts arguments.
%and similar arguments in \cite{Deng,DIPP}.
Denoting spherical coordinates by $\xi = \rho \theta$, $\eta = r\phi$, with $\theta,\phi \in\mathbb{S}^2$,
we can write Duhamel's formula \eqref{Stra3} as
\begin{align}
\partial_t \what{f}(t,\xi) & \approx \int_{0}^\infty \int_{\mathbb{S}^2_\phi} e^{it(-\rho+|\rho\theta - r\phi|- r)}
  \frac{|\rho\theta - r\phi|}{r} \widehat{f}(t,\rho\theta - r\phi) \widehat{\bar{f}}(t,r\phi)\, r^2 \, \mathrm{d}r \mathrm{d}\phi + R,
\end{align}
where the integral over $\mathbb{S}^2_\phi$ is supported in a small cap around $\theta$.
Then, a stationary phase argument on the sphere gives
\begin{align}
\begin{split}
\partial_t \what{f}(t,\xi) \approx \int_0^\infty \Big[ \int_{\mathbb{S}^2} e^{it(-\rho + |\rho\theta + r\phi|-r)} \, \mathrm{d}\phi \Big] (r-\rho)
  \widehat{f}(t, (\rho-r)\theta) \bar{\widehat{f}}(t,-r\theta)\, r \, \mathrm{d}r + R
\\
\approx \int_0^\infty \Big[ \frac{1}{it} \frac{(\rho-r)}{r\rho} \Big] (r-\rho)
  \widehat{f}(t, (\rho-r)\theta) \bar{\widehat{f}}(t,-\rho\theta)\, r \, \mathrm{d}r + R.
\end{split}
\end{align}
Thus, we have arrived at
\begin{align}
\label{Stra15}
\partial_t \what{f}(t,\rho \theta) &
  \approx \frac{1}{it \rho} \int_0^\infty (r-\rho)^2
  \widehat{f}(t, (\rho-r)\theta) \overline{\widehat{f}}(t,-\rho\theta)\, \, \mathrm{d}r + R.
\end{align}

\smallskip
{\it Step 5: The asymptotic PDE}.
Similar arguments to those sketched above for \eqref{Stra3} can be used for the other terms in \eqref{Stra1}-\eqref{Stra2}.
Gathering all the contributions, and defining the ``radial profile''
\begin{align}\label{Stra20}
F_\theta(t,\rho) & := \widehat{f}(t,\rho\theta), \quad \rho\geq0,\,\theta \in \mathbb{S}^2,
\qquad F_\theta(t,\rho) = \overline{F_\theta(t,-\rho)}, \quad \rho <0,
%\\
%G_\theta(t,\rho) & := \widehat{f}(t,\rho\theta), \quad \rho \gtrsim (1+|t|)^{-1+},
%  \qquad G_\theta(t,\rho) = \overline{G_\theta(t,-\rho)}, \quad \rho \lesssim -(1+|t|)^{-1+},
\end{align}
one obtains %for $\rho = |\xi|, r = |\eta| \gtrsim (1+|t|)^{-1+}$ the PDE
%\footnote{We are disregarding some unimportant absolute constants here}
\begin{align}
\label{Stra21}
\partial_t F_\theta(t,\rho) = -i\rho C_\theta(t) F_\theta(t,\rho) + \frac{1}{i4(2\pi)^{1/2}t}
  \int_\R \frac{(\rho-r)^2}{\rho} F_\theta(t,\rho-r)F_\theta(t,r) \, \mathrm{d}r + R,
\end{align}
where %$c := 4(2\pi)^{-1/2}$, and
%\begin{align}
%\label{StraB}
%B_\theta(t) = \int_{0}^\infty H(t,r) \, r \, \mathrm{d}r,
%\end{align}
$C_\theta(t)$ is a real-valued nonlinear function of $F_\theta$ which takes into account
the low frequencies contributions \eqref{Stra11}-\eqref{Stra12},
and we have disregarded some less important terms. See Theorem \ref{Mainth1} for details.
%Notice that since $H(t,r) = \int_0^t h(s,r) \, \mathrm{d}s$ is supported on $r \lesssim (1+|t|)^{-1+}$, we have $B_\theta(t) \approx (1+|t|)^{-1}$.
We then define $U_\theta = U(s,q;\theta)$ by setting
\begin{align}
\label{Stra22}
(\mathcal{F}_q U_\theta)(s,\rho) := F_\theta(e^s,\rho) e^{-i \rho D_\theta(s)},
\qquad D_\theta(s) := \int_0^{e^s} C_\theta(s') \, \mathrm{d}s',
\end{align}
and observe that equation \eqref{Stra21} is the $1$ dimensional Fourier transform of the equation
\begin{align}
\label{Stra23}
\partial_t \partial_q U_\theta = U_\theta \, \partial_q^2 U_\theta + \what{R}(e^s,q).
\end{align}
This is a perturbation of the asymptotic PDE \eqref{asyPDE0}. %upon changing $t = e^s$.
An important aspect of our analysis is the control of all the remainders $R$ in a proper ``critical norm'',  see \eqref{defRem}, % improve later ?
to guarantee a strong enough proof of convergence to the asymptotic system.

\smallskip
{\it Step 6: Asymptotics}.
The last step consists in rigorously proving convergence of $U_\theta$ to an asymptotic profile, which will still depends on time $s$, and deduce from there asymptotics for $u$.
Combining the solution of the perturbed asymptotic system \eqref{Stra23} for $U_\theta$ with
the improved bounds obtained by analyzing Duhamel's formula \eqref{Stra1},
we explicitly construct an asymptotic profile $\wt{U}_\theta = \wt{U}(s,q,\theta)$
that solves the unperturbed asymptotic system \eqref{asyPDE0} together with suitable initial conditions, and such that
\begin{align}
\label{Stra29}
\partial_q U_\theta - \partial_q \wt{U_\theta} \stackrel{s\rightarrow \infty}{\longrightarrow} 0,
\end{align}
with the convergence happening in a suitable norm.
Combining this with identities \eqref{Stra20} and \eqref{Stra22}
%relating $U_\theta$, $F_\theta$, and \eqref{Stra22}
and the linear asymptotic \eqref{vlin}, one obtains
%\begin{align}
%\label{Stra30}
%\Big\| \rho \widehat{f}(t,\theta\rho) - \exp \big(i\rho \mathcal{A}_\infty(t) \big)
%  \widehat{\mathcal{U}}_\infty(\log |t|,\rho) \Big\|_{L^2(\rho \geq \langle t \rangle^{\delta-1} )} \lesssim
%  \e^2 \langle t \rangle^{-\delta}, \quad \d > 0,
%\end{align}
\begin{align}
\label{Stra30.5}
v(t,x) \approx \frac{c}{|x|} \partial_q \wt{U} \big(s, |x|-t- D_{x/|x|}(\log t); x/|x| \big).
\end{align}
%which corresponds to an asymptotic behavior like \eqref{infmain}.
%where $\wt{D}_\theta$ is an asymptotic shift,
%which is essentially the limit of the factor appearing in the exponential of \eqref{Stra22}.
We refer the reader to %Theorem \ref{Mainth1}, Theorem \ref{Mainth2}, and 
Propositions \ref{finaldata} and \ref{transfer0}, and the formulas \eqref{defU}-\eqref{defU'} for more details on the asymptotic profile.

\bigskip
\section{A priori bounds and improvements}\label{secPre}

In this section we gather several supporting estimates that we are going to use in the proof of Theorem \ref{Mainth1} in the next section.
%In Subsection \ref{secBB}
We begin by recalling the weighted energy bounds from \cite{Lindblad2} and state some
basic consequences in terms of the profile \eqref{profile} in Lemmas \ref{lemEE} and \ref{LemBB}.
We then prove a general results about angular integration by parts in oscillatory integrals
in Lemma \ref{LemAIBP}; this will be used repeatedly in rest of this section and in Section \ref{SecProof1} 
to restrict many bilinear interactions to nearly parallel frequencies.
Finally, in Subsection \ref{secImp} we improve in various ways the basic energy bounds by iterating Duhamel's formula, 
and obtain several key bounds on the Fourier transform of $f$, see Lemmas \ref{lemImp} and \ref{Lemspe}.

\medskip
\subsection{Basic a priori bounds and norms}\label{secBB}
We will use the following energy bounds established by %Alinhac \cite{Alinhac1} and
Lindblad \cite{Lindblad2}:
\begin{equation}
\label{Ali0}
 \sup_{|\ell|+|\ell'|\leq N} {\| \nabla^{\ell}\Omega^{\ell'} \partial u(t) \|}_{L^2}
  + \sup_{|\ell|+|\ell'|\leq N} {\|\nabla^{\ell}\Omega^{\ell'} S\partial u(t) \|}_{L^2} \leq \e {(1+|t|)}^{C\e},
\end{equation} where $\Omega=(x_i\partial_j-x_j\partial_i)$, $S=t\partial_t+x\cdot\partial_x$
and $\partial=(\partial_t,\nabla)=(\partial_t,\partial_x)$.
Moreover we will use the decay information, see \cite[(6.8)]{Lindblad2},
\begin{equation}
\label{Aliinfty}
\begin{split}
\sup_{|\alpha|+|\beta|\leq N}| \partial_x^{\alpha}\Omega^{\beta} u(t,x)| &\lesssim \e(1+|t|)^{-1+C\e}(1+||x|-t|)^{C\e},
\\
\sup_{|\alpha|+|\beta|\leq N}| \partial_x^{\alpha}\Omega^{\beta} \Delta u(t,x)| &\lesssim\e(1+|t|)^{-1+C\e}(1+||x|-t|)^{-2+C\e}.
\end{split}
\end{equation}

%\comment{Deng: a possibility is the avoid using these decay estimates directly and bootstrap the $L^{\infty}$ norm.
%I think this will make the proof a little harder and may not be worth it... QAQ}

This first Lemma translate the weighted energy bounds on the profile.
\begin{lemma}[Energy bounds]\label{lemEE}
We have
\begin{align}
\label{Nrg}
\sup_{|\ell|+|\ell'|\leq N} {\| \nabla^{\ell}\Omega^{\ell'}f(t) \|}_{L^2} + \sup_{|\ell|+|\ell'|\leq N} {\| \nabla^{\ell}\Omega^{\ell'}Sf(t) \|}_{L^2}
  \lesssim \e {(1+|t|)}^{C\e},
\end{align}
where $f$ is the profile for $(\partial_t - i|\nabla|)u$ defined in \eqref{profile}. Moreover we have the bound
\begin{align}
\label{aprioriE}
\sup_{(j,k)\in\mathcal{J}}\sup_{|\ell|+|\ell'|\leq N-2} 2^{j+k} 2^{|\ell|k^+} \|\Omega^{\ell'}Q_{jk}f(t)\|_{L^2} \leq \e {(1+|t|)}^{C\e}.
\end{align}

\end{lemma}

\begin{proof}
A direct computations shows that $e^{it|\nabla|}$ commutes with $\partial_x$, $\Omega$ and $S$, so \eqref{Ali0}) directly implies \eqref{Nrg}.
Now by \eqref{Nrg} and commutation we have
\[
\sup_{|\ell|+|\ell'|\leq N} {\| S\nabla^{\ell}\Omega^{\ell'}f(t) \|}_{L^2}
  \lesssim \e {(1+|t|)}^{C\e},
\]
and by \eqref{profileeq} and \eqref{Aliinfty} we have
\[
\|(t\partial_t)\partial_x^{\ell}\Omega^{\ell'}f(t)\|_{L^2}\lesssim t\|\nabla^{\ell}\Omega^{\ell'}(u\Delta u)\|_{L^{2}}\lesssim\e^2(1+|t|)^{C\e},
\]
for any $|\ell|+|\ell'|\leq N-2$. This implies
\[
\|(x\cdot\partial_x)\nabla^{\ell}\Omega^{\ell'}f(t)\|_{L^2}\lesssim \e {(1+|t|)}^{C\e},
\]
and combining this with the bounds
\[
\|(x_j\partial_j-x_j\partial_i)\nabla^{\ell}\Omega^{\ell'}f(t)\|_{L^2}\lesssim \e {(1+|t|)}^{C\e},
\]
which is again true for $|\ell|+|\ell'|\leq N-2$, we obtain that
\[
\||x|\cdot\nabla\nabla^{\ell}\Omega^{\ell'}f(t)\|_{L^2}\lesssim \e {(1+|t|)}^{C\e}
\]
which, via standard harmonic analysis estimates, implies \eqref{aprioriE}.
\end{proof}

The next lemma contains some basic bound that are consequences of the energy bounds \eqref{Nrg}.
More refined bounds will be proved later in Section \ref{secImp}.

\begin{lemma}[Basic bounds]\label{LemBB}
Let $|t| \approx 2^m$, $m \in \{0,1,\dots \}$. Under the a priori assumptions \eqref{Ali0}-\eqref{Aliinfty}, we have
\begin{align}
\label{aprioriLinfty}\sup_{|\ell|\leq N/2} \|\widehat{\Omega^{\ell}Q_{jk}f}(t)\|_{L^\infty} \leq \e 2^{-20k^+}\cdot2^{-2k-j/2}\cdot2^{C\e m},
\end{align}
for any $(j,k)\in\mathcal{J}$.
As a consequence we see that
\begin{align}
\label{aprioriLinfty2}
\sup_{|\ell|\leq N/2} \|\widehat{P_k\Omega^{\ell}f}(t)\|_{L^\infty} \leq \e 2^{-20k^+} 2^{-3k/2} 2^{C\e m}.
\end{align}

Moreover
\begin{align}
\label{aprioriLinfty3}
\sup_{|\ell|\leq N/2} \|\widehat{P_k\Omega^{\ell}f}(t)\|_{L^\infty} \leq \e 2^{-20k^+} 2^{(1+C\e)m}.
\end{align}
In particular, combining \eqref{aprioriLinfty2}-\eqref{aprioriLinfty3}, we have
\begin{align}
\label{aprioriLinfty4}
& \sup_{|\ell|\leq N/2} \|\widehat{P_k\Omega^{\ell}f}(t)\|_{L^\infty} \lesssim \e 2^{-20k^+}\min\big(2^{(1+C\e)m}, 2^{-3k/2}2^{C\e m}\big).
\end{align}
%which we will often use to estimate
%\begin{align}
%\label{aprioriLinfty4'}
%& \sup_{|\ell|\leq N/2} \sum_k 2^{15k_+} 2^k \|\widehat{P_k\Omega^{\ell}f}(t)\|_{L^\infty} \lesssim \e 2^{(1/3 + C\e)m}.
%\end{align}

\end{lemma}

\begin{proof}
Since \eqref{aprioriLinfty2} is a direct consequence of \eqref{aprioriLinfty},
and \eqref{aprioriLinfty4} is direct consequences of \eqref{aprioriLinfty3}, we only need to prove \eqref{aprioriLinfty} and \eqref{aprioriLinfty3}.

\noindent
{\it Proof of \eqref{aprioriLinfty}.}
Let $\widehat{\Omega^{\ell}Q_{jk}f}(t,\xi)=F(\xi)$, then \eqref{aprioriE} implies
\[\|\Omega^{N/4} \partial_{\xi}F\|_{L^2} \lesssim 2^{-k-20k^+}2^{C\e m};
  \quad \|\Omega^{N/4}F\|_{L^2}\lesssim 2^{-j-k-20k^+}2^{C\varepsilon m}.
\]
Writing $\xi=\rho\theta$ in polar coordinates, and using Sobolev embedding in $\theta$ and Gagliardo-Nirenberg in $\rho$,
and using the support information of $F$, we see that
\[
\|F\|_{L^{\infty}} = \|F(\rho\theta)\|_{L_{\rho,\theta}^{\infty}}
  \lesssim \| \partial_{\theta}^4 F(\rho\theta)\|_{L_{\rho,\theta}^2}^{1/2}\| \partial_{\theta}^4\partial_{\rho} F(\rho\theta)\|_{L_{\rho,\theta}^2}^{1/2}
  \lesssim \e 2^{-k} 2^{-j/2-k} 2^{-20k^+} 2^{C\e m}.
\]

\noindent
{\it Proof of \eqref{aprioriLinfty3}.}
This bound is not a direct consequence of the a priori bounds \eqref{Ali0} when frequencies are very small.
We therefore prove it by a bootstrap argument.
More precisely, we assume the bounds \eqref{Ali0}-\eqref{Aliinfty}, which in particular imply \eqref{aprioriLinfty2}, and assume a priori
\begin{align}
\label{apLinfty3.1}
\sup_{|\ell|\leq N/2} 2^{20k^+} \|\widehat{P_k\Omega^{\ell}f}(t)\|_{L^\infty} \lesssim \e 2^{(1 +C'\e)m},
\end{align}
where $C'\gg C$ but still is an absolute constant, and show that these imply
\begin{align}
\label{apLinfty3.2}
\sup_{|\ell|\leq N/2} 2^{20k^+} \|\widehat{P_k\Omega^{\ell}f}(t)\|_{L^\infty} \lesssim \e + \e^2 2^{(1 + C'\e/2)m}.
\end{align}

Since $f(0)=v(0)=\ve(\psi^0-i|\nabla|\varphi^0)$, in view of the decomposition \eqref{duhamel}-\eqref{duhamel3}, and the identity
\begin{multline}
\label{commomg}
\Omega^{\ell}\int_{\mathbb{R}^3}e^{is\Phi_{\kappa_1\kappa_2}(\xi,\eta)}\frac{|\xi-\eta|}{|\eta|}\widehat{f}(s,\xi-\eta)\widehat{g}(s,\eta)\,\mathrm{d}\eta
\\
=\sum_{\ell_1+\ell_2=\ell}\int_{\mathbb{R}^3}e^{is\Phi_{\kappa_1\kappa_2}(\xi,\eta)}\frac{|\xi-\eta|}{|\eta|}
  \widehat{\Omega^{\ell_1}f}(s,\xi-\eta)\widehat{\Omega^{\ell_2}g}(s,\eta)\,\mathrm{d}\eta,
\end{multline}
we only need to estimate, for fixed $\kappa_{1,2}\in\{\pm\}$ and fixed $\ell_1,\ell_2$, the integral
\begin{equation*}
\bigg|\int_0^t\tau_m(s)\int_{\mathbb{R}^3}e^{is\Phi_{\kappa_1\kappa_2}(\xi,\eta)}\frac{|\xi-\eta|}{|\eta|}
  \widehat{\Omega^{\ell_1}f_{\kappa_1}}(s,\xi-\eta)\widehat{\Omega^{\ell_2}f_{\kappa_2}}(s,\eta)\,\mathrm{d}\eta\mathrm{d}s\bigg|\lesssim\ve^22^{(1+C'\e/2)m}.
\end{equation*}
Let $\Omega^{\ell_j}f_{\kappa_j}=g_j$, then for $|\ell|\leq N/2$ and each $k$ we have
\begin{equation}
\label{apLinftym'}
{\|P_kg_j\|}_{L^2} \lesssim \e 2^{-30k^+}2^{C\varepsilon m},\quad \|\varphi_k(\xi)\cdot \widehat{g_j}\|_{L^\infty}
  \lesssim \e 2^{-20k^+}\min\big(2^{(1+C'\e)m},2^{-3k/2}2^{C\e m}\big),
\end{equation}
and by fixing time $s \approx 2^m$ it suffices to bound
\begin{align}
\label{apLinftym}
\begin{split}
|I|[g_1,g_2](\xi) := & \int_{\R^3} \frac{|\xi-\eta|}{|\eta|}
  |\widehat{g_1}(\xi-\eta)| \, |\widehat{g_2}(\eta)| \,\mathrm{d}\eta\lesssim \e^22^{C'\e m/2}.
\end{split}
\end{align}

Now for $|\xi| \approx 2^k$ we have
%\begin{align}\label{aprioriLinfty3.3}
%\begin{split}
%\sup_{|\ell|\leq N_1-10} \Big| \varphi_k(\xi)\Omega^{\ell} \int_{\R^3} e^{it\Phi_{\kappa_1\kappa_2}(\xi,\eta)}\frac{|\xi-\eta|}{|\eta|}
%  \widehat{f_{\kappa_1}}(t,\xi-\eta)\widehat{f_{\kappa_2}}(t,\eta)\,\mathrm{d}\eta \Big| \lesssim \e^2 2^{\delta_1 m},
%\\
%\Phi_{\kappa_1\kappa_2}(\xi,\eta) := |\xi|-\kappa_1|\xi-\eta|-\kappa_2|\eta|, \qquad \kappa_1,\kappa_2 = \pm.
%\end{split}
%\end{align}
\begin{align*}
2^{20k^+}\sum_{k_1\leq k_2+10}|I|[ P_{k_1}g_1, P_{k_2}g_2] \lesssim 
  \sum_{k_1\leq k_2+10} 2^{20k_1^+}{\|P_{k_1}g_1\|}_{L^2} 2^{20k_2^+}{\|P_{k_2}g_2\|}_{L^2} 2^{k_1-k_2} \lesssim\e^2 2^{C\e m},
\end{align*}
using \eqref{apLinftym'} and also the consequence that ${\|P_{k_2}g_2\|}_{L^2}  \lesssim \e2^{3k/2}2^{(1+C'\e)m}$.
For the remaining terms we have
\begin{align*}
2^{20k^+}\sum_{k_1\geq k_2+10}|I|[P_{k_1}g_1, P_{k_2}g_2]
  & \lesssim \sum_{k_1\geq k_2+10} 2^{k_1} 2^{20k_1^+}{\|\varphi_{k_1}\widehat{g_1}\|}_{L^\infty} 2^{20k_2^+}{\| P_{k_2}g_2\|}_{L^2} 2^{k_2/2}
  \\ & \lesssim \ve 2^{C\e m} \cdot\sum_{k_1} 2^{3k_1/2}2^{-10k_1^+}{\|\varphi_{k_1}\widehat{g_1}\|}_{L^\infty} \lesssim \e^22^{C\e m},
\end{align*}
using \eqref{aprioriLinfty2}. This proves \eqref{aprioriLinfty3}.
\end{proof}

\medskip
\subsection{Angular integration by parts}
Here is a Lemma about angular integration by parts which we will use several times in the rest of the proof.

\begin{lemma}\label{LemAIBP}
Consider the
integral\begin{equation}
\label{LemAIBP1}
J=\int_{\mathbb{R}^3}e^{is\Phi(\xi,\eta)}\chi_{k}(\xi)\chi_{k_1}(\xi-\eta)\chi_{k_2}(\eta)\varphi_{p+k+\min(k_1,k_2)}(\xi\wedge\eta)
  \widehat{f}(\xi-\eta)\widehat{g}(\eta)\,\mathrm{d}\eta,
\end{equation}
where $p\leq 0$, $\chi$ is a compactly supported and smooth cutoff,
\begin{equation*}
|s|\approx 2^m,\quad \Phi(\xi,\eta)=|\xi|\pm|\xi-\eta|\pm|\eta|,
\end{equation*}
and, for all $|\alpha|\leq M$, the functions $f$ and $g$ satisfy
\begin{equation}
\label{LemAIBP2}
\|\nabla^\alpha\widehat{f}\|_{L^2}\lesssim 2^{j_1|\alpha|}\|\widehat{f}\|_{L^2},
\quad \|\nabla^{\alpha}\widehat{g}\|_{L^2}\lesssim 2^{j_2|\alpha|}\|\widehat{g}\|_{L^2};
\quad j_r\geq\max(-k_r,0),\,r\in\{1,2\},
\end{equation}
as well as
\begin{equation}
\label{LemAIBP3}
\|\Omega^{\alpha}\widehat{f}\|_{L^2}\lesssim
2^{\kappa|\alpha|}\|\widehat{f}\|_{L^2},\quad\|\Omega^{\alpha}\widehat{g}\|_{L^2}\lesssim
2^{\kappa|\alpha|}\|\widehat{g}\|_{L^2};\quad \kappa\geq 0,\end{equation}
where $\Omega$ denotes the rotation vector fields. Then, we
have
\begin{equation}
\label{LemAIBPconc}
|J|\lesssim 2^{-\nu Mm}\|\widehat{f}\|_{L^2}\|\widehat{g}\|_{L^2},
\end{equation}
provided that
\begin{equation}
\label{assump}
\min(j_1,j_2)\leq (1-\nu) m,\quad p\geq\frac{-m-\min(k,k_1,k_2)+\kappa+\nu m}{2}
\end{equation} for some constant $\nu>0$.
\end{lemma}

\begin{proof}
We may assume that $\widehat{f}(\zeta)$ is supported in the region $|\zeta|\approx 2^{k_1}$, and similarly for $\widehat{g}$; by symmetry
and (\ref{assump}) we may also assume $k_1\geq k_2$ and
$\min(k,k_1,k_2)\geq -m$. Writing
\begin{equation}\varphi_{p+k+k_2}(\xi\wedge\eta)=\varphi_{p+k+k_2}(\xi\wedge\eta)\cdot\sum_{j=1}^3\psi_{p+k_1+k_2}((\xi\wedge\eta)_j)\end{equation}
for some suitable cutoff $\chi$, we may assume that
$|\xi_1\eta_2-\xi_2\eta_1|\approx 2^{p+k+k_2}$ in the support of integral.
There are then three cases to consider.

\medskip
{\it Case 1: $|k-k_1|\leq 4$, and $j_1\leq (1-\nu)m$}.
First express $\widehat{f}$ in polar coordinates and use the Fourier transform in the
radial coordinate to write
\begin{equation}\widehat{f}(\xi-\eta)=\int_{\mathbb{R}}e^{i\rho|\xi-\eta|}H\bigg(\rho,\frac{\xi-\eta}{|\xi-\eta|}\bigg)\,\mathrm{d}\rho,\end{equation}
where $H(\rho,\zeta)$ is homogeneous of degree $0$ in $\zeta$, then we have
that (due to our assumption about the support of $\widehat{f}$, and
$|k-k_1|\leq
4$)\begin{equation}\|2^k\rho^MH(\rho,\theta)\|_{L_{\rho,\theta}^2}\lesssim\bigg\|\bigg(\frac{\zeta}{|\zeta|}\cdot\partial_{\zeta}\bigg)^M\widehat{f}(\zeta)\bigg\|_{L^2}\lesssim
2^{j_1M}\|\widehat{f}\|_{L^2},\end{equation} so if we define
\begin{equation}H'(\rho,\theta)=H(\rho,\theta)\varphi_{\leq
m-10}(\rho),\quad
H''(\rho,\theta)=H(\rho,\theta)\varphi_{>m-10}(\rho),\end{equation} and
define $(f',f'')$ and $(J',J'')$ accordingly, then we have
\begin{equation}\|\widehat{f''}\|_{L^2}\approx\|2^kH''(\rho,\theta)\|_{L_{\rho,\theta}^2}\lesssim
2^{-\nu Mm}\|\widehat{f}\|_{L^2}\end{equation} since $j_1\leq
(1-\nu)m$. By Cauchy-Schwartz we then have
\begin{equation}|J''|\lesssim\|\widehat{f''}\|_{L^2}\|\widehat{g}\|_{L^2}\lesssim
2^{-\nu Mm}\|\widehat{f}\|_{L^2}\|\widehat{g}\|_{L^2},\end{equation}
which is acceptable. Below we will consider
$f'$ and $J'$ only, and will omit the ``prime'' symbol for simplicity. Since now
\begin{equation}
\widehat{f}(\xi-\eta)=\int_{|\rho|\lesssim 2^m}\varphi_{\leq m-10}(\rho)
  e^{i\rho|\xi-\eta|}H\bigg(\rho,\frac{\xi-\eta}{|\xi-\eta|}\bigg)\,\mathrm{d}\rho,
\end{equation}
we will first fix one $\rho$ and denote the corresponding contribution by
$J_{\rho}$. Note that $|\rho|\ll s$ by our assumption, so $|\pm s+\rho|\approx 2^m$.

Let $D=\eta_1\partial_{\eta_2}-\eta_2\partial_{\eta_1}$ and
\begin{equation}L:f\mapsto\frac{Df}{D|\xi-\eta|};\quad
L':f\mapsto-D\bigg(\frac{f}{D|\xi-\eta|}\bigg),\end{equation} then we have
$L(e^{is\Phi+i\rho|\xi-\eta|})=i(\pm s+\rho)e^{is\Phi+i\rho|\xi-\eta|}$,
and thus\begin{equation}J_{\rho}=(i(\pm
s+\rho))^{-M}\int_{\mathbb{R}^3}e^{is\Phi+i\rho|\xi-\eta|}(L')^MG(\eta)\,\mathrm{d}\eta,\end{equation}
where
\begin{equation}
G(\eta) := \chi_{k}(\xi)\chi_{k_1}(\xi-\eta)\chi_{k_2}(\eta)\varphi_{p+k+k_2}(\xi\wedge\eta)
  \psi_{p+k_1+k_2}((\xi\wedge\eta)_3)H\bigg(\rho,\frac{\xi-\eta}{|\xi-\eta|}\bigg)\widehat{g}(\eta).
\end{equation}
By induction in $M$ one can prove that
\begin{equation}
(L')^MG=\sum_{r=0}^M\sum_{\alpha_0+\cdots+\alpha_r=M-r}D^{\alpha_0}G\cdot\frac{D^{\alpha_1+2}|\xi-\eta|\cdots
  D^{\alpha_r+2}|\xi-\eta|}{(D|\xi-\eta|)^{r+M}},
\end{equation}
where the coefficients are omitted for simplicity (the same below);
now we analyze each factor appearing in this expression.

First, we have
\begin{equation}D|\xi-\eta|=\frac{\xi_1\eta_2-\xi_2\eta_1}{|\xi-\eta|},\quad
|D|\xi-\eta||\approx 2^{k_2+p};\end{equation} by Fa\`{a} di Bruno's formula one has that
\begin{multline}
D^q|\xi-\eta|=\sum_{r=1}^q\sum_{\alpha_1+\cdots+\alpha_r=q-r}|\xi-\eta|^{1-2r}\prod_{j=1}^r
D^{\alpha_j+1}|\xi-\eta|^2\\=\sum_{r=1}^q\sum_{\alpha_1+\cdots+\alpha_r=q-r}|\xi-\eta|^{1-2r}\prod_{\alpha_j\,\textrm{even}}(\xi_1\eta_2-\xi_2\eta_1)\prod_{\alpha_j\,\textrm{odd}}(\xi_1\eta_1+\xi_2\eta_2),
\end{multline}again with coefficients omitted, therefore
\begin{equation}\|D^q|\xi-\eta|\|_{L^\infty}\lesssim\sup_{r\geq 1}2^{(1-2r)k}2^{r(k+k_2)}\lesssim 2^{k_2}
\end{equation}
with constants depending on $q$. As for the $D^{\alpha_0}G$ factor, we have
\begin{equation}
\label{est1}
\|D^q(\chi_k(\xi)\chi_{k_1}(\xi-\eta)\chi_{k_2}(\eta))\|_{L^\infty}\lesssim 1,
  \qquad \|D^q\widehat{g}(\eta)\|_{L^2}\lesssim 2^{\kappa q}\|\widehat{g}\|_{L^2},
\end{equation}
and
\begin{equation}
\label{est2}
D^q(\varphi_{p+k+k_2}(\xi\wedge\eta))=\sum_{r=1}^q\sum_{\alpha_1+\cdots+\alpha_r=q-r}(\partial^r\varphi_{p+k+k_2})(\xi\wedge\eta)\prod_{j=1}^rD^{\alpha_j+1}(\xi\wedge\eta),
\end{equation}
which implies that
\begin{equation}
\|D^q(\varphi_{p+k+k_2}(\xi\wedge\eta))\|_{L^\infty}\lesssim
\sup_{1\leq r\leq q}2^{(-p-k-k_2)r}2^{r(k+k_2)}\lesssim 2^{-qp},
\end{equation}
and similarly
\begin{equation}
\|D^q(\psi_{p+k+k_2}((\xi\wedge\eta)_3))\|_{L^\infty}\lesssim 2^{-qp}.
\end{equation}
Using also that
\begin{equation}
D^qH\bigg(\rho,\frac{\xi-\eta}{|\xi-\eta|}\bigg)=\sum_{|\beta|,|\gamma|\leq |\alpha|\leq q}
\eta^{\alpha}(\xi-\eta)^\beta|\xi-\eta|^{-|\alpha|-|\beta|}(\nabla_{\theta}^{\gamma}H)\bigg(\rho,\frac{\xi-\eta}{|\xi-\eta|}\bigg),\end{equation}
we get
\begin{equation}
\label{est6}
\|D^q G\|_{L^1}\lesssim\sup_{|\gamma|+l\leq q}2^{l(\kappa-p)}
  \|2^k\nabla_{\theta}^\gamma H(\rho,\theta)\|_{L_{\theta}^2} \|\widehat{g}\|_{L^2},
\end{equation}
and therefore
\begin{equation}
\|(L')^MG\|_{L^1}\lesssim\sup_{|\gamma|+l+r\leq M} 2^{l(\kappa-p)-(r+M)(k_2+p)+rk_2}
  \|2^k\nabla_{\theta}^\gamma H(\rho,\theta)\|_{L_{\theta}^2} \|\widehat{g}\|_{L^2},
\end{equation}
and
\begin{equation}
|J_{\rho}|\lesssim\sup_{|\gamma|+l+r\leq M}2^{-M m + l(\kappa-p)-(r+M)(k_2+p)+rk_2}
  \|2^k\nabla_{\theta}^\gamma H(\rho,\theta)\|_{L_{\theta}^2} \|\widehat{g}\|_{L^2}.
\end{equation}
Integrating in $\rho$, we get that
\begin{align*}
|J| & \lesssim\int_{|\rho|\lesssim 2^m}\varphi_{\leq m-10}(\rho)|J_{\rho}|
\\ & \lesssim 2^{m/2}\|\widehat{g}\|_{L^2} \sup_{|\gamma|+l+r\leq M}2^{-Mm+l(\kappa-p)-(r+M)(k_2+p)+rk_2}\|2^k\nabla_{\theta}^\gamma
  H(\rho,\theta)\|_{L_{\rho,\theta}^2}
\\
& \lesssim 2^{m/2}\|\widehat{f}\|_{L^2} \|\widehat{g}\|_{L^2}
  \sup_{|\gamma|+l+r\leq M} 2^{-Mm+l(\kappa-p)-(r+M)(k_2+p)+rk_2+\kappa|\gamma|},
\end{align*}
using
the fact that
\begin{equation}
\|2^k\nabla_{\theta}^\gamma H(\rho,\theta)\|_{L_{\rho,\theta}^2} \approx \|\Omega^\gamma\widehat{f}\|_{L^2} \lesssim
  2^{\kappa|\gamma|} \|\widehat{f}\|_{L^2}.
\end{equation}
Optimizing the last line, we get that
\begin{equation}
|J|\lesssim 2^{m/2} \|\widehat{f}\|_{L^2} \|\widehat{g}\|_{L^2}\cdot 2^{M(\kappa-k_2-2p-m)}\lesssim
  2^{-\nu Mm}\|\widehat{f}\|_{L^2} \|\widehat{g}\|_{L^2}
\end{equation}
due to our choice \eqref{assump}.

\medskip
{\it Case 2: $|k_1-k_2|\leq 4$, and $j_1\leq (1-\nu)m$}.
We then apply the same argument as in Case 1, except that the bounds are
now\begin{equation}|D|\xi-\eta||\approx
2^{k+p},\quad\|D^{q}|\xi-\eta|\|_{L^\infty}\lesssim 2^k,
\end{equation}
and
\eqref{est1}$\approx$\eqref{est6} still hold, except that the factor $2^k$ in \eqref{est6}
is replaced by $2^{k_1}$. Following the same lines, we get
\begin{equation}
|J|\lesssim2^{m/2}\|\widehat{f}\|_{L^2} \|\widehat{g}\|_{L^2}
  \sup_{|\gamma|+l+r\leq M}2^{-Mm+l(\kappa-p)-(r+M)(k+p)+rk+\kappa|\gamma|},
\end{equation}
and consequently
\begin{equation}
|J|\lesssim 2^{m/2} \|\widehat{f}\|_{L^2} \|\widehat{g}\|_{L^2}\cdot 2^{M(\kappa-k-2p-m)}
  \lesssim 2^{-\nu Mm}\|\widehat{f}\|_{L^2}|\widehat{g}\|_{L^2}
\end{equation}
due to our choice \eqref{assump}.

\medskip
{\it Case 3: $|k-k_1|\leq 4$, and $j_2 \leq (1-\nu)m$}.
In this case we switch the role of $\eta$ and $\xi-\eta$;
first we manipulate $\widehat{g}(\eta)$ in the same way as Case 1 above, and reduce to studying $H(\rho,\eta/|\eta|)$.
Define $D=(\xi-\eta)_1\partial_{\eta_2}-(\xi-\eta)_2\partial_{\eta_1}$ and define
$L$ and other quantities accordingly, then we have, replacing the estimates in Case 1, that
\begin{equation}
D|\eta|=\frac{\xi_1\eta_2-\xi_2\eta_1}{|\eta|},\quad |D|\eta||\approx 2^{p+k},
\end{equation}
and
\begin{multline}
D^q|\eta|=\sum_{r=1}^q\sum_{\alpha_1+\cdots+\alpha_r=q-r}|\eta|^{1-2r}\prod_{j=1}^r
  D^{\alpha_j+1}|\eta|^2\\=\sum_{r=1}^q\sum_{\alpha_1+\cdots+\alpha_r=q-r}|\eta|^{1-2r}\prod_{\alpha_j\,\textrm{even}}(\xi_1\eta_2-\xi_2\eta_1)\prod_{\alpha_j\,\textrm{odd}}(\xi_1(\xi-\eta)_1+\xi_2(\xi-\eta)_2),
\end{multline}
which implies
\begin{equation}
\|D^q|\eta|\|_{L^\infty}\lesssim 2^{k_2+q(k-k_2)};
\end{equation}
moreover, we have
\begin{equation}
\|D^q(\chi_k(\xi)\chi_{k_1}(\xi-\eta)\chi_{k_2}(\eta))\|_{L^\infty}
  \lesssim 2^{q(k-k_2)},\quad \|D^q\widehat{f}(\xi-\eta)\|_{L^2}\lesssim 2^{\kappa q}\|\widehat{g}\|_{L^2},
\end{equation}
and \eqref{est2} now implies that
\begin{equation}
\|D^q(\varphi_{p+k+k_2}(\xi\wedge\eta))\|_{L^\infty}\lesssim \sup_{1\leq r\leq q}2^{(-p-k-k_2)r}2^{2kr}\lesssim 2^{(-p+k-k_2)q},
\end{equation}
and similarly
\begin{equation}
\|D^q(\psi_{p+k+k_2}((\xi\wedge\eta)_3))\|_{L^\infty}\lesssim 2^{(-p+k-k_2)q}.
\end{equation}
Using also that
\begin{equation}
D^q H\bigg(\rho,\frac{\eta}{|\eta|}\bigg)=\sum_{|\beta|,|\gamma|\leq |\alpha|\leq q}
  (\xi-\eta)^{\alpha}\eta^\beta|\eta|^{-|\alpha|-|\beta|}(\nabla_{\theta}^{\gamma}H)\bigg(\rho,\frac{\eta}{|\eta|}\bigg),
\end{equation}
we get
\begin{equation}
\|D^qG\|_{L^1}\lesssim\sup_{|\gamma|+l\leq q} 2^{l(\kappa-p+k-k_2)}\|2^k\nabla_{\theta}^\gamma
  H(\rho,\theta)\|_{L_{\theta}^2}\|\widehat{g}\|_{L^2},
\end{equation}
so following the same line as Case 1 we get that
\begin{equation}
|J| \lesssim 2^{m/2}\|\widehat{f}\|_{L^2} \|\widehat{g}\|_{L^2}\cdot 2^{M(\kappa-k_2-2p-m)}
  \lesssim 2^{-\nu Mm} \|\widehat{f}\|_{L^2} \|\widehat{g}\|_{L^2}
\end{equation}
due to our choice \eqref{assump}.
\end{proof}

Here is an additional simple result for non-stationary  integrals:
\begin{lemma}\label{lemIBP0}
Assume that $\epsilon \in(0,1)$, $\epsilon K\geq 1$, $M \geq 1$ is an integer, and $F,g \in C^M (\R^n)$.
Assume also that $F$ is real-valued and satisfies
\begin{equation*}%\label{IBP01}
|\nabla F| \geq \mathbf{1}_{\supp(g)}, \qquad \big| D^{\alpha} F \big| \lesssim_M \epsilon^{1-|\alpha|} \quad \forall \,\, 2 \leq |\alpha| \leq M.
\end{equation*}
Then
\begin{align}
\label{IBP02}
\Big| \int_{\R^n} e^{iKF} g \, dx \Big| \lesssim \frac{1}{(\epsilon K)^M} \sum_{|\alpha|\leq M} \epsilon^{|\alpha|} {\| D^\alpha g \|}_{L^1}
\end{align}
\end{lemma}

\medskip
\subsection{Improved bounds on $\widehat{f}$}\label{secImp}
We now show how starting from the basic bounds of Lemma \ref{LemBB}, it is possible to get improved estimates by using,
among other things, the angular integration by parts estimate of Lemma \ref{LemAIBP}.

\begin{lemma}\label{lemImp}
Let $f$ be the profile \eqref{duhamel}-\eqref{duhamel2}, and assume the apriori bounds \eqref{aprioriE}.
Then, for all $t \approx 2^m$, $m\geq 0$, $k\in\mathbb{Z}$ we have
\begin{align}
\label{fhatimp}
\sup_{|\ell|\leq N/3+2} \|\widehat{P_k\Omega^{\ell}f}(t)\|_{L^\infty} \leq \e 2^{-20k^+} \max\big( 2^{\delta m}, 2^{-k}\big) 2^{(C\e+12\delta)m}.
\end{align}
\end{lemma}

\medskip
Note that, combining Lemmas \ref{lemImp} and \ref{LemBB} yields a useful bound: for all $t \approx 2^m$ we have
\begin{equation}
\label{fhatimpcor2}
\sup_{|\ell|\leq N/3+2} \sum_k 2^{15k_+} 2^k \|\widehat{P_k\Omega^{\ell} f}(t)\|_{L^\infty} \lesssim \e 2^{\delta'm},\quad \delta':=C\varepsilon+13\delta.
\end{equation}
This follows by using \eqref{aprioriLinfty3} for the sum over $k\leq -2m$,
and \eqref{fhatimp} and \eqref{aprioriLinfty2} when summing over $k \geq -2m$.

\medskip
\begin{proof}
The bound \eqref{fhatimp} is the consequence of iterating many times the following claim, starting at $\beta=1$ and $M=0$ (see Lemma \ref{LemBB}):

\medskip
\underline{\it{Claim.}}
Suppose for some $\beta,M>0$ we have
\begin{equation}\label{fhatimp01}
\sup_{|\ell|\leq N/3+2} \|\widehat{P_k\Omega^{\ell}f}(t)\|_{L^\infty} \lesssim \e 2^{-20k^+} \max\big( 2^{\beta m}, 2^{-k}\big) 2^{(M\delta+C\e)m},
\end{equation}
for all $t\approx 2^m$, then for the same $t$ we also have, with $\beta'=2\beta/3$ and $M'=2M/3+4$, that
\begin{equation}\label{fhatimp02}
\sup_{|\ell|\leq N/3+2} \|\widehat{P_k\Omega^{\ell}f}(t)\|_{L^\infty} \lesssim \e 2^{-20k^+} \max\big( 2^{\beta'm}, 2^{-k}\big) 2^{(M'\delta+C\e)m}.
\end{equation}

To prove the claim, assume that (\ref{fhatimp01}) holds. Using the Duhamel formula (\ref{duhamel})
and the commutation formula (\ref{commomg}), in the same way as in the proof of Lemma \ref{LemBB},
we see that (\ref{fhatimp02}) would follow if we can prove the estimate
\begin{align}
\label{fhatimp1}
\begin{split}
\big|I_k[g_1,g_2](\xi)\big| := \bigg|\varphi_k(\xi)\int_{\R^3} e^{is(|\xi|\pm|\xi-\eta|\pm|\eta|)}
  \frac{|\xi-\eta|}{|\eta|} \widehat{g_1}(\xi-\eta) \widehat{g_2}(\eta) \,\mathrm{d}\eta\bigg|
\\
\lesssim\e^2 2^{-20k^+}2^{(-1+M'\delta+C\e)m} \max\big( 2^{\beta'm}, 2^{-k}\big),
\end{split}
\end{align}
 for all $s\approx 2^m$ and all functions $g_r$ (which are $g_r=\Omega^{\ell_r}f_{\kappa_r}$, see the proof of Lemma \ref{LemBB}) satisfying
\begin{align}\label{fhatimp1'}
\sup_{|\ell|\leq N/2}{\| \Omega^{\ell}P_kg_r \|}_{L^2} &\lesssim \e2^{-20k^+}2^{C\e m}, 
&
{\| \what{Q_{j,k} g_r} \|}_{L^\infty} &\lesssim \e2^{-20k^+} 2^{-j/2-2k}2^{C\e m},
\\
\label{fhatimp1''}
& & \|\what{P_kg_r}\|_{L^\infty}&\lesssim\e 2^{-20k^+}2^{(M\delta+C\e)m} \max\big( 2^{\beta m}, 2^{-k}\big)
\end{align} for $r=1,2$. We decompose the inputs $g_{r}$ into $Q_{j_rk_r}g_r$ for $r=1,2$,
 and divide the proof into two cases.

(1) Assume that $\min(j_1,j_2) \geq (1-4\delta)m$ or $\max(j_1,j_2) \geq 2m$. To treat the case when both inputs have large spatial localization it suffices to use the second inequality in \eqref{fhatimp1'} and estimate
\begin{align}
 \label{fhatimp9}
\begin{split}
\big| I_k[Q_{j_1k_1}g_1, Q_{j_2k_2}g_2](\xi) \big| \lesssim 2^{k_1-k_2} \int_{\R^3}
  \big| \widehat{Q_{j_1k_1}g_1}(\xi-\eta) \big| \, \big| \widehat{Q_{j_2k_2}g_2}(\eta) \big| \,\mathrm{d}\eta
  \\
\lesssim\e^2 2^{C\e m}2^{k_1-k_2} \cdot 2^{-20k_1^+}2^{-j_1/2-2k_1} \cdot 2^{-20k_2^+}2^{-j_2/2-2k_2} \cdot 2^{3\min(k_1,k_2)}
\\
\lesssim\e^2 2^{-20k_1^+} 2^{-20k_2^+} 2^{-\max(k_1,k_2)} \cdot 2^{-(j_1+j_2)/2} 2^{C\e m}.
\end{split}
\end{align}
Summing over all indices $j_1,j_2,k_1,k_2$ with $\min(j_1,j_2) \geq (1-4\delta)m$, or $\max(j_1,j_2) \geq 2m$, we obtain
\begin{align}
\label{fhatimp10}
\begin{split}
\sum_{\substack{\min(j_1,j_2) \geq (1-4\delta)m,\\\textrm{or }\max(j_1,j_2)\geq 2m}} \big| I_k[Q_{j_1,k_2}g_1, Q_{j_2,k_2}g_2](\xi) \big|
  \lesssim\e^2 2^{-k-20k^+} 2^{-m} 2^{(C\e+4\delta) m},
\end{split}
\end{align}
which is stronger than (\ref{fhatimp1}).

(2) Assume $\min(j_1,j_2)\leq (1-4\delta)m$ and $\max(j_1,j_2)\leq 2m$. We can then be able to apply Lemma \ref{LemAIBP}
in combination with the first inequality in \eqref{fhatimp1'}, see the first condition in \eqref{assump}, to restrict the integral to the region where
\[|\xi\wedge\eta|\lesssim 2^{p_0+k+\min(k_1,k_2)},\quad p_0=-\frac{m+\min(k,k_1,k_2)}{2}+2\delta m.
\]
This then restricts the integral to a cone of angular aperture $2^{p_0}$ and radius $2^{\min(k_1,k_2)}$; arguing as in \eqref{fhatimp9},
we obtain that
\begin{equation}\label{newest}
\big| I_k[Q_{j_1k_1}g_1, Q_{j_2k_2}g_2](\xi) \big| \lesssim 2^{k_1-k_2}\cdot 2^{2p_0+3\min(k_1,k_2)}\cdot\|\what{Q_{j_1k_1}g_1}\|_{L^\infty}\cdot\|\what{Q_{j_2k_2}g_2}\|_{L^\infty}.
\end{equation} If $k_2\geq k-10$, then we have \[k_1-k_2+2p_0+3\min(k_1,k_2)\leq -m-k+3(k_1+k_2)/2+4\delta m,\]
which implies, using the second inequality in \eqref{fhatimp1'}, that
\[
\big| I_k[Q_{j_1k_1}g_1, Q_{j_2k_2}g_2](\xi) \big| \lesssim \e^22^{-20k_1^+-20k_2^+}2^{-m-k}2^{(4\delta+C\e)m},
\]
which implies (\ref{fhatimp1}) upon summing over $(j_r,k_r)$.
If $k_2\leq k-10$, then $|k-k_1|\leq 5$ and \eqref{newest} gives
\begin{equation}
\label{newest2}
\big| I_k[Q_{j_1k_1}g_1, Q_{j_2k_2}g_2](\xi) \big| \lesssim
  \e^2 2^{(-1+4\delta)m} \cdot 2^{k_1+k_2}\cdot\|\what{Q_{j_1k_1}g_1}\|_{L^\infty}\cdot\|\what{Q_{j_2k_2}g_2}\|_{L^\infty}.
\end{equation}
Now if $k_2\leq -\beta m$, then in (\ref{newest2}) we may use the second inequality in (\ref{fhatimp1'})
to bound the first $L^\infty$ factor and (\ref{fhatimp1''}) to bound the second, obtaining
\[\big| I_k[Q_{j_1k_1}g_1, Q_{j_2k_2}g_2](\xi) \big| \lesssim \e^22^{-20k^+} 2^{(-1+4\delta+C\e)m}\cdot2^{-k/2}2^{(M\delta+C\e)m}.
\] This does not imply (\ref{fhatimp1}), but upon integration in time, this implies that
\[\sup_{|\ell|\leq N/3+2} \|\widehat{P_k\Omega^{\ell}f}(t)\|_{L^\infty} \lesssim \e 2^{-20k^+}2^{-k/2} 2^{((M+4)\delta+C\e)m},
\] which then implies (\ref{fhatimp02}) via interpolation with (\ref{aprioriLinfty2}) (we obtain $M'=(M+4)/2$).

 Finally, if $k\geq k_2\geq -\beta m$, then the second inequality in (\ref{fhatimp1'}) and (\ref{fhatimp1''}) implies, 
for $k_*\in\{k,k_2\}$, that
\[
2^{k_*}\|\what{Q_{j_*k_*}g_r}\|_{L^\infty}\lesssim \e2^{-20k_*^+}\min\big(2^{C\e m-k_*/2},2^{(M\delta+C\e)m}2^{\beta m+k_*}\big)
  \lesssim\e2^{-20k_*^+}2^{\beta m/3}2^{(M\delta/3+C\e)m}.
\] 
This then gives
\[ \big| I_k[Q_{j_1k_1}g_1, Q_{j_2k_2}g_2](\xi) \big| \lesssim \e^22^{-20k_1^+-20k_2^+}\cdot 2^{(-1+2\beta/3)m}\cdot 2^{(4\delta+2M\delta/3+C\e)m},
\] 
which again implies (\ref{fhatimp1}) upon summation. This completes the proof.
\end{proof}

\medskip

\medskip
Another direct consequence of Lemma \ref{lemImp} is the following bound for time derivative of the profile:

\begin{lemma}\label{Lemdtf}
Under the assumptions \eqref{aprioriE}, for all $t \approx 2^m$, $m\geq 0$, we have the following estimates:
\begin{align}
\label{dtf1}
& \sup_{|\ell|\leq N/3+2} \|P_k\Omega^{\ell}\partial_t f(t)\|_{L^2} \lesssim \e^2  2^{-20k^+} 2^{-(1-C\e)m},
\\
\label{dtf2}
& \sup_{|\ell|\leq N/3+2} \|\widehat{P_k\Omega^{\ell}\partial_t f}(t)\|_{L^\infty} \lesssim \e^2 2^{-m} 2^{-20k^+} 2^{-k+(\delta'+4\delta)m}.
\end{align}

\end{lemma}

\begin{proof}
To prove \eqref{dtf1} is suffices to use \eqref{profileeq}, H\"older's inequality, and the $L^\infty$ decay \eqref{Aliinfty}.
To prove \eqref{dtf2} it suffices to combine the commutation formula \eqref{commomg}, the estimate \eqref{newest2} and \eqref{fhatimp}, 
and follow the proof of Lemma \ref{lemImp}.
\end{proof}

Finally, we have the following improvement of (\ref{fhatimp}) for $Q_{jk}f$:

\begin{lemma}\label{Lemspe}
For all $t \approx 2^m$, $m\geq 0$ we have
\begin{align}
\label{fhatimp'}
\begin{split}
& \sup_{|\ell|\leq N/3+1} \big\|\what{Q_{jk}\Omega^{\ell}f}(t) \big\|_{L^\infty} \leq \e 2^{-k} 2^{-m/16},
  \\ & \quad \mbox{when} \quad -7m/8\leq k\leq (-1/2+2\gamma)m,\quad j\geq m-\gamma m.
\end{split}
\end{align}
\end{lemma}

\begin{proof}
Fix a time $s\in[0,t]$, assume $s\approx 2^n$, where $0\leq n\leq m$. By Duhamel formula (\ref{duhamel3}), and the commutation formula (\ref{commomg}),
it suffices to prove that
\begin{equation}\label{fhatimp'b}
\|\mathcal{F}Q_{jk}\mathcal{F}^{-1}I_k[g_1,g_2](s,\xi)\|_{L^{\infty}}\lesssim \e^2 2^{-n-k}2^{-m/16},
\end{equation} 
where
\begin{align}
\label{fhatimp'a}
I_k[g_1,g_2](s,\xi) := \int_{\R^3} e^{is(|\xi|\pm|\xi-\eta|\pm|\eta|)}
  \frac{|\xi-\eta|}{|\eta|} \widehat{g_1}(\xi-\eta) \widehat{g_2}(\eta) \,\mathrm{d}\eta,
\end{align} and $g_r=\Omega^{\ell_r}f_{\kappa_r}$ as before. In particular $g_r$ satisfy the bounds
\begin{align}\label{estgr1}
\quad \sup_{|\ell|\leq N/2}{\| \Omega^{\ell}P_kg_r \|}_{L^2} &\lesssim \e2^{-20k^+}2^{C\e n}, 
\\
\label{estgr2}{\| \widehat{Q_{j,k} g_r }\|}_{L^\infty} &\lesssim \e2^{-20k^+} \min\big(2^{(1+C\e)n}, 2^{-j/2-2k}2^{C\e n}, 2^{-k+\delta'n}\big),
\\
\label{estgr3}{\| \widehat{Q_{j,k} \partial_tg_r }\|}_{L^2} &\lesssim \e 2^{-20k^+}2^{-(1-C\e )n}, \quad {\| \widehat{Q_{j,k} \partial_tg_r }\|}_{L^\infty}
  \lesssim \e 2^{-n}2^{-20k^+}2^{-k+\delta'n+4\delta n}.
\end{align}
We decompose $g_r$ into $Q_{j_rk_r}g_r$, and subdivide the proof of \eqref{fhatimp'b} 
in various cases along the same lines of the proof of \eqref{fhatimp} above.

{\it Case $\max(n,\min(j_1,j_2))\leq (1-2\gamma) m$}.
Without loss of generality assume $j_1\leq (1-2\gamma)m$; we differentiate \eqref{fhatimp'a},
using the fact that 
\[\|\nabla^{\alpha}\widehat{Q_{j_1k_1}g_1}\|_{L^2}\lesssim 2^{j_1|\alpha|}\|g_1\|_{L^2},\] 
we obtain that
\begin{align*}
\big|\partial_{\xi}^{\alpha}I_k[Q_{j_1k_1}g_1,Q_{j_2k_2}g_2](s,\xi)\big|
  \lesssim 2^{|\alpha|\max(s,|k|,j_1)}\cdot 2^{k_1-k_2}\|g_1\|_{L^2}\|g_2\|_{L^2}
\\ \lesssim 2^{(1-2\gamma)m|\alpha|}\cdot 2^{k_1-k_2}\|g_1\|_{L^2}\|g_2\|_{L^2}.
\end{align*}
Since $j\geq (1-\gamma)m$, by choosing $|\alpha|$ big enough we see that
\[\|\mathcal{F}Q_{jk}\mathcal{F}^{-1}I_k[Q_{j_1k_1}g_1,Q_{j_2k_2}g_2](s,\xi)\|_{L^{\infty}}
  \lesssim 2^{-j|\alpha|} 2^{(1-2\gamma)m|\alpha|}\cdot 2^{k_1-k_2}\|g_1\|_{L^2}\|g_2\|_{L^2}\lesssim 2^{-100m}.\]

{\it Case $\min(j_1,j_2) \geq (1-4\delta)n$ or $\max(j_1,j_2) \geq 3m$}.
The case $\max(j_1,j_2) \geq 3m$ can be simply treated as in \eqref{fhatimp9}, so we skip it.
%and assume from now on that sums over $j_1,j_2$ are
We then assume $\min(j_1,j_2)\geq (1-4\delta)n$ and separate into different cases depending on the size of $k_2$ relative to $k$. Note also that $\min(j_1,j_2)\geq (1-4\delta-2\gamma)m$ since when $n\leq (1-2\gamma)m$ we must have $\min(j_1,j_2)\geq (1-2\gamma)m$.

\smallskip
\noindent
{\it Subcase 1: $k_2 \leq k-10$}.
In this case $k_2$ is the smallest frequency and $|k_1-k|\leq 10$.
If $k_2 \leq -2m/3$ we estimate
\begin{align*}
\begin{split}
\big| I_k[Q_{j_1,k_2}g_1, Q_{j_2,k_2}g_2](\xi) \big| & \lesssim 2^{k_1-k_2}
  {\|\widehat{Q_{j_1,k_1}g_1}\|}_{L^\infty} {\|\widehat{Q_{j_2,k_2}g_2}\|}_{L^\infty} 2^{3k_2}
  \\
& \lesssim\e^2 2^{k_1-k_2} \cdot 2^{-j_1/2-2k_1}\cdot 2^{-k_2+\delta'n} \cdot 2^{3k_2}
\\
& \lesssim\e^2 2^{-k_1+k_2} 2^{-j_1/2} 2^{\delta'n}
\\
& \lesssim\e^22^{-n-k}\cdot 2^{k_2+n/2}\cdot 2^{(\delta'+2\delta+\gamma)n}
\end{split}
\end{align*}
which suffices since $n\leq m$.
If instead  $k_2 \geq -2m/3$ we estimate
\begin{align*}
\begin{split}
\big| I_k[Q_{j_1,k_2}g_1, Q_{j_2,k_2}g_2](\xi) \big| &\lesssim 2^{k_1-k_2}
  {\|\widehat{Q_{j_1,k_1}g_1}\|}_{L^2} {\|\widehat{Q_{j_2,k_2}g_2}\|}_{L^\infty} \cdot 2^{3k_2/2}
  \\
&\lesssim\e^2 2^{k_1-k_2} \cdot 2^{-j_1-k_1} \cdot 2^{-k_2} 2^{(\delta'+C\e)n} \cdot 2^{3k_2/2}
\\
&\lesssim\e^2 2^{-j_1} \cdot 2^{-k_2/2} 2^{(\delta'+C\e)m}
\\
&\lesssim \e^22^{(-2/3+\delta'+4\delta+2\gamma+C\e)m}
\end{split}
\end{align*}
using that $n\leq m$, $-2m/3 \leq k_2 \leq k \leq -m/2 + 2\gamma m$ and $j_1\geq (1-4\delta-2\gamma)m$, which suffices.

\smallskip
\noindent
{\it Subcase 2: $k_2 \geq k-10$}.
%Let us assume that $|k_2- k_1|\leq 20$, and skip the other simpler case when $k_1$ is the smallest frequency.
Here we can use an $L^2\times L^2$ estimate and obtain
\begin{align*}
\begin{split}
\big| I_k[Q_{j_1,k_2}g_1, Q_{j_2,k_2}g_2](\xi) \big| & \lesssim 2^{k_1-k_2}
  {\|\widehat{Q_{j_1,k_1}g_1}\|}_{L^2} {\|\widehat{Q_{j_2,k_2}g_2}\|}_{L^2}
  \\
& \lesssim \e^2 2^{k_1-k_2}2^{-20k_1^+-20k_2^+} \cdot 2^{-j_1-k_1} 2^{C\e n} \cdot 2^{-j_2-k_2} 2^{C\e n}
\\
& \lesssim \e^2 2^{-j_1-j_2} 2^{-2k} 2^{-20k_1^+-20k_2^+}2^{C\e m}
\\
& \lesssim \e^2 2^{-20k_1^+-20k_2^+}2^{-2m-2k}\cdot 2^{(8\delta+4\gamma+C\e)m},
\end{split}
\end{align*}
which is more than sufficient since $k \geq -7m/8$ and $j_1,j_2 \geq (1-4\delta-2\gamma)m$.

{\it Case $\min(j_1,j_2) \leq (1-4\delta)n$}.
Using Lemma \ref{LemAIBP} we can restrict the angle between $\xi$ and $\eta$ in the integral $I_k$ in \eqref{fhatimp'a},
and reduce matters to estimating
\begin{align*}
\begin{split}
|I_{k,p_0}|[g_1,g_2](\xi) := \varphi_k(\xi) \int_{\R^3}
  \frac{|\xi-\eta|}{|\eta|} \big| \widehat{g_1}(\xi-\eta) \big| \, \big| \widehat{g_2}(\eta) \big|
  \, \chi\big(|\xi\wedge\eta| \, 2^{-(p_0+k+\min(k_1,k_2))}\big)\,\mathrm{d}\eta,
  \\ p_0:= -\frac{n + \min(k,k_1,k_2)}{2} + 2\delta n.
\end{split}
\end{align*}

\noindent
{\it Subcase 1: $k_2 \leq k_1-10$}.
In this case $|k_1-k|\leq 5$ and we estimate, using (\ref{fhatimpcor2}),
\begin{align*}
\begin{split}
| I_{k,p_0}|[Q_{j_1,k_2}g_1, Q_{j_2,k_2}g_2](\xi)  &\lesssim 2^{k_1-k_2}
  {\|\widehat{Q_{j_1,k_1}g_1}\|}_{L^\infty} {\|\widehat{Q_{j_2,k_2}g_2}\|}_{L^\infty} 2^{2p_0 + 3k_2}
\\
&\lesssim 2^{k_1} {\|\widehat{Q_{j_1,k_1}g_1}\|}_{L^\infty} \cdot 2^{k_2} {\|\widehat{Q_{j_2,k_2}g_2}\|}_{L^\infty} \cdot 2^{-n +4\delta n}
\\
&\lesssim \e^2 2^{-n+(4\delta+2\delta')n},
\end{split}
\end{align*}
which suffices since $k\leq -n/2+2\gamma n$ and $n\geq (1-2\gamma)m$. The case $k_1\leq k_2- 10$ is treated similarly and easier in view of the factor $2^{k_1-k_2}$.

\noindent
{\it Subcase 2: $|k_2 -k_1| \leq 10$, and $\max(j_1,j_2)\geq 2n/5$}.
Suppose $j_1\geq 2n/5$, we proceed similarly and estimate
\begin{align*}
\begin{split}
| I_{k,p_0}|[Q_{j_1,k_2}g_1, Q_{j_2,k_2}g_2](\xi) &\lesssim
  {\|\widehat{Q_{j_1,k_1}g_1}\|}_{L^\infty} {\|\widehat{Q_{j_2,k_2}g_2}\|}_{L^\infty} 2^{2p_0 + 3k_2}\\
&\lesssim 2^{3k_1/2} {\|\widehat{Q_{j_1,k_1}g_1}\|}_{L^\infty} \cdot 2^{3k_2/2} {\|\widehat{Q_{j_2,k_2}g_2}\|}_{L^\infty} \cdot 2^{-k}2^{-n +4\delta n}
\\
&\lesssim\e^2  2^{-k-n}2^{4\delta n}\cdot\min(2^{-(j_1+k_1)/2+C\e n}, 2^{k_1/2}) \cdot2^{C\e n},
\end{split}
\end{align*}
which suffices noticing that $n\geq (1-2\gamma)m$, and that
\[\min(-(j_1+k_1)/2,k_1/2)\leq -j_1/4\leq -n/10.\]

From now on we may assume $\max(j_1,j_2)\leq 2n/5$, so in particular $\min(k_1,k_2)\geq -2n/5$.
In this case we will have to consider the integral $I_{k,p_0}$ with the phase, namely
\begin{align*}
\begin{split}
I_{k,p_0}[Q_{j_1k_1}g_1,Q_{j_2k_2}g_2](\xi) & := \varphi_k(\xi) \int_{\R^3}e^{is\Phi(\xi,\eta)}
  \frac{|\xi-\eta|}{|\eta|} \widehat{Q_{j_1k_1}g_1}(\xi-\eta)\, \widehat{Q_{j_2k_2}g_2}(\eta)
  \\ & \times \, \chi\big(|\xi\wedge\eta| \, 2^{-(p_0+k+k_1)}\big)\,\mathrm{d}\eta, \qquad
  \Phi(\xi,\eta)= |\xi|-\kappa_1|\xi-\eta|-\kappa_2|\eta|.
\end{split}
\end{align*}For simplicity we also assume $\max(k_1,k_2)\leq 0$.
Recall that  in the region of integration we have $\big|\xi/|\xi|\pm\eta/|\eta|\big|\lesssim 2^{p_0}$.

{\it Case $\kappa_1=\kappa_2$, or when $\big|\xi/|\xi|+\kappa_1\eta/|\eta|\big|\lesssim 2^{p_0}$}.
If $\kappa_1=\kappa_2$, since $k\leq -m/2+2\gamma m$ and $k_1,k_2\geq -2m/5$, we know that in the region of integration we have
\[|\nabla_{\eta}\Phi|=\bigg|\frac{\eta}{|\eta|}+\frac{\eta-\xi}{|\eta-\xi|}\bigg|\gtrsim 1,
\] so we get $|I_{k,p_0}[Q_{j_1k_1}g_1,g_2]|\lesssim 2^{-100m}$ via integrating by parts in $\eta$ many times, noticing also that
\[\big|\nabla_\eta^\alpha\chi\big((\xi\wedge\eta)2^{-p_0-k-k_1}\big)\big|\lesssim 2^{|\alpha|(-p_0-k_1)},\quad -p_0-k_1\leq 7m/8,
\] due to Fa\`{a} di Bruno's formula and the observation that $(\xi 2^{-k})\wedge\eta$ has all $\eta$ derivatives bounded.

If $\kappa_2=-\kappa_1$, and $\big|\xi/|\xi|+\kappa_1\eta/|\eta|\big|\lesssim 2^{p_0}$,
then in the region of integration we have $|\nabla_\xi\Phi|\lesssim 2^{p_0+}$, which implies, via Fa\`{a} di Bruno's formula, that
\[\big|\nabla_{\xi}^{\alpha}e^{is\Phi}\big|\lesssim 2^{|\alpha|\max(m+p_0,(m-k)/2)}\lesssim 2^{(31m/32)|\alpha|}.\]
Similarly
\[\big|\nabla_\xi^\alpha\chi\big((\xi\wedge\eta)2^{-p_0-k-k_1}\big)\big|\lesssim 2^{|\alpha|(-p_0-k)},\quad -p_0-k\leq 31m/32,
\]
so by taking $\xi$ derivatives in the integral definition of $I_{k,p_0}$ we get that
\[\big|\partial_{\xi}^{\alpha}I_{k,p_0}[Q_{j_1k_1}g_1,Q_{j_2k_2}g_2](s,\xi)\big|\lesssim 2^{(31m/32)|\alpha|}\cdot 2^{k_1-k_2}\|g_1\|_{L^2}\|g_2\|_{L^2}.
\]
Since $j\geq (1-\gamma)m$, by choosing $|\alpha|$ big enough we see that
\[\|\mathcal{F}Q_{jk}\mathcal{F}^{-1}I_{k,p_0}[Q_{j_1k_1}g_1,Q_{j_2k_2}g_2](s,\xi)\|_{L^{\infty}}
  \lesssim 2^{-j|\alpha|} 2^{(31m/32)|\alpha|}\cdot 2^{k_1-k_2}\|g_1\|_{L^2}\|g_2\|_{L^2}\lesssim 2^{-100m}.\]

\bigskip
{\it Case $\kappa_2=-\kappa_1$, and $\big|\xi/|\xi|-\kappa_1\eta/|\eta|\big|\lesssim 2^{p_0}$}. We may assume $\kappa_1=1$, $\kappa_2=-1$, and $|\angle(\xi,\eta)|\lesssim 2^{p_0}$ (the other case being similar). In this case we have
\[|\Phi|=\big||\xi|-|\xi-\eta|+|\eta|\big|\gtrsim |\xi|\approx 2^{k},
\] so we will integrate by parts in time. Recall that to prove (\ref{fhatimp'}) it suffices to show
\begin{equation}
\label{fhatimp'd}
|J_{k,p_0}[Q_{j_1k_1}g_1,Q_{j_2k_2}g_2](\xi)|\lesssim\varepsilon^2 2^{-k-m/16},
\end{equation}
where
\[J_{k,p_0}[Q_{j_1k_1}g_1,Q_{j_2k_2}g_2](\xi)=\int_{\mathbb{R}}\tau_n(s)I_{k,p_0}[Q_{j_1k_1}g_1,Q_{j_2k_2}g_2](s,\xi)\,\mathrm{d}s.
\]
Now, integrating by parts in $s$, we get that
\begin{multline*}
J_{k,p_0}[Q_{j_1k_1}g_1,Q_{j_2k_2}g_2](\xi)=-\int_{\R} \tau_n^\prime(s)K[Q_{j_1k_1}g_1,Q_{j_2k_2}g_2](s,\xi)\,\mathrm{d}s
\\
-\int_{\R}\tau_n(s)K[Q_{j_1k_1} \partial_sg_1, Q_{j_2k_2}g_2](s,\xi)\,\mathrm{d}s
  -\int_{\R} \tau_n(s)K[Q_{j_1k_1}g_1, Q_{j_2k_2} \partial_sg_2](s,\xi)\,\mathrm{d}s,
\end{multline*} where
\begin{align*}
\begin{split}
K[Q_{j_1k_1}g_1,Q_{j_2k_2}g_2](s,\xi) :=  \int_{\R^3}
  \frac{|\xi-\eta|}{|\eta|\Phi} \big| \widehat{Q_{j_1k_1}g_1}(\xi-\eta) \big| \, \big| \widehat{Q_{j_2k_2}g_2}(\eta) \big|
  \, \chi\big(\angle (\xi,\eta)2^{-p_0}\big)\,\mathrm{d}\eta.
\end{split}
\end{align*}
Now we have
\begin{align*}
K[Q_{j_1k_1}g_1,Q_{j_2k_2}g_2](s,\xi)\lesssim
2^{-k}{\|\widehat{Q_{j_1,k_1}g_1}\|}_{L^\infty} {\|\widehat{Q_{j_2,k_2}g_2}\|}_{L^\infty} 2^{2p_0 + 3k_2}
\\
\lesssim 2^{-k}2^{3k_1/2} {\|\widehat{Q_{j_1,k_1}g_1}\|}_{L^\infty} \cdot 2^{3k_2/2} {\|\widehat{Q_{j_2,k_2}g_2}\|}_{L^\infty} \cdot 2^{-k}2^{-n +4\delta n}
\\
\lesssim \e^2 2^{-2k-n+(4\delta+C\e)n}
\lesssim \e^2 2^{-k}\cdot 2^{-k-n} 2^{(4\delta+C\e)n}
\end{align*}
which suffices since $k\geq -7m/8$ and $n\geq (1-2\gamma)m$; similarly we have
\begin{align*}
K[Q_{j_1k_1}\partial_tg_1,Q_{j_2k_2}g_2](s,\xi)\lesssim
2^{-k}{\|\widehat{Q_{j_1,k_1}\partial_tg_1}\|}_{L^\infty} {\|\widehat{Q_{j_2,k_2}g_2}\|}_{L^\infty} 2^{2p_0 + 3k_2}
\\
\lesssim 2^{-k}2^{3k_1/2} {\|\widehat{Q_{j_1,k_1}\partial_tg_1}\|}_{L^\infty} \cdot 2^{3k_2/2} {\|\widehat{Q_{j_2,k_2}g_2}\|}_{L^\infty} \cdot 2^{-k}2^{-n +4\delta n}
\\
\lesssim \e^2 2^{-2k-n-n+(\delta'+4\delta+C\e)n} \lesssim \e^2 2^{-k}\cdot 2^{-k-n} \cdot 2^{(\delta'+4\delta+C\e)n},
\end{align*}
which also suffices since $k\geq -7m/8$ and $n\geq (1-2\gamma)m$. This completes the proof.
\end{proof}

\bigskip
\section{Proof of Theorem \ref{Mainth1}: The Asymptotic System}\label{SecProof1}

The proof of Theorem \ref{Mainth1} will be built on the approximation Lemmas \ref{Lem--}--\ref{LemAux}.
We first state all these lemmas and then prove them in the remaining of this section.

\medskip
\subsection{The main Approximation Lemmas}\label{secMainLemmas}
In what follows, we will always assume $m\geq 1$,
and the a priori bounds \eqref{Ali0}-\eqref{Aliinfty} and their consequences, see Lemma \ref{lemEE}, Lemma \ref{LemBB} and the result in Subsection \ref{secImp}.

We will frequently write a vector $\xi\in\R^3$ in polar coordinates,
say $\xi=\rho\theta$ where $\rho=|\xi|$ and $\theta=\xi/|\xi|$; we may use similar notations such as $\xi=r\phi$, etc.
In these coordinates we will define $F_{\theta}(t,\rho):=\widehat{f}(t,\rho\theta)$, and similarly $F_{\phi}(t,r):=\widehat{f}(t,r\phi)$, etc.

\begin{lemma}[The $--$ interactions]\label{Lem--}
We have
\begin{align}
\label{Lem--.1}
I_{--}^m(t,\xi) \in \mathcal{R},
\end{align}
see \eqref{duhamel2} and Definition \ref{def} for notations.
\end{lemma}
This is proven in Section \ref{PrLem--} by exploiting oscillations in time
and making use of the angular integration by parts Lemma \ref{LemAIBP}.

We then look at low frequency outputs and prove the following:
\begin{lemma}[Low Frequency Outputs]\label{LemL1}
We have
\begin{align}
\label{LemL1.1}
|\xi|^{-1/5}\varphi_{\leq0} \big(\xi\langle t \rangle^{7/8}\big) I_{++}^m(t,\xi), \quad
|\xi|^{-1/5}\varphi_{\leq0} \big(\xi\langle t \rangle^{7/8}\big) I_{--}^m(t,\xi) \in \mathcal{R}.
\end{align}
Moreover, %recall the notation $\xi=\rho\theta$ as above,
we have
\begin{align}
\label{LemL1.2}
|\xi|^{-1/5}\Big[ \varphi_{\leq0} \big(\xi\langle t \rangle^{7/8}\big) \big( I_{+-}^m(t,\xi) + I_{-+}^m(t,\xi) \big)
  - \frac{1}{2(2\pi)^{3/2}} \tau_m(t)h_\theta(t,\rho) \Big] \in \mathcal{R}
\end{align}
where
\begin{align}
\label{LemL1.2'}
\begin{split}
%& F_\theta(t,\rho) := \widehat{f}(t,\rho\theta), \\
& h_\theta(t,\rho) :=  \varphi_{\leq0}(\rho\langle t \rangle^{7/8})
\int_\R \int_{\mathbb{S}^2} e^{it\r(1-\theta\cdot\phi)} \big| F_\phi(t,r) \big|^2 r^2 \,\mathrm{d}\phi \mathrm{d}r.
\end{split}
\end{align}

In particular, by Duhamel's formula \eqref{duhamel}-\eqref{duhamel3}, it follows that
\begin{align}
\label{LemL1.3}
\Big\| |\xi|^{-1/5} \Big( \varphi_{\leq -10}(\xi\langle t\rangle^{7/8})[ \widehat{f}(t,\xi) - \widehat{f}(0,\xi) ]
  + \frac{1}{2(2\pi)^{3/2}} H_{\theta}(t,\rho) \Big) \Big\|_X\lesssim \e^2,
\end{align}
where
\begin{align}
\label{LemL1.4}
H_\theta(t,\rho) := \varphi_{\leq-10}(\rho\langle t\rangle^{7/8})\int_0^t h_{\theta}(s,\rho) \, \mathrm{d}s.
\end{align}
\end{lemma}

The above Lemma is proven in \ref{PrLemL1}.
Notice that in \eqref{LemL1.1} we prove stronger bounds than what is needed for the terms to be considered acceptable remainders.
These bounds will be used later on in the proof of Lemma \ref{LemHL1}.

Next, we decompose $I_{\kappa_1\kappa_2}$ to distinguish between high and low frequency inputs. Define
\begin{equation}\label{fhigh}\quad \widehat{f_{\kappa_j}^0}(t,\xi)=\varphi_{\leq -10}(\xi\langle t\rangle^{7/8})\widehat{f_{\kappa_j}}(t,\xi);
  \quad \widehat{f_{\kappa_j}^1}(t,\xi)=\varphi_{> -10}(\xi\langle t\rangle^{7/8})\widehat{f_{\kappa_j}}(t,\xi),
\end{equation}
and
\begin{equation}\label{defim01}I_{\kappa_1\kappa_2}^{m,0}=I_{\kappa_1\kappa_2}^m[f_{\kappa_1}^0,f_{\kappa_2}]+I_{\kappa_1\kappa_2}^m[f_{\kappa_1}^1,f_{\kappa_2}^0];\quad I_{\kappa_1\kappa_2}^{m,1}=I_{\kappa_1\kappa_2}^m[f_{\kappa_1}^1,f_{\kappa_2}^1].
\end{equation}

The next lemma treats the high-low frequency interactions $I^{m,0}_{\kappa_1\kappa_2}$ and shows the appearance of
a low frequency bulk term correction. It will be proven in \ref{PrLemHL1}.

\begin{lemma}[High-Low Interactions]\label{LemHL1}
We have
\begin{align}
\label{LemHL1.1}
\varphi_{>0}(\xi\langle t\rangle^{7/8}) I_{-+}^{m,0}(t,\xi) %\quad \varphi_{\geq (-1+\delta)m}(\xi) I_{--}^{m,0}(t,\xi)
\in \mathcal{R}.
\end{align}
Moreover
\begin{align}
\label{LemHL1.2}
\varphi_{>0}(\xi\langle t\rangle^{7/8}) \big[ I_{++}^{m,0}(t,\xi) - I_{+-}^{m,0}(t,\xi) \big]
  + i \rho F_\theta(t,\rho)\cdot B_\theta(t) \tau_m(t)\in \mathcal{R}
\end{align}
where
\begin{align}
\label{LemHL1.3}
& B_\theta(t) := \frac{1}{32\pi^3} \mathrm{Im} \Big( \int_0^\infty \! \int_{\mathbb{S}^2}
  e^{itr(1-\theta\cdot\phi)} H_\phi(t,r) \,r \,\mathrm{d}\phi \mathrm{d}r \Big),
\end{align} and $H_{\phi}(t,r)$ is defined in (\ref{LemL1.4}).
\end{lemma}

The next lemma treats the high-high frequency interactions $I^{m,1}_{\kappa_1\kappa_2}$, and shows that the leading order behavior is governed by an asymptotic PDE.

\begin{lemma}[High-High Interactions and the Asymptotic PDE]\label{LemHH1}
Suppose $m\geq 1$, define
\begin{align}
G_\theta(t,\rho) &:= F_\theta(t,\rho)\varphi_{> -10}(\rho\langle t\rangle^{7/8}), \quad\rho>0;
\\
G_\theta(t,\rho) &:= \overline{G_\theta(t,-\rho)},\quad \rho<0,
\end{align}
then we have
\begin{align}
\varphi_{\geq 0}(\xi\langle t\rangle^{7/8}) \big[ I_{++}^{m,1}(t,\xi) - I_{+-}^{m,1}(t,\xi)  - I_{-+}^{m,1}(t,\xi)  \big]
  + \tau_m(t)\mathcal{N}(t,\rho) \in \mathcal{R}
\end{align}
where
\begin{align}\label{1dint}
\mathcal{N}(t,\rho) := \frac{i}{4\sqrt{2\pi}} \frac{1}{t}\int_\R \frac{(\rho-r)^2}{\rho} G_\theta(t,\rho-r)G_\theta(t,r) \, \mathrm{d}r.
\end{align}
\end{lemma}
The proofs of Lemma \ref{LemHH1} is performed in Section \ref{PrLemHH1}.

The last lemma reduces the function $G_\theta$ on the right hand side of (\ref{1dint}) back to $F_{\theta}$, thus recovering the
expression on right-hand side of (\ref{asymptpde}):
\begin{lemma}\label{LemAux} For $m\geq 1$ we have
\begin{multline}\label{LemAux1}
\tau_m(t)\mathcal{N}(t,\rho) \frac{i}{4\sqrt{2\pi}} \frac{1}{t}\int_\R \frac{(\rho-r)^2}{\rho} F_\theta(t,\rho-r)F_\theta(t,r) \, \mathrm{d}r
\\
-\tau_m(t)i\rho F_{\theta}(t,\rho)\cdot \frac{1}{32\pi^2}\int_{\R}\varphi_{\leq -10}(r\langle t\rangle^{7/8})
  H_\theta(t,r)\,\mathrm{d}r \in \mathcal{R}.
\end{multline}
\end{lemma}

\medskip
\subsection{Proof of Lemma \ref{Lem--}}\label{PrLem--}
Recall that
\begin{align}
\label{pr--0}
\begin{split}
I_{--}^m[g_1,g_2](t,\xi) &:= \frac{1}{4}(2\pi)^{-3/2} \tau_m(t)\int_{\R^3} e^{it\Phi_{--}(\xi,\eta)}\frac{|\xi-\eta|}{|\eta|}
  \widehat{\bar{g_1}}(t,\xi-\eta)\widehat{\bar{g_2}}(t,\eta)\,\mathrm{d}\eta\mathrm{d}s,
\\
\Phi_{--}(\xi,\eta) &:= |\xi|+|\xi-\eta|+|\eta|.
\end{split}
\end{align}
Since the phase $\Phi_{--}$ is elliptic, $\Phi_{--}\approx \max(|\xi|,|\xi-\eta|,|\eta|)$, it is convenient to write
\begin{align}
\label{pr--1}
\begin{split}
& I_{--}^{m}[g_1,g_2] = \partial_t \big( \tau_m K_{--} [g_1,g_2] \big) - \tau_m^\prime  K_{--}[g_1,g_2] - \tau_m K_{--}[\partial_t g_1,g_2]
  - \tau_m K_{--}[g_1,\partial_t g_2]
\\
& K_{--}[g_1,g_2](t,\xi) := \int_{\R^3} e^{it\Phi_{--}(\xi,\eta)} \frac{|\xi-\eta|}{|\eta|\Phi_{--}(\xi,\eta)}
  \widehat{\overline{g_1}}(t,\xi-\eta)\widehat{\overline{g_2}}(t,\eta)\,\mathrm{d}\eta\mathrm{d}s.
\end{split}
\end{align}
Now, by using \eqref{commomg}, we only need to estimate $I_{--}^m[g_1,g_2]$ with
\begin{equation}
\label{defgr}
g_1=\Omega^{\ell_1}f,\quad g_2=\Omega^{\ell_2}f,\quad |\ell_1|+|\ell_2|\leq N_1.
\end{equation}
Note that they satisfy the estimates
\begin{align}\label{estgr1'}
& \sup_{|\ell|\leq N/2}{\| \Omega^{\ell}P_kg_r \|}_{L^2} \lesssim \e2^{-20k^+}2^{C\e m},
\\
\label{estgr2'}
& {\| \widehat{\Omega Q_{j,k} g_r }\|}_{L^\infty} + {\| \widehat{Q_{j,k} g_r }\|}_{L^\infty}
  \lesssim \e2^{-20k^+} \min\big(2^{(1+C\e)m}, 2^{-j/2-2k}2^{C\e m}, 2^{-k} 2^{\delta'm}\big),
\end{align}
and
\begin{align}
\label{estgr3'}
\begin{split}
& {\| \widehat{Q_{j,k} \partial_tg_r }\|}_{L^2} \lesssim \e 2^{-20k^+}2^{-(1-C\e )m},
\\
& {\| \widehat{Q_{j,k} \partial_tg_r }\|}_{L^\infty} \lesssim \e 2^{-k}2^{-20k^+} 2^{-m+\delta'm+4\delta m}.
\end{split}
\end{align}
We begin by decomposing the two inputs in frequencies $k_1,k_2$ and distinguish a few cases.

\medskip
{\it Case 1: $\min(k_1,k_2) \leq -3m/4$}.
Let us start by looking at the term $K_{--}[g_1,g_2]$, and estimate, for all $|\xi|\approx 2^k$ and $t \approx 2^m$
\begin{align*}
2^k \big| K_{--}[P_{k_1}g_1, P_{k_2}g_2](\xi) \big| & \lesssim
  2^k \cdot 2^{-k_2} {\|\what{P_{k_1}g_1}\|}_{L^\infty} {\|\what{P_{k_2}g_2}\|}_{L^\infty} \cdot 2^{3\min(k_1,k_2)}
\end{align*}
We can then sum over frequencies using \eqref{fhatimpcor2}:
\begin{align*}
& \sum_{\min(k_1,k_2) \leq -3m/4} 2^k 2^{15k^+} \big| K_{--}[P_{k_1}g_1, P_{k_2}g_2](\xi) \big|
\\
& \lesssim \sum_{\min(k_1,k_2) \leq -3m/4} 2^{15k^+} 2^{\max(k_1,k_2)+2\min(k_1,k_2)} {\|\what{P_{k_1}g_1}\|}_{L^\infty}
  {\|\what{P_{k_2}g_2}\|}_{L^\infty}
\\
& \lesssim 2^{\delta'm} \e \cdot 2^{\delta'm} \e \cdot 2^{-3m/4} \lesssim \e^2 2^{-\gamma m}.
\end{align*}
According to \eqref{pr--1} and the definition of acceptable remainder \ref{def}, this suffices to deal with this term.
For the same reason the term $\tau_m^\prime K_{--}[g_1,g_2]$ in \eqref{pr--1} is also an acceptable remainder
when the frequencies of the inputs satisfy $\min(k_1,k_2) \leq -3m/4$, see \eqref{timedecomp}.

To deal with the other two terms we estimate similarly, using \eqref{fhatimpcor2} and \eqref{dtf1}:
\begin{align*}
& \sum_{\min(k_1,k_2) \leq -3m/4} 2^k 2^{15k^+} \big| K_{--}[P_{k_1}g_1, P_{k_2}\partial_tg_2](\xi) \big|
\\
& \lesssim \sum_{\min(k_1,k_2) \leq -3m/4} 2^{15k^+}
  2^{\max(k_1,k_2)+2\min(k_1,k_2)} {\|\what{P_{k_1}g_1}\|}_{L^\infty} {\|\what{P_{k_2}\partial_t g_2}\|}_{L^\infty}
\\
& \lesssim 2^{-3m/4} \cdot \e 2^{\delta'm} \cdot \e 2^{-m+\delta'm+4\delta m} \lesssim \e^2 2^{-(1+\gamma) m}.
\end{align*}
The term $K_{--}[\partial_t g_1,g_2]$ can be handled in the same way.

Next, we decompose $g_r=\sum_{(j_r,k_r)}Q_{j_rk_r}g_r$, and treat first the case when $\min(j_1,j_2)$ is large.

\medskip
{\it Case 2: $\min(j_1,j_2) \geq 9m/10$ or $\max(j_1,j_2) \geq 3m$}. When $\max(j_1,j_2) \geq 3m$ estimating directly using (\ref{pr--0}) suffices. When $\min(j_1,j_2) \geq 9m/10$, the starting point is again the identity \eqref{pr--1}.
Then, using the bound \eqref{aprioriLinfty}, for all $|\xi|\approx2^k$ and
$t\approx 2^m$, we estimate
\begin{align*}
2^k 2^{15k^+} \big| K_{--}[Q_{j_1k_1}g_1, Q_{j_2k_2}g_2](\xi) \big| \lesssim
  2^k 2^{15k^+} \cdot 2^{-k_2} {\|\what{Q_{j_1k_1}g_1}\|}_{L^\infty} {\|\what{Q_{j_2k_2}g_2}\|}_{L^\infty} \cdot 2^{3\min(k_1,k_2)}
\\
\lesssim 2^{-5\max(k_1,k_2)^+} 2^{\max(k_1,k_2)+2\min(k_1,k_2)}
  \cdot \e 2^{-j_1/2 -2k_1}\cdot \e 2^{-j_2/2 -2k_2}2^{C\e m}
\\
\lesssim 2^{-5\max(k_1,k_2)^+} 2^{-\max(k_1,k_2)} 2^{C\e m} \cdot \e^2 2^{-(j_1+j_2)/2}
\end{align*}
For $j_1,j_2 \geq 9m/10$, and $-3m/4 \leq \min(k_1,k_2) \leq \max(k_1,k_2)$,
we can sum the above bound over all parameters $k_1,k_2,j_1,j_2$ obtaining the desired bound of $\e^2 2^{-\gamma m}$.
By the same argument, the term $\tau_m^\prime K_{--}[g_1,g_2]$ is also an acceptable remainder.

For the other two terms in \eqref{pr--1} we proceed similarly as above but also using \eqref{dtf2}:
\begin{align*}
2^k 2^{15k^+} \big| K_{--}[Q_{j_1k_1}g_1, \partial_t Q_{j_2k_2}g_2](\xi) \big| \lesssim
  2^{15k^+}  2^{\max(k_1,k_2)+2\min(k_1,k_2)} {\|\what{Q_{j_1k_1}g_1}\|}_{L^\infty} {\|\what{\partial_t Q_{j_2k_2}g_2}\|}_{L^\infty}
\\
\lesssim 2^{-5\max(k_1,k_2)^+} 2^{\max(k_1,k_2)+2\min(k_1,k_2)} \cdot \e 2^{-j_1/2 -2k_1}2^{C\e m}  \cdot\e 2^{-k_2}2^{-2m/3}
\\
\lesssim 2^{-5\max(k_1,k_2)^+} 2^{C\e m} \cdot \e^2 2^{-2m/3} 2^{-j_1/2}.
\end{align*}
%{Leave $2/3$?}
Using that $j_1 \geq 9m/10$ and summing over all parameters gives a more than sufficient bound of $\e^2 2^{-11m/10}$.
The term $K_{--}[ \partial_t f,f]$ can be treated identically.

\medskip
{\it Case 3: $\min(k_1,k_2) \geq -3m/4$, and $\min(j_1,j_2) \leq 9m/10$}.
In this last case we can use Lemma \ref{LemAIBP},
thanks to the restrictions on the frequencies and spatial localization guaranteeing that the hypotheses \eqref{LemAIBP2}-\eqref{LemAIBP3}
are satisfied. The conclusion \eqref{LemAIBPconc} then takes care of the integral in \eqref{pr--0} with inputs $f_1,f_2$,
%$f_\ell = Q_{j_\ell,k_\ell}f$ as above,
when the support is restricted to $|\xi\wedge\eta| \gtrsim 2^{p_0+k+\min(k_1,k_2)}$ with
\begin{align*}
p_0 := -\frac{m + \min(k,k_1,k_2)-4\delta m}{2}.
\end{align*}
Therefore, matters are reduced to estimating the term
\begin{multline}
\label{pr--10}
I_{--,p_0}^m[Q_{j_1k_1}g_1,Q_{j_2k_2}g_2](t,\xi) := \tau_m(t)\int_{\R^3} e^{it\Phi_{--}(\xi,\eta)}\frac{|\xi-\eta|}{|\eta|}
 \\ \times\widehat{\overline{Q_{j_1k_1}g_1}}(t,\xi-\eta)\widehat{\overline{Q_{j_2k_2}g_2}}(t,\eta)\, \chi(2^{-(p_0+k+\min(k_1,k_2))}\xi\wedge\eta)\, \mathrm{d}\eta,
\end{multline}
for a compactly supported function $\chi$.
Similarly to \eqref{pr--1} we can write
\begin{align}
\label{pr--11}
\begin{split}
I_{--,p_0}^{m}[Q_{j_1k_1}g_1,Q_{j_2k_2}g_2] = \partial_t \big( \tau_m K_{--,p_0} [Q_{j_1k_1}g_1,Q_{j_2k_2}g_2] \big)
  - \tau_m^\prime  K_{--,p_0}[Q_{j_1k_1}g_1,Q_{j_2k_2}g_2]
  \\ - \tau_m K_{--,p_0}[\partial_t Q_{j_1k_1}g_1,Q_{j_2k_2}g_2] - \tau_m K_{--,p_0}[Q_{j_1k_1}g_1,\partial_t Q_{j_2k_2}g_2]
\end{split}
\end{align}
where
\begin{align}
\label{pr--12}
\begin{split}
K_{--,p_0}[g_1,g_2](t,\xi) := \int_{\R^3} e^{it\Phi_{--}(\xi,\eta)} \frac{|\xi-\eta|}{|\eta|\Phi_{--}(\xi,\eta)}
  \widehat{\overline{g_1}}(t,\xi-\eta)\widehat{\overline{g_2}}(t,\eta)\,
  \\ \times \chi(2^{-(p_0+k+\min(k_1,k_2))}\xi\wedge\eta)\, \mathrm{d}\eta\mathrm{d}s.
\end{split}
\end{align}
Notice that, for each fixed $\xi$, the support of the above integral is contained in a cone of aperture $2^{p_0}$ and length $2^{\min(k_1,k_2)}$, whose measure is $2^{2p_0 + 3\min(k_1,k_2)}$.
Then, for $t\approx2^m$ we can estimate
\begin{align*}
2^k 2^{15k^+} \big| K_{--,p_0}[Q_{j_1k_1}g_1, Q_{j_2k_2}g_2](\xi) \big| \lesssim
  2^k 2^{15k^+} \cdot 2^{-k_2} 2^{k_1-\max(k_1,k_2)}\\\times{\|\what{Q_{j_1k_1}g_1}\|}_{L^\infty} {\|\what{Q_{j_2k_2}g_2}\|}_{L^\infty}  2^{2p_0 + 3\min(k_1,k_2)}
\\
\lesssim 2^{15k^+} \cdot 2^{-m(1-4\delta)} \cdot 2^{\max(k_1,k_2) + \min(k_1,k_2)}{\|\what{Q_{j_1k_1}g_1}\|}_{L^\infty} {\|\what{Q_{j_2k_2}g_2}\|}_{L^\infty}
\\
\lesssim 2^{-5\max(k_1,k_2)^+} \cdot 2^{-m(1-4\delta)} \cdot \big( \e 2^{\delta'm}\big)^2
\end{align*}
where we have used \eqref{fhatimpcor2}, and this bound clearly suffices. This also takes care of the term $2^{-m} K_{--,p_0}[Q_{j_1k_1}g_1,Q_{j_2k_2}g_2]$.

In a similar fashion, we can deal with the other terms in \eqref{pr--11}, by using in addition \eqref{dtf2} as follows:
\begin{align*}
2^k 2^{15k^+} \big| K_{--,p_0}[Q_{j_1k_1}g_1, \partial_t Q_{j_2k_2}g_2](\xi) \big| \lesssim
2^k 2^{15k^+} \cdot 2^{-k_2} 2^{k_1-\max(k_1,k_2)}
\\
\times{\|\what{Q_{j_1k_1}g_1}\|}_{L^\infty} {\|\what{\partial_tQ_{j_2k_2}g_2}\|}_{L^\infty}  2^{2p_0 + 3\min(k_1,k_2)}
\\
\lesssim 2^{15k^+} \cdot 2^{-m(1-4\delta)}
  \cdot 2^{\max(k_1,k_2) + \min(k_1,k_2)}{\|\what{Q_{j_1k_1}g_1}\|}_{L^\infty} {\|\what{\partial_t Q_{j_2k_2}g_2}\|}_{L^\infty}
\\
\lesssim 2^{-5\max(k_1,k_2)^+} \cdot 2^{-m(1-4\delta)} \cdot \e 2^{\delta'm} \cdot \e 2^{-2m/3}
\end{align*}
%{Leave $2/3$?}
Summing over all parameters we see that we get at least a bound of $\e^2 2^{-5m/4}$ which suffices.
The last term $K_{--,p_0}[ \partial_tQ_{j_1k_1}g_1 ,Q_{j_2k_2}g_2]$ can be treated identically. This concludes the proof of this Lemma. \qed

\medskip
\subsection{Proof of Lemma \ref{LemL1}}\label{PrLemL1}

\begin{proof}[Proof of \eqref{LemL1.1}]
We begin by showing the first claim in \eqref{LemL1.1}, and in particular the following:
%Left stronger estimate here
\begin{align}
\label{pr1}
\begin{split}
& |\xi|^{4/5}\varphi_{\leq 0}(\xi\langle t\rangle^{7/8}) \Omega^{\ell}I_{++}^m(t,\xi) = R(t,\xi) + \partial_t S(t,\xi),\quad |\ell|\leq N_1,
\\
& |R(t,\xi)|\lesssim \e^2 2^{-(1+\gamma)m}, \qquad |S(t,\xi)|\lesssim \e^2 2^{-\gamma m}, \qquad t \approx 2^m, \quad m\geq 0.
\end{split}
\end{align}

By using (\ref{commomg}), we again reduce to estimating $I_{++}^m[g_1,g_2]$, where $g_{1,2}$ are defined in \eqref{defgr}.
We let $|\xi| \approx 2^k$, $k\in\Z$, and begin by splitting relative to the size of $|\eta|$:
\begin{align*}
\begin{split}
4(2\pi)^{3/2} I_{++}^m(t,\xi) & = I_{++}^{m,1}(t,\xi) +  I_{++}^{m,2}(t,\xi),
\\
I_{++}^{m,1}(t,\xi) &:= \tau_m(t)\int_{\R^3} e^{it\Phi_{++}(\xi,\eta)}\frac{|\xi-\eta|}{|\eta|}
\widehat{g_1}(t,\xi-\eta)\widehat{g_2}(t,\eta) \varphi_{\leq k + 10}(\eta) \,\mathrm{d}\eta\mathrm{d}s,
\\
I_{++}^{m,2}(t,\xi) &:= \tau_m(t)\int_{\R^3} e^{it\Phi_{++}(\xi,\eta)}\frac{|\xi-\eta|}{|\eta|}
\widehat{g_1}(t,\xi-\eta)\widehat{g_2}(t,\eta) \varphi_{> k + 10}(\eta) \, \mathrm{d}\eta\mathrm{d}s.
\end{split}
\end{align*}
Using \eqref{fhatimp}
\begin{align*}
\big| I_{++}^{m,1}(t,\xi) \big| \lesssim {\| |\xi|\widehat{g_r}(t)\|}_{L^\infty}^2 \cdot {\| |\xi|^{-2} \varphi_{\leq k + 10}(\cdot)\|}_{L^1}
  \lesssim \e^2 2^{\delta' m} \cdot 2^k.
\end{align*}
After multiplication by $\varphi_{\leq 0}(\xi\langle t\rangle^{7/8})|\xi|^{4/5}$ this term can be absorbed into $R$ in \eqref{pr1}.

The term $I_{++}^{m,2}$ is treated using integration by parts. More precisely we write
\begin{align}
\label{pr2}
\begin{split}
& I_{++}^{m,2} = \partial_t \big( \tau_m K_{++} [g_1,g_2] \big) - \tau_m^\prime K_{++}[g_1,g_2] - \tau_m K_{++}[\partial_t g_1,g_2]
  - \tau_m K_{++}[g_1,\partial_t g_2]
\\
& K_{++}[g_1,g_2](t,\xi) :=\sum_k \varphi_k(\xi)\int_{\R^3} e^{it\Phi_{++}(\xi,\eta)}\frac{|\xi-\eta|}{|\eta|\Phi_{++}(\xi,\eta)}
\widehat{g_1}(t,\xi-\eta)\widehat{g_2}(t,\eta) \varphi_{> k + 10}(\eta) \, \mathrm{d}\eta\mathrm{d}s.
\end{split}
\end{align}
Notice that for $|\xi|\approx 2^k$ with $k \leq -7m/8$, on the support of $K_{++}$ %and for $t\approx 2^m$,
we have $|\eta| \geq 2^{10} |\xi|$, and therefore $|\Phi_{++}(\xi,\eta)| \gtrsim |\eta| \approx |\xi-\eta|$.
It is then not hard to see that for all $k_2 \geq k + 10$ we have for $|k_1-k_2|\leq 5$ that
\begin{align}
\label{pr3}
\big| \varphi_k(\xi) K_{++}[P_{k_1}g_1,P_{k_2}g_2](\xi) \big| \lesssim
  2^{k_2/2}\min\big( {\|P_{k_1} g_1\|}_{L^2} {\|\widehat{P_{k_2}g_2}\|}_{L^\infty},
  {\|\widehat{P_{k_1}g_1}\|}_{L^\infty} {\|P_{k_2}g_2\|}_{L^2} \big).
\end{align}
Using \eqref{pr3}, followed by \eqref{aprioriLinfty4}, we immediately obtain, for all $t\approx 2^m$
\begin{align*}
\big| |\xi|^{4/5} \varphi_{\leq 0}(\xi\langle t\rangle^{7/8}) K_{++}[g_1,g_2](t,\xi) \big|
  \lesssim \sum_{k\leq \min\{-7m/8, k_2-10\}} 2^{4k/5} 2^{k_2/2} {\|g_1(t)\|}_{L^2} {\|\widehat{P_{k_2}g_2}(t)\|}_{L^\infty}
\\
\lesssim \e 2^{C\e m} \cdot 2^{-m/4} \sum_{k_2} 2^{k_2} {\|\widehat{P_{k_2}f}(t)\|}_{L^\infty}
\lesssim 2^{-m/4} \cdot \e^2 2^{\delta'm} \lesssim \e^2 2^{-\gamma m},
\end{align*}
having used \eqref{fhatimpcor2}. This term can then be absorbed in $S$ in \eqref{pr1}.
The same bound shows that also $\tau_m^\prime K_{++}^m[g_1,g_2]$ can be absorbed into $R$ in \eqref{pr1}.

Similarly, we can use again \eqref{fhatimpcor2}, and in addition \eqref{dtf1}, to obtain
\begin{align*}
\big| |\xi|^{4/5}\varphi_{\leq 0}(\xi\langle t\rangle^{7/8})K_{++}[\partial_t g_1,g_2](t,\xi) \big|
  \lesssim \sum_{k\leq \min\{-7m/8, k_2-10\}} 2^{4k/5} 2^{k_2/2} {\|\partial_t g_1\|}_{L^2} {\|\widehat{P_{k_2}g_2}\|}_{L^\infty}
\\
\lesssim \e^2 2^{-(1-C\e m)} \cdot 2^{-m/4} \sum_{k_2} 2^{k_2} {\|\widehat{P_{k_2}g_2}(t)\|}_{L^\infty}
%\\ \lesssim 2^{-5m/4} \cdot \e^3 2^{(1/6+2\delta)m}
\lesssim \e^3 2^{-(1+\gamma)m}.
\end{align*}
A similar bound holds for $K_{++}^m[g_1,\partial_tg_2]$, so that these terms can be also absorbed into $R$ in \eqref{pr1}.

The same argument above also show the claim \eqref{LemL1.1} for the term $I_{--}^m(t,\xi)$,
since $\Phi_{--}(\xi,\eta) = |\xi|+|\eta|+|\xi-\eta|$ so that we can directly resort to integration by parts in the time variable as in \eqref{pr2}.
\end{proof}

\begin{proof}[Proof of \eqref{LemL1.2}]
We need to show that for all $t \approx 2^m$, $m\geq 0$, and $|\ell|\leq N_1$, we have
\begin{align}
\label{pr10}
\begin{split}
|\xi|^{4/5} \varphi_{\leq 0}(\xi\langle t\rangle^{7/8})\Big|\Omega^{\ell} \Big( I_{+-}(t,\xi) - I_{-+}(t,\xi)
  - \frac{1}{2}(2\pi)^{-3/2} \int_{\R^3} e^{it(|\xi| - \xi\cdot\eta|\eta|^{-1})} |\widehat{f}(t,\eta)|^2 \,\mathrm{d}\eta \Big)\Big|
%\\ & = R(t,\xi) + \partial_t S(t,\xi),
\\
\lesssim \e^2 2^{-(1+\gamma)m}.
%\lesssim |R(t,\xi)|\lesssim \e^2 2^{-(1+\gamma)m}, \qquad |S(t,\xi)|\lesssim \e^2 2^{-\gamma m}.
\end{split}
\end{align}

From the definition of $I_{\kappa_1\kappa_2}$ we can write
\begin{align}
\begin{split}
I_{+-}(t,\xi) & = I_{+-,\leq q_0}(t,\xi) + I_{+-,> q_0}(t,\xi), \qquad q_0:= -\frac{m}{2}
\\
I_{+-,\ast}(t,\xi) & := \frac{1}{4}(2\pi)^{-3/2} \int_{\R^3} e^{it(|\xi|-|\xi-\eta|+|\eta|)}\frac{|\xi-\eta|}{|\eta|}
\widehat{f}(t,\xi-\eta)\widehat{\overline{f}}(t,\eta) \, \varphi_{\ast}(\eta) \,\mathrm{d}\eta.
\end{split}
\end{align}
We can immediately verify that $I_{+-,\leq q_0}(t,\xi)$ is an acceptable remainder term by using \eqref{fhatimp}:
\begin{align}
\label{esti0}
\begin{split}
\big| |\xi|^{4/5}\varphi_{\leq 0}(\xi\langle t\rangle^{7/8}) \Omega^{\ell}I_{+-,\leq -q_0}(t,\xi) \big| \lesssim
  2^{-7m/10} \int_{\R^3} \big| |\xi-\eta| \widehat{g_1}(t,\xi-\eta)\big|
  \big| |\eta| \widehat{\overline{g_2}}(t,\eta) \big| \, \frac{\varphi_{\leq q_0}(\eta)}{|\eta|^2} \,\mathrm{d}\eta
  \\
\lesssim 2^{-7m/10} \cdot\big( \e 2^{\delta'm} \big)^2 \cdot 2^{q_0} \lesssim \e^2 2^{-21m/20},
\end{split}
\end{align}
where $g_{1,2}$ are defined in \eqref{defgr}.

Next, for $(\kappa_1\kappa_2) \in \{ (+-),(-+)\}$, let
\begin{align}
I_{\kappa_1\kappa_2,0}(t,\xi) &:= \frac{1}{4}(2\pi)^{-3/2} \int_{\R^3} e^{it(|\xi| + \kappa_1\xi\cdot\eta|\eta|^{-1})}
  \widehat{f_{\kappa_1}}(t,-\eta)\widehat{f_{\kappa_2}}(t,\eta)\,\mathrm{d}\eta.
\end{align}

Let us look at the case $(\kappa_1\kappa_2)=(+-)$ and estimate
\begin{align}
\label{pr11}
\begin{split}
& \big| I_{+-,>q_0}(t,\xi) - I_{+-,0}(t,\xi) \big| \lesssim |I_0|+|I_1| + |I_2| + |I_3|,
\\
I_0&:= \int_{\R^3} e^{it(|\xi| + \kappa_1\xi\cdot\eta|\eta|^{-1})}
  \widehat{f_{\kappa_1}}(t,-\eta)\widehat{f_{\kappa_2}}(t,\eta) \varphi_{\leq q_0}(\eta)\,\mathrm{d}\eta
\\
I_1 & :=  \int_{\R^3} \big[ e^{it(|\xi|-|\xi-\eta|+|\eta|)} - e^{it(|\xi| + \xi\cdot\eta|\eta|^{-1})}\big] \frac{|\xi-\eta|}{|\eta|}
\widehat{f}(t,\xi-\eta)\widehat{\overline{f}}(t,\eta) \varphi_{>q_0}(\eta)\,\mathrm{d}\eta
\\
I_2 & := \int_{\R^3} e^{it(|\xi| + \xi\cdot\eta|\eta|^{-1})} \Big[ \frac{|\xi-\eta|}{|\eta|} - 1\Big]
\widehat{f}(t,\xi-\eta)\widehat{\overline{f}}(t,\eta) \varphi_{>q_0}(\eta)\,\mathrm{d}\eta
\\
I_3 & := \int_{\R^3} e^{it(|\xi| + \xi\cdot\eta|\eta|^{-1})}
\big[ \widehat{f}(t,\xi-\eta) - \widehat{f}(t,-\eta) \big] \widehat{\overline{f}}(t,\eta) \varphi_{>q_0}(\eta)\,\mathrm{d}\eta.
\end{split}
\end{align}

Estimating in the same way as \eqref{esti0} we can prove $I_0\in\mathcal{R}$. Since for all $|\xi|\lesssim 2^{-7m/8} \ll |\eta|$ one has
\begin{align*}
\big| |\xi|-|\xi-\eta|+|\eta| - (|\xi| + \xi\cdot\eta|\eta|^{-1})\big| \lesssim |\xi|^2 |\eta|^{-1},
\end{align*}
we can estimate, for $t \approx 2^m$ and $|\ell|\leq N_1$,
\begin{align*}
|\Omega^{\ell}I_1| & \lesssim \int_{\R^3} \Big| e^{it(|\xi|-|\xi-\eta|+|\eta|)} - e^{it(|\xi| + \xi\cdot\eta|\eta|^{-1})}\Big|
  \big||\xi-\eta|\widehat{g_1}(t,\xi-\eta)\big| \big|\widehat{\overline{g_2}}(t,\eta)\big|
  \frac{\varphi_{>q_0}(\eta)}{|\eta|}\,\mathrm{d}\eta
\\
& \lesssim \int_{\R^3} t |\xi|^2 |\eta|^{-1}
  \big||\xi-\eta|\widehat{g_1}(t,\xi-\eta)\big| \big||\eta|\widehat{\overline{g_2}}(t,\eta)\big|
  \frac{\varphi_{>q_0}(\eta)}{|\eta|^2}\,\mathrm{d}\eta
\\
& \lesssim t |\xi|^2 \cdot {\| |\xi|\widehat{g_{1,2}}(t) \|}_{L^\infty}^2 \cdot m
\lesssim |\xi|^2 \cdot 2^{m} \cdot \big( \e 2^{\delta'm} \big)^2 m.
\end{align*}
Here $g_{1,2}$ are as in (\ref{defgr}), and we have used \eqref{fhatimpcor2} for the last inequality.
Upon multiplying this by $|\xi|^{4/5}$, with $|\xi| \lesssim 2^{-7m/8}$, we obtain the bound
\begin{align*}
|\xi|^{14/5} \cdot 2^{m} \cdot \e^2 2^{2\delta' m} \lesssim \e^2 2^{-21m/20}.
\end{align*}

The second term in \eqref{pr11} can be estimated as follows:
\begin{align*}
|\Omega^{\ell}I_2| \lesssim \int_{\R^3} \Big| \frac{|\xi-\eta|}{|\eta|} - 1\Big|
  \big| \widehat{g_1}(t,\xi-\eta) \big| \big|\widehat{\overline{g_2}}(t,\eta)\big| \varphi_{>q_0}(\eta)\,\mathrm{d}\eta
  \lesssim |\xi| 2^{-q_0} \cdot {\| \what{g_{1,2}} \|}_{L^2}^2
\\
\lesssim  2^{-7m/8} \cdot 2^{m/2} \cdot \e^2 2^{C\e m}
\lesssim 2^{-3m/8+C\e m} \e^2.
\end{align*}
Multiplying this bound by $|\xi|^{4/5}$, with $|\xi| \lesssim 2^{-7m/8}$, we see that this is also an acceptable remainder.

To estimate $I_3$ we decompose $g_1=\sum_{(j_1,k_1)}Q_{j_1k_1}g_1$, and write
\begin{align*}
\Omega^\ell I_3 & = \sum_{j_1+k_1\geq 0} \Omega^\ell I_{j_1,k_1}, \qquad \widetilde{g_1} := Q_{j_1k_1}g_1,
\\
\Omega^\ell I_{j_1,k_1} & := \int_{\R^3} e^{it(|\xi| + \xi\cdot\eta|\eta|^{-1})}
\Big[ \widehat{\widetilde{g_1}}(t,\xi-\eta) - \widehat{\widetilde{g_1}}(t,-\eta) \Big] \widehat{\overline{P_{k_2}g_2}}(t,\eta) \varphi_{>q_0}(\eta)\,\mathrm{d}\eta,
\end{align*}
having used that $|\xi-\eta|\approx |\eta|$ on the support of the integral (noticing also that $|k_1-k_2|\leq 5$, so we omit the summation in $k_2$).
%$$f_{\leq J}(x) := f(x) \varphi_{\leq J}(x), \qquad f_{> J}(x) := f(x) - f_{\leq J}(x),$$
Observe that, for $|\xi| \approx 2^{k}$, with $k \leq k_1-20$, in view of the estimate \eqref{aprioriLinfty2}, we have
\begin{align}
 \label{fdiff1}
\big| \widehat{\widetilde{g_1}}(t,\xi -\eta) - \widehat{\widetilde{g_1}}(t,-\eta) \big| \lesssim
  \Big| \int_0^1 \nabla \widehat{f_{j_1,k_1}}(t,-\eta+\zeta\xi) \cdot \xi \, \mathrm{d}\zeta \Big|
  \lesssim \e 2^{j_1/2 - 2k_1} 2^{C\e m} \cdot 2^k.
\end{align}
Also, in view of the apriori bound \eqref{aprioriE}, we have, for all $t \approx 2^m$,
\begin{align}
\label{fdiff2}
{\|\widetilde{g_1}(t)\|}_{L^2} \lesssim \e 2^{-j_1-k_1} 2^{C\e m}.
\end{align}

For large values of $j_1$ we use H\"{o}lder, \eqref{fdiff2} and \eqref{fhatimpcor2} to obtain
\begin{align*}
|I_{j_1,k_1}| \lesssim \int_{\R^3} \Big[ \widehat{\widetilde{g_1}}(t,\xi-\eta) | + |\widehat{\widetilde{g_1}}(t,-\eta)| \Big]
 |\widehat{\overline{P_{k_2}g_2}}(t,\eta) | \varphi_{>q_0}(\eta)\,\mathrm{d}\eta
\\
\lesssim \e 2^{-j_1-k_1} 2^{C\e m} \cdot 2^{3k_1/2} {\| \widehat{P_{k_2} g_2}\|}_{L^\infty}
\\
\lesssim \e^2 2^{-j_1} \cdot 2^{-k_1/2} 2^{\delta'm}.
\end{align*}
When $j_1 \geq 2m$ this bound clearly suffices. For $j_1 \leq 2m$ with $j_1 \geq -(k_1+2k)/3$ we obtain
\begin{align*}
2^{4k/5} |\Omega^\ell I_{j_1,k_1}| \lesssim \e^2 2^{22k/15} \cdot 2^{-k_1/6} 2^{\delta'm}
\lesssim \e^2 2^{-77m/60} \cdot 2^{(1/12 + \delta') m} \lesssim \e^2 2^{-(1+\g)m}
\end{align*}
having used $k\leq -7m/8$ and $k_1 +10 \geq q_0 = -m/2$. This gives an acceptable remainder.

In the remaining case when $j_1 \leq -(k_1+2k)/3$, we use \eqref{fdiff1} and again \eqref{fhatimpcor2} to estimate
\begin{align*}
|\Omega^{\ell}I_{j_1,k_1}| \lesssim \int_{\R^3} \big| \widehat{\widetilde{g_1}}(t,\xi-\eta) - \widehat{\widetilde{g_1}}(t,-\eta) \big|
 | \widehat{\overline{P_{k_2}g_2}}(t,\eta)| \mathrm{d}\eta
\\
\lesssim \e 2^{j_1/2-2k_1} 2^{C\e m} \cdot 2^k \cdot {\| \widehat{P_{k_1}f}(t) \|}_{L^\infty} \cdot 2^{3k_1}
\\
\lesssim \e^2 2^{2k/3-k_1/6} \cdot 2^{\delta'm}.
\end{align*}
Then we have
\begin{align*}
2^{4k/5} |\Omega^{\ell}I_{j_1,k_1}| \lesssim \e^2 2^{22k/15} \cdot 2^{-k_1/6} \cdot 2^{\delta'm} %\lesssim \e^2 2^{-(1+\g)m}
\end{align*}
which is the same acceptable bound obtained above.
%For $k \leq -7m/8$ and $k_1 + 10 \geq -m/2$

From the above inequalities, and \eqref{pr11}, together with the analogous estimates for $(\kappa_1\kappa_2) = (-+)$, it follows that
\begin{align*}
& |\xi|^{4/5} \varphi_{\leq 0}(\xi\langle t\rangle^{7/8})\big|\Omega^{\ell}\big( I^m_{\kappa_1\kappa_2,> q_0}(t,\xi) - \tau_m(t) I_{\kappa_1\kappa_2,0}(t,\xi) \big)\big|
%\\ & \lesssim |\xi|^{19/20}  \cdot \e^2 2^{C\e m} \big( |\xi|^2 \cdot 2^{7m/6} + |\xi|2^{m/3} + |\xi|^{1/2}2^{m/4} 2^{\delta_1m} \big)
%\\ &
\lesssim \e^2 2^{-(1+\gamma)m},
\end{align*}
when $t\approx 2^m$, $\gamma < 1/60$ and $|\ell|\leq N_1$.
%with $\e$ and $\delta_1$ small enough. %specify size of $\delta_1$ since it's linked to $N_1$.
This implies the desired bound \eqref{LemL1.2} since
\begin{align*}
I_{+-,0}(t,\xi) +  I_{-+,0}(t,\xi)
%& = \frac{1}{4}(2\pi)^{-3/2} \int_{\R^3} e^{it(|\xi| + \xi\cdot\eta|\eta|^{-1})}
%  \widehat{f}(t,-\eta)\widehat{\overline{f}}(t,\eta) \varphi_{>-3\delta m}(\eta)\,\mathrm{d}\eta
%  \\ & + \frac{1}{4}(2\pi)^{-3/2} \int_{\R^3} e^{it(|\xi| - \xi\cdot\eta|\eta|^{-1})}
%  \widehat{\overline{f}}(t,-\eta)\widehat{f}(t,\eta) \varphi_{>-3\delta m}(\eta)\,\mathrm{d}\eta \\
& = \frac{1}{2}(2\pi)^{-3/2} \int_{\R^3} e^{it(|\xi| - \xi\cdot\eta|\eta|^{-1})}
  |\widehat{f}(t,\eta)|^2\,\mathrm{d}\eta ,
\end{align*}
which is \eqref{LemL1.2'} once we write the integral in polar coordinates $\xi = \rho\theta$, $\eta=r\phi$. Finally \eqref{LemL1.3}
follows from integrating in $t$, and noticing that in the support of $\varphi_{\leq-10}(\xi\langle t\rangle^{7/8})$,
we have $\varphi_{\leq0}(\xi\langle s\rangle^{7/8})=1$ for all $s\leq t$.
%and $\widehat{f}(t,\eta) = \widehat{f}(t,) =: F_\phi(t,r)$. %Note: cutoff in $\eta$ can be removed
\end{proof}

%\comment{Redo this using the improved tighter bounds to have larger separation between $k_2$ and $-m$.
%
%Will improve the range in \eqref{Pr4} and is needed above \eqref{Pr12} and after \eqref{Pr20}}

\medskip
\subsection{Proof of Lemma \ref{LemHL1}}\label{PrLemHL1}

\begin{proof}[Proof of \eqref{LemHL1.1}]
Recall the definition of $I_{\kappa_1\kappa_2}^{m,0}$ from \eqref{defim01}. Moreover, we will define $I_{\kappa_1\kappa_2}^{m,0,k,k_1,k_2}$ by replacing the inputs $f_{\kappa_j}^*$ in (\ref{defim01}) by $P_{k_j}f_{\kappa_k}^{*}$ for $j=1,2$ (where $*$ can be $0$, $1$ or nothing), and multiplying by $\varphi_k(\xi)$. Note that, for example, in the first term of $I_{\kappa_1\kappa_2}^{m,0,k,k_1,k_2}$ in (\ref{defim01}) we must have $k_1\leq -7m/8+5$ and in the second term we have $k_1\geq -7m/8-5$ and $k_2\leq -7m/8+5$.

Moreover, in this setting we may modify the $g_{1,2}$ defined in \eqref{defgr}
by attaching in Fourier space cutoff factors like $\varphi_{\leq -10}(\xi\langle t\rangle^{7/8})$ or $\varphi_{> -10}(\xi\langle t\rangle^{7/8})$,
and note that they still satisfy the same estimates \eqref{estgr1}--\eqref{estgr3}.
By abusing notation, below we will still write them as $g_{1,2}$.

\medskip
{\it Case 1: $k_1\leq k_2+20$.}
We first look at the terms $I^{m,0,k,k_1,k_2}_{-+}$ in the case $k_1 = \min(k_1,k_2)$.
We have
\begin{align*}
& \big| |\xi|\varphi_{>0}(\xi\langle t\rangle^{7/8}) \Omega^\ell I^{m,k,k_1,k_2}_{-+}(t,\xi) \big|
\\
& \lesssim 2^k \varphi_{>0}(\xi\langle t\rangle^{7/8})\varphi_k(\xi) \cdot \tau_m(t) \cdot 2^{k_1-k_2}
\int_{\mathbb{R}^3} \big|\widehat{\overline{P_{k_1}g_1}}(t,\xi-\eta)\big| \, \big| \widehat{P_{k_2}g_2}(t,\eta) \big| \,\mathrm{d}\eta
\\
& \lesssim 2^{4k_1} {\|\what{P_{k_1}g_1}(t)\|}_{L^\infty} {\|\what{P_{k_2}g_2}(t)\|}_{L^\infty}
%\lesssim 2^{k-k_2} \cdot 2^{3k_1/2} \e 2^{(C\e+\delta_1)m} \cdot \e 2^{-N_+k_2}2^{C\e m}.
\end{align*}
which, when multiplied by $2^{15k_+}$ and summed over $k,k_1,k_2$ with $|k-k_2| \leq 30$ and $k_1 \leq -7m/8+30$, gives us
\begin{align*}
\sum_{k_1\leq -7m/8+30,\,|k-k_2|\leq 30} 2^{15k_+} \big| |\xi| \varphi_{>0}(\xi\langle t\rangle^{7/8}) I^{m,0,k,k_1,k_2}_{-+}(t,\xi) \big|
\\
%\lesssim \sum_{k_1\leq -(1-\delta)m -D, |k-k_2|\leq 10} 2^{2k_+} 2^{4k_1} {\|P_{k_1}f(t)\|}_{L^\infty} {\|P_{k_2}f(t)\|}_{L^\infty} \\
\lesssim \sum_{k_1\leq -7m/8+20} 2^{2k_1} \big(\e 2^{\delta'm} \big)^2 \lesssim \e^2 2^{-5m/4},
\end{align*}
having used \eqref{fhatimpcor2}.

\medskip
{\it Case 2: $k_1>k_2+20$.}
In this case $k_2 -10 \leq -7m/8 \leq k -15$, and we can integrate by parts in time using
\begin{align*}
\big| |\xi|+|\xi-\eta|-|\eta| \big| \approx |\xi|.
\end{align*}
We define
\begin{align*}
& K^{0,k,k_1,k_2}_{-+}[g_1,g_2](t,\xi)
\\
& = %\tau_m(t)
\int_{\mathbb{R}^3} e^{it(|\xi|+|\xi-\eta|-|\eta|)}\frac{\varphi_k(\xi)}{i(|\xi|+|\xi-\eta|-|\eta|)}\frac{|\xi-\eta|}{|\eta|}
\widehat{\overline{P_{k_1}g_1}}(t,\xi-\eta) \widehat{P_{k_2}g_2}(t,\eta)\,\mathrm{d}\eta
\end{align*}
and write, noticing that $I^{m,0,k,k_1,k_2}_{-+}[g_1,g_2]=\Omega^{\ell}I^{m,0,k,k_1,k_2}_{-+}$ (when omitting the input functions in these bilinear operators we always understand that the input functions are $f$ or $\overline{f}$),
\begin{align}
 \label{est0}
\begin{split}
I^{m,0,k,k_1,k_2}_{-+}[g_1,g_2](t,\xi) & = \tau_m(t)\frac{\partial}{\partial t} K^{0,k,k_1,k_2}_{-+}[g_1,g_2](t,\xi)
  %- \tau_m^\prime(t) 2^{-m} K^{m,k,k_1,k_2}_{-+}[g_1,g_2](t,\xi)
  \\ & - \tau_m(t)\Big[ K^{0,k,k_1,k_2}_{-+}[\partial_t g_1,g_2](t,\xi) + K^{0,k,k_1,k_2}_{-+}[g_1,\partial_t g_2](t,\xi) \Big].
\end{split}
\end{align}
To prove the desired bound observe that for all $|t| \approx 2^m$, and $k_2 \leq k- 10$,
\begin{align}
%\label{est1}
\big| |\xi| \varphi_{>0}(\xi\langle t\rangle^{7/8}) K^{0,k,k_1,k_2}_{-+}[g_1,g_2](t,\xi) \big| \lesssim
  \min\big( {\|\widehat{g_1}\|}_{L^\infty} {\|\widehat{g_2}\|}_{L^2}, {\|\widehat{g_1}\|}_{L^2}
  {\|\widehat{g_2}\|}_{L^\infty} \big) \cdot 2^{k_1} 2^{k_2/2}.
\end{align}
This estimate implies
\begin{align*}
& \sum_{k_1\in\Z, \,k_2\leq \min\{-7m/8+10, k-10\}} 2^{15k^+} \big| \varphi_{>0}(\xi\langle t\rangle^{7/8})|\xi| K^{k,k_1,k_2}_{-+}[g_1,g_2](t,\xi) \big|
  \\
& \lesssim \sum_{k_1\in\Z, \,k_2\leq -7m/8} 2^{10k^+} \cdot 2^k {\|\widehat{P_{k_1}g_1}\|}_{L^\infty} {\|\widehat{P_{k_2}g_2}\|}_{L^\infty }2^{2k_2}
\\
& \lesssim \sum_{k_1\in\Z, \,k_2\leq-7m/8} 2^{k_2} \cdot \e 2^{-5k_1^+} 2^{\delta'm} \cdot \e 2^{\delta'm}
  \lesssim \e^2 2^{-m/2}
\end{align*}
having used \eqref{fhatimp} and \eqref{fhatimpcor2}. This is consistent with our definition of remainder \eqref{defRem}.
Moreover, using again \eqref{fhatimpcor2} together with \eqref{dtf1} gives
\begin{align*}
& \sum_{k_1\in\Z, \,k_2\leq \min\{-7m/8+10, k-10\}}
  2^{15k^+} \big| \varphi_{\geq -7m/8}(\xi)|\xi| K^{k,k_1,k_2}_{-+}[g_1,\partial_t g_2](t,\xi) \big|
  \\
& \lesssim \sum_{k_1\in\Z, \,k_2\leq -7m/8+10} 2^k 2^{15k^+} {\|\widehat{P_{k_1}g_1}\|}_{L^\infty} {\|\widehat{P_{k_2}\partial_t g_2}\|}_{L^2}2^{k_2/2}
\\
& \lesssim \sum_{k_2\leq -7m/8+10} \e 2^{\delta'm} 2^{-5k_1^+} \cdot \e^2  2^{-(1-C\e)m} 2^{k_2/2} \lesssim \e^3 2^{-4m/3},
\end{align*}
which again suffices.
For the last term $K^{k,k_1,k_2}_{-+}[\partial_t g_1,g_2](t,\xi)$ we use the bound \eqref{dtf2} on $\partial_tg_1$, and \eqref{fhatimp} obtaining
\begin{align*}
& \sum_{k_1\in\Z, \,k_2\leq \min\{-7m/8+10, k-10\}}
  2^{15k^+} \big| \varphi_{\geq -7m/8}(\xi)|\xi| K^{k,k_1,k_2}_{-+}[\partial_t g_1,g_2](t,\xi) \big|
  \\
& \lesssim \sum_{k_1\in\Z, \,k_2\leq \min\{-7m/8+10, k-10\}} 2^k 2^{15k^+} {\|\widehat{P_{k_1}\partial_t g_1}\|}_{L^\infty}
  {\|\widehat{P_{k_2}g_2}\|}_{L^\infty} 2^{2k_2}
\\
& \lesssim \sum_{k_1\in\Z, \,k_2\leq -7m/8+10} \e^2 2^{-m+\delta'm+4\delta m} 2^{-5k_1^+} \cdot \e 2^{\delta'm} 2^{k_2}
  \lesssim \e^3 2^{-3m/2},
\end{align*}
which concludes the proof of \eqref{LemHL1.1}.
\end{proof}

\medskip
\begin{proof}[Proof of \eqref{LemHL1.2}]
Let us begin by observing that we can argue as in the beginning of the proof of \eqref{LemHL1.1} to obtain that
the contributions from $k_1 \leq k_2+20$ are acceptable remainders, that is, for $\kappa_2 \in \{+,-\}$,
\begin{align*}
& \varphi_{>0}(\xi\langle t\rangle^{7/8})I^{m,0}_{+\kappa_2}(t,\xi)
  - \varphi_{>0}(\xi\langle t\rangle^{7/8})\sum_{k_2 \leq k_1-20} I^{m,0,k,k_1,k_2}_{+\kappa_2}(t,\xi) \in \mathcal{R},
\end{align*}
so we can concentrate on the case $k_2 \leq \min(-7m/8+10,k_1-20)$. We let
\begin{align}
\Phi_{\kappa_2}(\xi,\eta) := \frac{\xi\cdot\eta}{|\xi|} - \kappa_2|\eta|,
\end{align}
recall the definition \eqref{LemL1.2'}, and write
\begin{align}
\label{HLpr10}
\begin{split}
\Omega^{\ell}\bigg(4(2\pi)^{3/2} & I^{m,0,k,k_1,k_2}_{+\kappa_2}(t,\eta) + \frac{\tau_m(t)}{2(2\pi)^{\frac{3}{2}}}
  \widehat{P_{k_1}f}(t,\xi) \varphi_k(\xi)|\xi|
  \\ & \times \int_{\R^3} e^{it\Phi_{\kappa_2}(\xi,\eta)} \varphi_{k_2}(\eta) \big(H_{\kappa_2\arg\eta}\big)_{\kappa_2}(t,|\eta|)\,\frac{\mathrm{d}\eta}{|\eta|}\bigg)
=  \sum_{q=1}^4 \varphi_k(\xi) K_\ell[P_{k_1}g_1, P_{k_2}g_2](t,\xi)
\end{split}
\end{align}
where $g_2=\Omega^{\ell_2}f_{\kappa_2}$, $|\ell_2|\leq N_1$, and
\begin{align}
\label{HLpr11}
K_1[g_1,g_2] &:= \int_{\mathbb{R}^3} \Big[ e^{it(|\xi|-|\xi-\eta|-\kappa_2|\eta|)}|\xi-\eta|- e^{it\Phi_{\kappa_2}(\xi,\eta)}|\xi| \Big]
  \widehat{g_1}(t,\xi-\eta) \widehat{g_2}(t,\eta)\, \frac{\mathrm{d}\eta}{|\eta|},
\\
\label{HLpr12}
K_2[g_1,g_2] &:= |\xi| \int_{\mathbb{R}^3} e^{it\Phi_{\kappa_2}(\xi,\eta)} \frac{1}{|\eta|}
  \big[ \widehat{g_1}(t,\xi-\eta) - \widehat{g_1}(t,\xi)\big] \widehat{g_2}(t,\eta)\, \mathrm{d}\eta, %\qquad j_0:= 19m/10,
\\
\label{HLpr13}
K_3[g_1,g_2] &:= |\xi|\widehat{g_1}(t,\xi) \int_{\mathbb{R}^3} e^{it\Phi_0(\xi,\eta)} \frac{1}{|\eta|}\Omega^{\ell_2}
  \Big[ \widehat{f_{\kappa_2}}(t,\eta) - \widehat{f_{\kappa_2}}(0,\eta)
  + \frac{\varphi_{k_2}(\eta)}{2(2\pi)^{3/2}} \big(H_{\kappa_2\arg\eta}\big)_{\kappa_2}(t,|\eta|) \Big]\, \mathrm{d}\eta,
\\
\label{HLpr14}
K_4[g_1,g_2] &:= |\xi|\widehat{g_1}(t,\xi) \int_{\mathbb{R}^3} e^{it\Phi_0(\xi,\eta)} \frac{1}{|\eta|}
  \widehat{g_2}(0,\eta)\, \mathrm{d}\eta.
\end{align}

We now show
\begin{align}
\label{HLpr20}
\sum_{k_2 \leq \min\{-7m/8+10,k_1-10\}} K_q[P_{k_1}g_1, P_{k_2}g_2](t,\xi) \in \mathcal{R}, \qquad q =1,\dots,4.
\end{align}
For the first term we can use  $| |\xi|-|\xi-\eta|-\kappa_2|\eta|-\Phi_0(\xi,\eta) | \lesssim |\eta|^2/|\xi|$ and estimate
\begin{align*}
\big| |\xi| \varphi_k(\xi) K_1[P_{k_1}g_1, P_{k_2}g_2](t,\xi) \big|
  %\lesssim \int_{\R^3} |\widehat{P_{k_1}f}(t,\xi-\eta)| |\widehat{P_{k_2}f}(t,\eta)| \, \mathrm{d}\eta
  \lesssim 2^k \cdot {\| \widehat{P_{k_1}f} \|}_{L^\infty} {\| \widehat{P_{k_2}f} \|}_{L^\infty}
  \cdot 2^{2k_2}\big( 2^m 2^{2k_2} + 2^{k_2} \big)
%\\ \lesssim 2^{-k_1} 2^{-10k_1^+}2^{(C\e+\delta_1) m}\e \cdot   2^{{C\e+\delta_1} m} \e \cdot 2^{2k_2}(2^m 2^{k_2-k} + 1).
\end{align*}
so that summing over $k_1,k_2$ in the current frequency configuration, and using \eqref{fhatimp}-\eqref{fhatimpcor2}, we get a bound of
\begin{align*}
\e 2^{\delta'm} \cdot \e 2^{\delta'm} \cdot 2^{3k_2} 2^m \lesssim \e^2 2^{-3m/2}.
\end{align*}

To estimate \eqref{HLpr12} we decompose $g_1$ into $Q_{j_1k_1}g_1$ as before. Let $\widetilde{g_1}=Q_{j_1k_1}g_1$, $(j_1,k_1)\in\mathcal{J}$. %and use the inequality \eqref{fdiff1} and an argument similar to the one that follows it.
For all $k_2 \leq k_1-10$ we can use \eqref{aprioriLinfty} and \eqref{aprioriLinfty3} to estimate
\begin{align*}
\big| |\xi| \varphi_k(\xi) K_2[\widetilde{g_1}, P_{k_2}g_{2}](t,\xi) \big|
& \lesssim 2^k 2^{k_1-k_2} \cdot {\| \what{\widetilde{g_1}} \|}_{L^\infty} \cdot {\| \what{P_{k_2}g_2} \|}_{L^\infty} 2^{3k_2}
\\
& \lesssim 2^{k_2} \cdot 2^{-j_1/2} 2^{-15k_1^+ }2^{C\e m} \e \cdot 2^{\delta'm} \e.
\end{align*}
We can then sum this as in \eqref{HLpr20} when $j_1 \geq 2m/3$, %and $k_2 \leq \min\{-7m/8,k_1\}-10$ as in \eqref{HLpr20},
to obtain a bound of $\e^2 2^{-10m/9}2^{-15k^+}$, which is an acceptable contribution.
When instead $j_1 \leq 2m/3$ we can use again \eqref{aprioriLinfty}, see also \eqref{fdiff1}, and \eqref{fhatimp} to estimate
\begin{align*}
\big| |\xi| \varphi_k(\xi) K_2[\widetilde{g_1}, P_{k_2}g_{2}](t,\xi) \big|
  \lesssim \varphi_k(\xi) 2^k 2^{k_1-k_2}\int_{\mathbb{R}^3}
  \big|\widehat{\widetilde{g_1}}(t,\xi-\eta) - \widehat{\widetilde{g_1}}(t,\xi)\big| |\widehat{P_{k_2}g_{2}}(t,\eta)|
  \,\mathrm{d}\eta
\\
\lesssim 2^{2k-k_2} \cdot 2^{k_2} 2^{j_1/2 - 2k_1} \e 2^{C\e m} \cdot \e 2^{\delta'm} \cdot 2^{2k_2}
\\
\lesssim \e^2 2^{(\delta'+C\e)m} 2^{j_1/2} 2^{2k_2}.
\end{align*}
Summing this bound over $j_1\leq 2m/3$ and $k_1,k_2$ in the current frequency configuration, we conclude that \eqref{HLpr20} holds for $K_2$.

For the third term we use \eqref{LemL1.3} to deduce, for $k_2 \leq k_1-10$,
\begin{align*}
& \big| |\xi|\varphi_k(\xi) K_3[P_{k_1}g_1, P_{k_2}g_2](t,\xi) \big|
\\
& \lesssim 2^k 2^{k_1}{\|\widehat{P_{k_1}g_1}\|}_{L^\infty} \cdot 2^{2k_2}
  \sup_{|\eta|\approx 2^{k_2}}\sup_{|\ell|\leq N_1} \Big| \Omega^{\ell}\bigg(\widehat{f}(t,\eta) - \widehat{f}(0,\eta)
  + \frac{\varphi_{k_2}(\eta)}{2(2\pi)^{3/2}} \big(H_{\kappa_2\arg\eta}\big)_{\kappa_2}(t,|\eta|)\bigg) \Big|
\\
& \lesssim \e 2^{\delta'm}2^{-15k_1^+} \cdot \e^2 2^{6k_2/5}.
\end{align*}
Summing over $k_2 \leq \min\{-7m/8+10,k_1-10\}$ gives a bound of $\e^2 2^{-(1+\gamma)m} 2^{-15k^+}$ for $\gamma < 1/40$. Finally the estimate for \eqref{HLpr14} follows directly from the assumption on initial data.

Putting together \eqref{HLpr20} and \eqref{HLpr10} we have obtained that for all $k \geq -7m/8$,
\begin{align*}
 \begin{split}
& I^{m,0,k,k_1,k_2}_{++}(t,\eta) - I^{m,0,k,k_1,k_2}_{+-}(t,\eta)
  + \frac{\tau_m(t)}{8(2\pi)^3} \widehat{P_{k_1}f}(t,\xi) \varphi_k(\xi)|\xi|
\\
& \times \Big[
\int_{\R^3} e^{it(\frac{\xi\cdot\eta}{|\xi|} - |\eta|)} \varphi_{k_2}(\eta) H_{\arg\eta}(t,|\eta|)\,\frac{\mathrm{d}\eta}{|\eta|}
- \int_{\R^3} e^{it(\frac{\xi\cdot\eta}{|\xi|} +|\eta|)} \varphi_{k_2}(\eta)
  \overline{H_{-\arg\eta}}(t,|\eta|)\,\frac{\mathrm{d}\eta}{|\eta|} \Big] \in \mathcal{R}.
\end{split}
\end{align*}
By writing $\xi=\rho\theta$ and $\eta = r\phi$, summing over $(k_1,k_2)$, and making the change of variables $\eta\mapsto-\eta$ in the second integral above, we arrive at the desired conclusion \eqref{LemHL1.2}.
\end{proof}

\medskip
\subsection{Proof of Lemma \ref{LemHH1}}\label{PrLemHH1}
As seen from (\ref{defim01}), the term $I_{\kappa_1\kappa_2}^{m,1}$ involves only $f_{\kappa}^1$
which is the not-low frequency part as in \eqref{fhigh}. Define $g_1^1$ and $g_2^1$ in the same way as in \eqref{defgr}.
For simplicity we will still use $g_{1,2}$ to denote $g_{1,2}^1$.
%\comment{Deng: should put all these notations before the proofs}
Define $I_{\kappa_1\kappa_2}^{m,1,k,k_1,k_2}$ by
\[I_{\kappa_1\kappa_2}^{m,1,k,k_1,k_2}:=\varphi_k(\xi)\cdot I_{\kappa_1\kappa_2}^{m}[P_{k_1}f_{\kappa_1}^1,P_{k_2}f_{\kappa_2}^1],\] in the same way as the proof of \eqref{LemHL1.1} above, then we have
\begin{align}
\label{Pr1}
\begin{split}
\varphi_k(\xi)I^{m,1}_{\kappa_1\kappa_2}(t,\xi) & = \varphi_k(\xi)I^{m,1}_{\kappa_1\kappa_2}(t,\xi) = \sum_{\min(k_1,k_2) > -7m/8-10}I_{\kappa_1\kappa_2}^{m,1,k,k_1,k_2}[f^1,f^1](t,\xi),
\\
I^{m,1,k,k_1,k_2}_{\kappa_1\kappa_2}(t,\xi)[g_1,g_2] &:= \frac{\varphi_k(\xi)}{4(2\pi)^{3/2}}
\int_{\mathbb{R}^3} e^{it(|\xi|-\kappa_1|\xi-\eta|-\kappa_2|\eta|)}\frac{|\xi-\eta|}{|\eta|}
  \widehat{P_{k_1}g_1}(t,\xi-\eta) \widehat{P_{k_2}g_{2}}(t,\eta)\,\mathrm{d}\eta\,
\end{split}
\end{align}
For lighter notation, we will often denote the term $I^{m,k,k_1,k_2}_{\kappa_1\kappa_2}(t,\xi)[f,g]$ just by $I^{m,k}_{\kappa_1\kappa_2}[f,g]$,
and sometimes omit the dependence on $(\kappa_1\kappa_2)$ when this causes no confusion.
%with the understanding that

The main idea for the proof of this Lemma is to restrict the interactions to parallel ones using
integration by parts in the angular directions via Lemma \ref{LemAIBP}.
However, before being able to extract the main contribution from the nonlinear terms, we need several reductions.
We subdivide as usual our inputs according to their frequency and spatial localization, that is $g_r= \sum Q_{j_rk_r}g_r$, and analyze various cases.

\bigskip
{\bf Step 1}: {\it $\min(j_1,j_2) \geq (1-4\delta)m$ or $\max(j_1,j_2) \geq 3m$}.
We first treat inputs with large spatial localization, and show that these only contribute to the remainder terms.
We look at three cases depending on the size of $k_2$.
First, using the a priori weighted bound \eqref{aprioriE} we see that
\begin{align*}
\begin{split}
2^k\big| \Omega^{\ell}I_{\kappa_1\kappa_2}^{m,k}[Q_{j_1,k_1}f^1, Q_{j_2,k_2}f^1](\xi) \big|
  \lesssim 2^{k_1-k_2} {\|\widehat{Q_{j_1,k_1}g_1}\|}_{L^2} \, {\|\widehat{Q_{j_2,k_2}g_2}\|}_{L^2}
  \\
\lesssim 2^{k_1-k_2} \cdot \e 2^{-15k_1^+}2^{-j_1-k_1}2^{C\e m} \cdot \e 2^{-15k_1^+}2^{-j_2-k_2}2^{C\e m}
\\
\lesssim  \e^2 2^{-15k^+} 2^{-j_1-j_2} 2^{C\e m} \cdot 2^{-2k_2}.
\end{split}
\end{align*}
This suffices to obtain an estimate compatible with a remainder of the type \eqref{defRem} provided that $k_2\geq -m/2 + 2\gamma m$,
noticing that $\gamma=15\delta$, see \eqref{param}.

In the case $k_2 \leq -m/2 - 2\gamma m$, we estimate using the bounds \eqref{aprioriLinfty} and \eqref{fhatimp}:
\begin{align*}
\begin{split}
2^k \big| \Omega^\ell I_{\kappa_1\kappa_2}^{m,k}[Q_{j_1,k_2}f^1, Q_{j_2,k_2}f^1](\xi) \big|
  \lesssim 2^k 2^{k_1-k_2} {\|\widehat{Q_{j_1,k_1}g_1}\|}_{L^\infty} \, {\|\widehat{Q_{j_2,k_2}g_2}\|}_{L^\infty} \cdot 2^{3\min(k_1,k_2)}
  \\
\lesssim 2^k 2^{k_1-k_2} \cdot \e 2^{-15k_1^+}2^{-j_1/2-2k_1}2^{C\e m} \cdot \e 2^{-k_2} 2^{\delta' m} \cdot 2^{3\min(k_1,k_2)}
\\
\lesssim  \e^2 2^{-15k^+}2^{(\delta'+C\e)m} \cdot 2^k \cdot 2^{-j_1/2-k_1}\cdot 2^{\min(k_1,k_2)}  .
\end{split}
\end{align*}
This suffices since $j_1 \geq (1-4\delta)m$ and $k-k_1+\min(k_1,k_2)\leq -m/2 - 2\gamma m + 10$, and $\gamma=15\delta$.

We are then left with the case $ -2\gamma m \leq k_2 + m/2 \leq 2\gamma m$,
which we treat using \eqref{aprioriLinfty} similarly to the case above, this time together with the additional improved bound \eqref{fhatimp'}:
\begin{align*}
\begin{split}
2^k \big| I_{\kappa_1\kappa_2}^{m,k}[Q_{j_1,k_2}f^1, Q_{j_2,k_2}f^1](\xi) \big|
  \lesssim 2^k 2^{k_1-k_2} {\|\widehat{Q_{j_1,k_1}g_1}\|}_{L^\infty} \, {\|\widehat{Q_{j_2,k_2}g_2}\|}_{L^\infty} \cdot 2^{3\min(k_1,k_2)}
  \\
\lesssim 2^k 2^{k_1-k_2} \cdot \e 2^{-15k_1^+}2^{-j_1/2-2k_1}2^{C\e m} \cdot \e 2^{m/2} 2^{-m/16} \cdot 2^{3\min(k_1,k_2)}
\\
\lesssim \e^2 2^{-15k^+}2^{(3\delta+C\e) m} \cdot 2^{-m/16} \cdot 2^{k-k_1} \cdot 2^{2\min(k_1,k_2)},
\end{split}
\end{align*}
where we used $j_1 \geq (1-4\delta)m$ in the last inequality.
Since $k_2 \leq -m/2+2\gamma m$ this gives us a bound of $\e^2 2^{-15k^+} 2^{(3\delta+C\e) m} 2^{(-1-1/16+4\gamma)m}$
which is sufficient with $\gamma=1/90$, see \eqref{param}. %See \eqref{defRem} and \eqref{aprioriLinfty} and definition of $N_1$.

\bigskip
{\bf Step 2}: {\it Restriction to parallel interactions}. %Case 2: $\min(j_1,j_2) \leq (1-3\delta)m$, $\max(j_1,j_2) \geq 3m$}
In view of the previous step, we may assume that $\min(j_1,j_2) \leq (1-4\delta)m$ and $\max(j_1,j_2) \leq 3m$.
Moreover, by definition of \eqref{Pr1} we already have $\min(k,k_1,k_2) \geq -7m/8-D$,
so that we are within the hypotheses of Lemma \ref{LemAIBP}.
Using this Lemma reduces matters to estimating the terms
\begin{align}
\label{Pr2}
\begin{split}
4(2\pi)^{3/2} I^{m,k,k_1,k_2}_{\kappa_1\kappa_2,p_0}(t,\xi)
:= & \varphi_k(\xi) \int_{\R^3} e^{it(|\xi|-\kappa_1|\xi-\eta|-\kappa_2|\eta|)}\frac{|\xi-\eta|}{|\eta|}
  \\ & \times \chi\big(|\xi\wedge \eta| 2^{-p_0-k-\min(k_1,k_2)}\big)
  \widehat{f_{\kappa_1}^1}(t,\xi-\eta) \widehat{f_{\kappa_2}^1}(t,\eta)\,\mathrm{d}\eta,
\end{split}
\end{align}
with
\begin{align}
\label{Pr2.5}
p_0 := & -\frac{m}{2} - \frac{\min(k,k_1,k_2)}{2} + 2\delta m,
\end{align}
where we denoted as usual $f_r := Q_{j_r,k_r} f$, $r=1,2$, and
\begin{align}
\label{Pr4}
\begin{split}
k,k_1,k_2 \geq -7m/8-10, \qquad \min(j_1,j_2) \leq (1-2\delta)m.
\end{split}
\end{align}

We next treat the case when $\max(j_1,j_2)\geq (1-2\delta)m$.
Using that $\eta$ is in a solid cone of approximate aperture $2^{p_0}$ and height $2^{\min(k_1,k_2)}$, we get that
\begin{equation*}
|\Omega^{\ell}I_{\kappa_1\kappa_2}|\lesssim 2^{k_1-k_2}2^{2p_0+3\min(k_1,k_2)}
  \cdot\|\widehat{Q_{j_1k_1}g_1}\|_{L^\infty}\cdot \|\widehat{Q_{j_2k_2}g_2}\|_{L^\infty}.
\end{equation*}
We may assume $k_2\leq k-10$, other cases being similar and simpler; then we have
\begin{equation*}
|\Omega^{\ell}I_{\kappa_1\kappa_2}| \lesssim 2^{-(m-4\delta m)} 2^{-k} 2^{2k_1+k_2}
  \cdot \|\widehat{Q_{j_1k_1}g_1}\|_{L^\infty}\cdot \|\widehat{Q_{j_2k_2}g_2}\|_{L^\infty}.
\end{equation*}
If $j_1\geq (1-2\delta)m$, using first \eqref{aprioriLinfty} and then \eqref{fhatimpcor2}, we have
\begin{equation*}
2^{m+k}|\Omega^{\ell}I_{\kappa_1\kappa_2}|\lesssim 2^{2k_1} 2^{4\delta m} \cdot \e 2^{-j_1/2-2k_1} 2^{C\e m} \cdot 2^{k_2}
  \|\widehat{Q_{j_2k_2}g_2}\|_{L^\infty} \lesssim \e^2 2^{(-1/2+5\delta+\delta')m},
\end{equation*}
which suffices;
if $j_2\geq (1-2\delta)m$ and $k_2\geq (-1/2+2\gamma)m$, again by \eqref{aprioriLinfty} and \eqref{fhatimpcor2}, we have
\begin{equation*}
2^{m+k}|\Omega^{\ell}I_{\kappa_1\kappa_2}| \lesssim \e^2 2^{(4\delta + C\e) m} \cdot 2^{k_2} 2^{-j_2/2-2k_2}
  \lesssim \e^2 2^{-k_2-m/2} 2^{6\delta m} \lesssim \e^2 2^{-\gamma m},
\end{equation*}
which again suffices.
Finally, if $j_2\geq (1-2\delta)m$ and $-7m/8\leq k_2\leq (-1/2+2\gamma)m$, then using Lemma \ref{Lemspe}, we get
\[2^{m+k}|\Omega^{\ell}I_{\kappa_1\kappa_2}|\lesssim \e^2 2^{C\e m}2^{k_2} 2^{-k_2-m/16},
\]
which suffices, and completes the proof in the case $\max(j_1,j_2)\geq (1-2\delta)m$.

%Note that the angular cutoff in \eqref{Pr2} restricts to $\angle(\xi,\eta) \approx 0$ or $\pi$.

Now we will assume $\max(j_1,j_2)\leq (1-2\delta)m$. In what follows we show that the expressions in \eqref{Pr2}, under the restrictions \eqref{Pr2.5}-\eqref{Pr4},
are well approximated by the sole contributions from frequencies $\eta$ which are parallel to $\xi$
(whether in the same direction or the opposite depends on the signs $\kappa_1\kappa_2$).
%and in a proper size relation to $\xi$

Let us introduce the following nonlinear terms
\begin{align}
\label{J}
\begin{split}
J_{++} & := \int_{\R^3} e^{it(|\xi|-|\xi-\eta|-|\eta|)}\frac{|\xi-\eta|}{|\eta|} \, \varphi_{\leq 5}(|\eta|/|\xi|) \,
  \chi\big(\angle(\xi,\eta) 2^{-p_0+(k_2-k_1)^+}\big) \what{f^1}(t,\xi-\eta) \what{f^1}(t,\eta)\,\mathrm{d}\eta,
\\
J_{+-} & := \int_{\R^3} e^{it(|\xi|-|\xi-\eta|+|\eta|)}\frac{|\xi-\eta|}{|\eta|}
  \chi\big( [\angle(\xi,\eta) - \pi]2^{-p_0+(k_2-k_1)^+}\big) \what{f^1}(t,\xi-\eta) \what{\overline{f^1}}(t,\eta)\,\mathrm{d}\eta,
\\
J_{-+} & := \int_{\R^3} e^{it(|\xi|+|\xi-\eta|-|\eta|)}\frac{|\xi-\eta|}{|\eta|} \, \varphi_{\leq 5}(|\xi|/|\eta|) \,
  \chi\big(\angle(\xi,\eta) 2^{-p_0+(k_2-k_1)^+})\big) \what{\overline{f^1}}(t,\xi-\eta) \what{f^1}(t,\eta)\,\mathrm{d}\eta,
\end{split}
\end{align}
where, for lighter notation, we have omitted the dependence on $m,k,k_1,k_2,p_0$.
Let us also define
\begin{align}
\label{Jpar}
\begin{split}
J_{++}^{\parallel} & := \int_{\R^3} e^{it(|\xi|-|\xi-\eta|-|\eta|)}\frac{||\xi|-|\eta||}{|\eta|}
  \, \varphi_{\leq 5}(|\eta|/|\xi|) \,
  \chi\big(\angle(\xi,\eta) 2^{-p_0+(k_2-k_1)^+})\big)
  \\ & \hskip300pt \times \what{f^1}\big(t,\xi-\xi\tfrac{|\eta|}{|\xi|}\big) \what{f^1} \big(t,\xi\tfrac{|\eta|}{|\xi|}\big)\,\mathrm{d}\eta,
\\
J_{+-}^{\parallel} & := \int_{\R^3} e^{it(|\xi|-|\xi-\eta|+|\eta|)} \frac{|\xi|+|\eta|}{|\eta|}
  \chi\big( [\angle(\xi,\eta) - \pi]2^{-p_0+(k_2-k_1)^+}\big)
  \what{f^1}\big(t,\xi+\xi\tfrac{|\eta|}{|\xi|}\big) \what{\overline{f^1}}\big(t,-\xi\tfrac{|\eta|}{|\xi|}\big)\,\mathrm{d}\eta,
\\
J_{-+}^{\parallel} & := \int_{\R^3} e^{it(|\xi|+|\xi-\eta|-|\eta|)} \frac{||\eta| - |\xi||}{|\eta|}
  \, \varphi_{\leq 5}(|\xi|/|\eta|) \,\chi\big(\angle(\xi,\eta) 2^{-p_0+(k_2-k_1)^+})\big)
  \\ & \hskip300pt \times \what{\overline{f^1}}\big(t,\xi-\xi\tfrac{|\eta|}{|\xi|}\big) \what{f^1}\big(t,\xi\tfrac{|\eta|}{|\xi|}\big)\,\mathrm{d}\eta.
\end{split}
\end{align}

Next we show how the terms $I_{\kappa_1\kappa_2} := 4(2\pi)^{3/2} I^{m,k,k_1,k_2}_{\kappa_1\kappa_2,p_0}$ in \eqref{Pr2},
are well approximated by the terms $J_{\kappa_1\kappa_2}^{\parallel}$ by first approximating them by $J_{\kappa_1\kappa_2}$.

%\comment{Could fix $\xi = |\xi|e_1 = re_1$, and do $\mathrm{d}\eta$ integral in polar coordinates all throughout}

\medskip
\noindent
{\it Step 2.1: Proof that $I_{\kappa_1\kappa_2} - \tau_m(t)J_{\kappa_1\kappa_2} \in \mathcal{R}$}.
Let us look at the case $\kappa_1\kappa_2 = ++$ and write
\begin{align}
\label{I-J++}
\begin{split}
& I_{++} - \tau_m(t)J_{++} = \tau_m(t)(A + B)
\\
& A := \int_{\R^3} e^{it(|\xi|-|\xi-\eta|-|\eta|)}\frac{|\xi-\eta|}{|\eta|}
  \chi\big([\angle(\xi,\eta) -\pi] 2^{-p_0+(k_2-k_1)^+}\big) \what{f^1}(t,\xi-\eta) \what{f^1}(t,\eta)\,\mathrm{d}\eta,
\\
& B :=  \int_{\R^3} e^{it(|\xi|-|\xi-\eta|-|\eta|)}\frac{|\xi-\eta|}{|\eta|} \, \varphi_{\geq 4}(|\eta|/|\xi|) \,
  \chi\big(\angle(\xi,\eta) 2^{-p_0+(k_2-k_1)^+}\big) \what{f^1}(t,\xi-\eta) \what{f^1}(t,\eta)\,\mathrm{d}\eta.
\end{split}
\end{align}
The expressions for $\Omega^\ell A$ and $\Omega^\ell B$ are similar, just with $f^1$ replaced by $g_1$ and $g_2$. Both of these terms can be treated by integration by parts in $\eta$ many times using Lemma \ref{lemIBP0},
since on the support of both $A$ and $B$ we have
%\begin{align}
%\big| |\xi|-|\xi-\eta|-|\eta| \big| \gtrsim |\eta|,
%\end{align}
\begin{align*}
\big| \nabla_\eta(|\xi|-|\xi-\eta|-|\eta|) \big| = \Big| \frac{\eta-\xi}{|\eta-\xi|}+\frac{\eta}{|\eta|} \big| \gtrsim 1,
\end{align*}
and $\max(j_1,j_2) \leq (1-2\delta)m$, $\min(k_1,k_2) \geq -(1-2\delta)m$, and moreover
\[\big|\nabla_\eta^\alpha\chi\big(\angle(\xi,\eta)2^{-p_0+(k_2-k_1)^+}\big)\big|\lesssim 2^{|\alpha|(-p_0-\min(k_1,k_2))},\quad -p_0-\min(k_1,k_2)\leq (15/16+2\delta)m,
\] using Fa\`{a} di Bruno's formula and the observation that $\angle(\xi,\eta)$ is homogeneous of degree $0$ and smooth on $\mathbb{S}^2$.

%\comment{Put more details here?}

In the $+-$ case we decompose depending on the size of $\eta$ relative to $\xi$:
\begin{align}
\label{I-J+-}
\begin{split}
& I_{+-} - \tau_m(t)J_{+-} = \tau_m(t)(C+D),
\\
& C := \int_{\R^3} e^{it(|\xi|-|\xi-\eta|+|\eta|)}\frac{|\xi-\eta|}{|\eta|} \, \varphi_{\leq -5}(|\eta|/|\xi|) \,
  \chi\big(\angle(\xi,\eta) 2^{-p_0}\big) \what{f^1}(t,\xi-\eta) \what{\overline{f^1}}(t,\eta)\,\mathrm{d}\eta,
\\
& D := \int_{\R^3} e^{it(|\xi|-|\xi-\eta|+|\eta|)}\frac{|\xi-\eta|}{|\eta|} \, \varphi_{\geq -4}(|\eta|/|\xi|) \,
  \chi\big(\angle(\xi,\eta) 2^{-p_0}\big) \what{f^1}(t,\xi-\eta) \what{\overline{f^1}}(t,\eta)\,\mathrm{d}\eta.
\end{split}
\end{align}
On the support of $C$, since $\xi$ and $\eta$ are almost parallel and $2|\eta| \leq |\xi|$, we have
\begin{align}
\big| \nabla_\eta(|\xi|-|\xi-\eta|+|\eta|) \big| = \Big| \frac{\eta-\xi}{|\eta-\xi|}-\frac{\eta}{|\eta|} \big| \gtrsim 1,
\end{align}
We can then apply again Lemma \ref{lemIBP0} as done for the terms $A$ and $B$ above to deduce that $C$ is an acceptable remainder.

On the support of $D$ instead, the gradient of the phase is not lower bounded, but the phase itself satisfies a good lower bound:
\begin{align*}
\big| |\xi|-|\xi-\eta|-|\eta| \big| \gtrsim |\xi|.
\end{align*}
We can then use this to integrate by parts in time. More precisely we write
\begin{equation}
\label{newint}
\tau_m(t)\cdot\Omega^\ell D = \partial_t(\tau_m(t)L[g_1,g_2])- \tau_m^\prime(t)L[g_1,g_2]-\tau_m(t) L[g_1,\partial_tg_2]-\tau_m(t)L[\partial_tg_1,g_2],
\end{equation}
where
\begin{multline*}
L[g_1,g_2] := \int_{\R^3} e^{it(|\xi|-|\xi-\eta|+|\eta|)}\frac{|\xi-\eta|}{|\eta|((|\xi|-|\xi-\eta|+|\eta|)} \, \varphi_{\geq -4}(|\eta|/|\xi|)
 \\
\times \chi\big(2^{-(p_0+k+\min(k_1,k_2))}\xi\wedge\eta\big) \what{g_1}(t,\xi-\eta) \what{\overline{g_2}}(t,\eta)\,\mathrm{d}\eta.
\end{multline*}
We then have
\begin{multline*}
\big||\xi|\cdot L[g_1,g_2](t,\xi)\big|
  \lesssim 2^{k+k_1-k_2-\max(k_1,k_2)}\cdot 2^{2p_0+3\min(k_1,k_2)} \|\widehat{P_{k_1}g_1}\|_{L^\infty} \|\widehat{P_{k_2}g_2}\|_{L^\infty}
\\
\lesssim 2^{k+k_1-k_2-\max(k_1,k_2)} 2^{-m+4\delta m-\min(k,k_1,k_2)} \cdot
  \e 2^{\delta'm-k_1} \cdot \e 2^{\delta'm-k_2}\lesssim \e^2 2^{-(2/3-2\delta')m},
\end{multline*}
which settles the first two terms in \eqref{newint}; the last two terms are treated similarly, using
\begin{multline*}
\big||\xi|\cdot L[\partial_tg_1,g_2](t,\xi)\big|
\lesssim 2^{k+k_1-k_2-\max(k_1,k_2)}\cdot 2^{2p_0+3\min(k_1,k_2)} \|\widehat{\partial_tP_{k_1}g_1}\|_{L^\infty} \|\widehat{P_{k_2}g_2}\|_{L^\infty}
\\
\lesssim 2^{k+k_1-k_2-\max(k_1,k_2)}2^{-m-\min(k,k_1,k_2)}\cdot\e 2^{-m+\delta'm+4\delta m-k_1} \cdot \e 2^{\delta'm-k_2}
\lesssim \e^2 2^{-(3/2-2\delta')m},
\end{multline*}
and similarly for $L[g_1,\partial_tg_2]$. The estimate of the remaining term $ I_{-+} - J_{-+}$ is similar and we omit the details.

\medskip
\noindent
{\it Step 2.2: Proof that $J_{\kappa_1\kappa_2} - J^\parallel_{\kappa_1\kappa_2} \in \mathcal{R}$.}
We look at the case with $(\kappa_1\kappa_2) = (++)$, the other cases being similar. Write
\begin{align}
\label{Pr10}
\begin{split}
& \big| \Omega^{\ell}(J_{++} - J_{++})^\parallel \big| \lesssim L_1 + L_2 + L_3 + L_4,
\\
& L_1 := \int_{|\eta|\leq 100|\xi|} \frac{|\xi-\eta|}{|\eta|}
  \chi\big(\angle(\xi,\eta) 2^{-p_0+(k_2-k_1)^+}\big) \big| \what{g_1}(t,\xi-\eta) - \what{g_2}\big(t,|\xi-\eta|\tfrac{\nu\xi}{|\xi|}\big) \big| \,
  \big| \what{f_2}(t,\eta) \big| \,\mathrm{d}\eta,
\\
& L_2 := \int_{|\eta|\leq 100|\xi|} \frac{|\xi-\eta|}{|\eta|}
  \chi\big(\angle(\xi,\eta) 2^{-p_0+(k_2-k_1)^+}\big) \big| \what{g_1}\big(t,|\xi-\eta|\tfrac{\nu\xi}{|\xi|}\big)
  -\what{g_1}\big(t,\xi-\xi\tfrac{|\eta|}{|\xi|}\big) \big| \, \big| \what{g_2}(t,\eta) \big| \,\mathrm{d}\eta,
\\
& L_3 := \int_{|\eta|\leq 100|\xi|} \frac{|\xi-\eta|}{|\eta|} \chi\big(\angle(\xi,\eta) 2^{-p_0+(k_2-k_1)^+}\big) \big| \what{g_1}(t,\xi-\eta) \big| \,
  \big| \what{g_2}(t,\eta) - g_2\big(t,\xi\tfrac{|\eta|}{|\xi|}\big) \big| \,\mathrm{d}\eta,
\\
& L_4 := \int_{|\eta|\leq 100|\xi|} \Big| \frac{|\xi-\eta|}{|\eta|} - \frac{||\xi|-|\eta||}{|\eta|} \Big|
  \chi\big(\angle(\xi,\eta) 2^{-p_0+(k_2-k_1)^+})\big) \big| \what{g_1}\big(t,\xi-\xi\tfrac{|\eta|}{|\xi|}\big) \big| \,
  \big| \what{g_2}\big(t,\xi\tfrac{|\eta|}{|\xi|}\big) \big| \,\mathrm{d}\eta,
\end{split}
\end{align}
where $\nu=\mathrm{sgn}(\xi\cdot(\xi-\eta))$. We then proceed to show that for all $\ell=1,\dots,4$, the terms $L_\ell \in \mathcal{R}$, see \eqref{defRem}.
For this we first prove that on the support of the integrals we have
\begin{equation}\label{aux}
\Big|\frac{\xi-\eta}{|\xi-\eta|} - \nu\frac{\xi}{|\xi|}\Big|\lesssim 2^{p_0},
  \quad \big|\nu|\xi-\eta|-(|\xi|-|\eta|)\big|\lesssim 2^{2p_0+\min(k_1,k_2)}.
\end{equation}
In fact, if $p_0\geq -10$ then first inequality in (\ref{aux}) is obvious, and so is the second inequality, since we must have $\nu=1$ if $|\eta|\ll|\xi|$. Assume $p_0\leq -10$, then from $|\xi\wedge\eta|\lesssim 2^{p_0+k+k_1}$ we know that either $|\angle(\xi,\xi-\eta)|\leq 2^{-p_0}$ and $\nu=1$ or $|\angle(\xi,\xi-\eta)-\pi|\leq 2^{-p_0}$ and $\nu=-1$, and in either case the first inequality in (\ref{aux}) follows. As for the second inequality, notice that $|\xi|^2-|\eta|^2=(\xi+\eta)\cdot(\xi-\eta)$ has the same sign as $\nu$, so we have
\[\big|\nu|\xi-\eta|-(|\xi|-|\eta|)\big|\lesssim\frac{|\xi-\eta|^2-(|\xi|-|\eta|)^2}{|\xi-\eta|+||\xi|-|\eta||}\lesssim 2^{-k_1+k+k_2}|1-\cos\angle(\xi,\eta)|\lesssim 2^{2p_0+\min(k_1,k_2)},
\]noticing also that $k_2\leq k+10$.

\medskip
{\it Estimate of $L_1$}.
Observe that on the support of $L_1$ one has, by (\ref{aux}), that
\begin{align}
\label{Pr11}
\big| \widehat{g_1}(t,\xi-\eta) - \widehat{g_1}\big(t,|\xi-\eta| \tfrac{\nu\xi}{|\xi|}\big) \big|
  \lesssim {\big\| \Omega \what{g_1} \big\|}_{L^\infty} 2^{p_0}.
\end{align}
Also, for every fixed $\xi$, the support of $L_1$ is contained in a solid cone of approximate aperture $2^{p_0}$ and height $2^{k_2}$.
Then we can estimate
\begin{align*}
\big| L_1 \big| \lesssim 2^{k_1-k_2} \cdot {\big\| \Omega \what{g_1} \big\|}_{L^\infty} 2^{p_0}
  \cdot {\| \what{g_2} \|}_{L^\infty} \cdot 2^{2p_0 +3k_2}
\\
\lesssim  2^{k_1+k_2/2} \cdot {\big\| \Omega \what{g_1} \big\|}_{L^\infty} \cdot {\| \what{g_2} \|}_{L^\infty} \cdot 2^{-3m/2}2^{6\delta m}.
\end{align*}
%
%If $k_2 \leq -m/2$
We can use the standard apriori bounds \eqref{aprioriLinfty2} and the improved bound \eqref{fhatimp} to obtain
\begin{align*}
2^{15k^+} 2^k \big| L_1 \big| \lesssim 2^{k/2-5k^+}2^{k_2/2} \cdot \e 2^{C\e m} \cdot \e 2^{-k_2} 2^{\delta' m}\cdot 2^{-3m/2 +6\delta m}
\end{align*}
which suffices in view of $k_2 \geq -7m/8-D$, see \eqref{Pr4}.
%
%If instead $k_2 \geq -m/2$ we can use the standard apriori bounds \eqref{aprioriLinfty2} and the improved bound \eqref{fhatimpcor2} to obtain
%\begin{align*}
%2^{10k^+} 2^k \big| L_1 \big| \lesssim 2^{k/2-5k^+}2^{-k_2/2} \cdot \e 2^{3\delta m} \cdot \e 2^{\delta'm} \cdot 2^{-3m/2}
%\end{align*}
%which again suffices.

\medskip
{\it Estimate of $L_2$}.
By (\ref{aux}) we have
\begin{align}
\big| \what{g_1}\big(t,|\xi-\eta|\tfrac{\nu\xi}{|\xi|}\big) -\what{g_1}\big(t,\xi-\xi\tfrac{|\eta|}{|\xi|}\big) \big|
  \lesssim {\| \nabla \what{g_1} \|}_{L^\infty} \cdot  2^{2p_0+k_2}.
\end{align}

In the case $j_1 \leq 5m/6$ we estimate
\begin{align*}
2^{15k^+} 2^k \big| L_2 \big| & \lesssim 2^{15k^+} 2^k \cdot 2^{k_1-k_2}
  \cdot {\| \what{g_1} \|}_{L^\infty} 2^{j_1} \cdot 2^{2p_0+k_2} \cdot {\| \what{g_2} \|}_{L^\infty} \cdot 2^{2p_0 + 3\min(k_1,k_2)}
\\
& \lesssim 2^{15k^+} 2^{k} \cdot 2^{j_1+k_1} {\| \what{g_1} \|}_{L^\infty} \cdot {\| \what{g_2} \|}_{L^\infty}
  \cdot 2^{-2m + \min(k_1,k_2)} 2^{8\delta m}
\\
& \lesssim 2^{k} \cdot 2^{j_1/2 - k_1} \e 2^{C\e m} \cdot \e 2^{-k_2 + \delta'm}
  \cdot 2^{-2m + \min(k_1,k_2) + 8\delta m}
%\\ & \lesssim
\end{align*}
which is easily seen to suffice.

If instead $j_1 \geq 5m/6$ we do not look at the difference of the profile $g_1$ at the two different locations,
and instead directly estimate using the apriori bound \eqref{aprioriLinfty}:
\begin{align*}
2^{15k^+} 2^k \big| L_2 \big| & \lesssim 2^{15k^+} 2^k \cdot 2^{k_1-k_2}
  \cdot {\| \what{g_1} \|}_{L^\infty} \cdot {\| \what{g_2} \|}_{L^\infty} \cdot 2^{2p_0 + 3\min(k_1,k_2)}
\\
& \lesssim \e 2^{-j_1/2} 2^{(4\delta+C\e) m} \cdot 2^{k_2} {\| \what{g_2} \|}_{L^\infty} \cdot 2^{-m}
\end{align*}
which can be seen to be largely sufficient using \eqref{fhatimpcor2}.

\medskip
{\it Estimate of $L_3$}.
Here we can use
\begin{align}
\label{Pr20}
\big| \widehat{g_2}(t,\eta) - \widehat{g_2}\big(t,|\eta| \tfrac{\xi}{|\xi|}\big)
  \big| \lesssim {\big\| \Omega \what{g_2} \big\|}_{L^\infty} 2^{p_0},
\end{align}
and estimate similarly to the term $L_1$ above:
\begin{align*}
2^k\big| L_3 \big| & \lesssim 2^{k+k_1-k_2} \cdot {\| \what{g_1}\|}_{L^\infty}
  \cdot {\| \Omega \what{g_2} \|}_{L^\infty}  2^{p_0} \cdot 2^{2p_0 + 3\min(k_1,k_2)}
\\
& \lesssim  2^{k+k_1} \cdot {\| \what{g_1} \|}_{L^\infty} \cdot 2^{k_2/2} {\|\Omega \what{g_2} \|}_{L^\infty} \cdot 2^{-3m/2} 2^{6\delta m}.
\end{align*}
This is sufficient in view of the usual bounds \eqref{aprioriLinfty2} and \eqref{fhatimpcor2}, and the lower bound \eqref{Pr4} on $k_2$, by separately considering the cases $k\geq k_1+10$ and $k\leq k_1+10$.

\medskip
{\it Estimate of $L_4$}.
Using \eqref{aux} we have that the symbol in the expression for $L_4$ is bounded by $2^{2p_0}$ and therefore
\begin{align*}
2^{15k^+} 2^k \big| L_4 \big| & \lesssim 2^{15k^+} 2^k \cdot 2^{2p_0} \cdot {\| \what{g_1} \|}_{L^\infty}
  \cdot {\| \what{g_2} \|}_{L^\infty} \cdot 2^{2p_0 + 3\min(k_1,k_2)},
\\
& \lesssim  2^{15k^+} 2^k \cdot {\| \what{g_1} \|}_{L^\infty} \cdot {\| \what{g_2} \|}_{L^\infty} \cdot 2^{-2m + \min(k_1,k_2)} 2^{8\delta m}.
\end{align*}
\iffalse
If $k_1 \geq -m/4$ we can bound this using \eqref{aprioriLinfty} and \eqref{fhatimp}:
\begin{align*}
2^{10k^+} 2^k \big| L_4 \big|
  & \lesssim 2^{-2m} 2^{10\delta_1 m} \cdot 2^k 2^{\min(k_1,k_2)} \cdot \e 2^{-3k_1/2} \cdot \e \max(2^{m/2}, 2^{-k_2})
\end{align*}
which is sufficient to obtain a bound of $\e^2 2^{-(1+2\gamma)m}$.
If instead  $k_1 \leq -m/4$ we estimate
\begin{align*}
2^{10k^+} 2^k \big| L_4 \big|
  & \lesssim 2^{-2m} 2^{10\delta_1 m} \cdot 2^k 2^{\min(k_1,k_2)} \cdot \e  \max(2^{m/2}, 2^{-k_1}) \cdot \e \max(2^{m/2}, 2^{-k_2})
\end{align*}
which is again sufficient, and concludes the estimate of $J-J^\parallel$.
\fi
Then we can directly invoke \eqref{fhatimpcor2} and see that this contribution is also controlled as desired.
This completes the estimate for the four terms in \eqref{Pr10}, hence for $I_{++} - J_{++}$

The estimates of $I_{\kappa_1\kappa_2} - J_{\kappa_1\kappa_2}$ in the remaining cases $(\kappa_1\kappa_2) = (+-),(-+)$ can be done similarly,
se we omit the details.
Putting together the above steps we have obtained
\begin{align}
\label{I-Jpar}
I_{\kappa_1\kappa_2} - J_{\kappa_1\kappa_2}^\parallel \in \mathcal{R}.
\end{align}
%see Definition \ref{defRem}.

%\comment{write out the proof for the other terms or omit them?}

\bigskip
{\bf Step 3}: {\it Phase oscillations and asymptotics}.
Let us write the integrals in \eqref{Jpar} in spherical coordinates, by denoting $\xi = \rho \theta$, $\theta \in \mathbb{S}^2$:
\begin{align}
\label{Jparpol}
\begin{split}
J_{++}^{\parallel} & := \int_{0}^\infty \int_{\mathbb{S}^2} e^{it(\rho-|\rho\theta-r\phi|-r)} |\rho-r| r \, \varphi_{\leq 5}(r/\rho) \,
  \chi\big(\angle(\theta,\phi) 2^{-p_0+(k_2-k_1)^+}\big) \,
  \\ & \hskip250pt \times \what{f^1}(t,(\rho-r)\theta) \what{f^1}(t,r\theta\big)\, \mathrm{d}\phi \mathrm{d}r,
\\
J_{+-}^{\parallel} & := \int_{0}^\infty \int_{\mathbb{S}^2} e^{it(\rho-|\rho\theta-r\phi|+r)} (\rho+r)r \,
  \chi\big( [\angle(\theta,\phi) - \pi]2^{-p_0+(k_2-k_1)^+}\big)
  \\ & \hskip250pt \times \what{f^1} (t,(\rho+r)\theta) \what{\overline{f^1}}\big(t,-r\theta\big)\, \mathrm{d}\phi \mathrm{d}r,
\\
J_{-+}^{\parallel} & := \int_{0}^\infty \int_{\mathbb{S}^2} e^{it(\rho+|\rho\theta-r\phi|-r)} |r-\rho|r \, \varphi_{\leq 5}(\rho/r) \,
  \chi\big(\angle(\theta,\phi)) 2^{-p_0+(k_2-k_1)^+})\big)
  \\ & \hskip250pt \times \what{\overline{f^1}}\big(t,(\rho-r)\theta\big) \what{f^1}\big(t,r\theta\big) \, \mathrm{d}\phi \mathrm{d}r.
\end{split}
\end{align}

We then define, for $\rho\geq0$ and $r\in\R$ such that $|\rho|\approx 2^{k}$, $|\rho-r|\approx 2^{k_1}$ and $|r|\approx 2^{k_2}$, that
\begin{align}
\label{Itheta}
I(\theta;t,r,\rho) & := |(\rho-r) r| \int_{\mathbb{S}^2} e^{it(\rho-|\rho\theta-r\phi|-r)} \,
  \chi\big(\angle(\theta,\phi) 2^{-p_0+(k_2-k_1)^+})\big) \,  \mathrm{d}\phi
\\
\label{Ftheta}
F_{n}^\theta(t,x) & := \what{f_n}(t,x\theta), \qquad f_n = Q_{j_n k_n}f^1, \qquad n=1,2,
\end{align}
and rewrite
\begin{align}
\label{Jparpol2}
\begin{split}
J_{++}^{\parallel} & := \int_{0}^\infty I(\theta;t,r,\rho) \, \varphi_{\leq 5}(r/\rho) \, F_1^\theta(t,\rho-r) F_2^\theta(t,r)\, \mathrm{d}r,
\\
J_{+-}^{\parallel} & := \int_{-\infty}^0 I(\theta;t,r,\rho) \, F_1^\theta(t,\rho-r) \overline{F_2^\theta(t,-r)}\,\mathrm{d}r,
\\
J_{-+}^{\parallel} & := -\int_{0}^\infty \overline{I(\theta;t,-r,-\rho)} \, \varphi_{\leq 5}(\rho/r) \,
  \overline{F_1^\theta(t,r-\rho)} F_2^\theta(t,r)\, \mathrm{d}r,
\end{split}
\end{align}
having changed $(r,\phi)\mapsto -(r,\phi)$ to obtain the expression for $J_{+-}^{\parallel}$.
By rotational symmetry we know that $I(\theta;t,r,\rho)$ is independent of $\theta$.
\iffalse
We first notice that one may reduce the support of the integrals \eqref{Jparpol2} to a region where
\begin{align}
\label{Jparpol2.5}
\min(|r|,|\rho-r|) \approx \min(2^{k_1}, 2^{k_2}) \gtrsim 2^{-3\gamma m},
\end{align}
for otherwise these terms are already acceptable remainders.
To see this let let look at the first term in \eqref{Jparpol}, the others being similar, and assume without loss of generality that $k_2 \leq k_1$.
We can then estimate, for all $\rho \approx 2^k$,
\begin{align*}
\big| \rho \Omega^\ell J_{++}^{\parallel} \big| \lesssim \rho \int_0^\infty |\rho-r| \, r
  \int_{\mathbb{S}^2}  \chi\big(\angle(\theta,\phi) 2^{-p_0+(k_2-k_1)^+})\big) \, \mathrm{d}\phi \, \big| \partial_\theta^{\ell_1}F_1^\theta(t,\rho-r) \big| \big|\partial_{\theta}^{\ell_2}F_2^\theta(t,r)\big| \,\mathrm{d}r
\\
\lesssim 2^{k+k_1+k_2} \cdot 2^{2p_0}
  \cdot \sup_{x,\theta}\big| \partial_\theta^{\ell_1}F_1^\theta(t,x) \big| \cdot 2^{k_2/2} \sup_{\theta} {\big\| \partial_{\theta}^{\ell_2}F_2^\theta(t,\cdot) \big\|}_{L^2_r}
\\
\lesssim 2^{-m + 4\delta m} \cdot 2^{2k_1} {\| \what{g_1} \|}_{L^\infty}
  \cdot 2^{k_2/2} {\|\Omega^3 \what{g_2}\|}_{L^2}.
\end{align*}
Using the basic a priori bounds \eqref{aprioriLinfty} we see that this is an acceptable remainder for $k_2 \leq -3\gamma m$. %Of course can do tighter restriction
\fi
To arrive at our final asymptotic expression we now calculate asymptotics for $I$ in \eqref{Itheta}.

\begin{lemma}\label{LemItheta}
Let $\rho \approx 2^k$, $|r-\rho| \approx 2^{k_1}$, $|r| \approx 2^{k_2}$.
Let $t \in [2^{m-1},2^{m+1}]$, $m\geq 1$, and assume that $\min(k,k_1,k_2)\geq -7m/8+10$.
Let $p_0$ be given as in \eqref{Pr2.5}. Then, for all $\theta \in \mathbb{S}^2$, we have
\begin{align}
\label{Ithetaexp}
%I_0(\theta) := \int_{\mathbb{S}^2} e^{it(\rho-|\rho\theta-r\phi|-r)} \, \chi\big(\angle(\theta,\phi) 2^{-p_0}\big) \,  \mathrm{d}\phi
% = \frac{2\pi}{it} \cdot \frac{\rho-r}{r \rho} + O\big( \rho^{-1} t^{-1-2\gamma} \big).
I(\theta;t,r,\rho) = \frac{\pi}{it} e^{it(\rho-r-|\rho-r|)}\cdot \frac{(\rho-r)^2}{\rho} + O\big( 2^{-k} 2^{5k_1^++5k_2^+}2^{(-2+12\delta)m}\big).
\end{align}

%\comment{Could rotate $\theta$ to $e_3$ for simplicity}

\end{lemma}

\begin{proof}[Proof of Lemma \ref{LemItheta}]
We assume $r\geq 0$ (the other case is similar), and denote $|\rho-r|=\sigma$.
For simplicity we will also assume $k_1,k_2\leq 0$. Note that in the support of the integral we have
\[
|\nu|:=|\angle(\theta,\phi)|\lesssim 2^{p_0}\min\bigg(1,\frac{\sigma}{r}\bigg);
\] 
moreover we write
\[
|\rho\theta-r\phi|=\sqrt{\rho^2+r^2-2\rho r\cos\nu}=\sigma\sqrt{1+\frac{2\rho r(1-\cos\nu)}{\sigma^2}}
  =\sigma+\frac{2\rho r}{\sigma}(1-\cos\nu)+O\bigg(\frac{\rho^2 r^2}{\sigma^3} \nu^4\bigg),
\]
which implies that
\[
e^{it(\rho-r-|\rho\theta-r\phi|)}=e^{it(\rho-r-|\rho-r|)}e^{-\frac{2it\rho r}{\sigma}(1-\cos\nu)} + O\bigg(\frac{t\rho^2 r^2}{\sigma^3} \nu^4\bigg).
\]
The integral of the error term is bounded by
\[ 2^{k_1+k_2}\cdot 2^{m + 2k+2k_2-3k_1} 2^{6(p_0-(k_2-k_1)^+)}\lesssim 2^{-2m+12\delta m}2^{-k}.
\]
We then consider the main term, which is now
\[|r\sigma|e^{it(\rho-r-|\rho-r|)}\int_{\mathbb{S}^2}e^{-\frac{2it\rho r}{\sigma}(1-\cos\nu)} \chi\big(2^{-p_0+(k_2-k_1)^+}\nu)\big)\,\mathrm{d}\phi.
\]
If one replaces the cutoff function $\chi\big(2^{-p_0+(k_2-k_1)^+}\nu)\big)$ by $\chi_1(\nu)$,
where $\chi_1$ is a fixed smooth cutoff function supported at $|\nu|\lesssim 1/100$ and equals $1$ for $|\nu|\lesssim 1/200$,
then by the same argument as in Lemma \ref{LemAIBP}, we can show that the difference introduced will be $O(2^{-100m})$,
since in the support of $\chi-\chi_1$ there is no critical point of the phase $1-\cos\nu$.
Therefore, below we will replace the cutoff function by $\chi_1$. Writing in spherical coordinates we get
\[
\int_{\mathbb{S}^2}e^{-\frac{2it\rho r}{\sigma}(1-\cos\nu)} \chi_1(\nu)\,\mathrm{d}\phi
  = 2\pi\int_{0}^\pi  e^{-\frac{2it\rho r}{\sigma}(1-\cos\nu)} \sin(\nu)\chi_1(\nu)\,\mathrm{d}\nu
  = 2\pi\int_{-1}^1e^{-\frac{2it\rho r}{\sigma}(1-\lambda)}\chi_2(\lambda)\,\mathrm{d}\lambda,
\]
where $\lambda=\cos\nu$, and $\chi_2$ is supported in $|\lambda-1|\leq 1/10$ and equals $1$ for $|\lambda-1|\leq 1/200$.
Integrating by parts in $\lambda$ and noticing that the boundary term at $\lambda=-1$ vanishes, we obtain that
\[
\int_{-1}^1e^{-\frac{2it\rho r}{\sigma}(1-\lambda)}\chi_2(\lambda)\,\mathrm{d}\lambda=\frac{\sigma}{2it\rho r}+O(2^{-100m}),
\]
noticing also that $(t\rho r)/\sigma\gtrsim 2^{m/8}$. Summing up, we get that
\[I(\theta;t,r,\rho)=\frac{\pi}{it} e^{it(\rho-r-|\rho-r|)}\cdot \frac{(\rho-r)^2}{\rho} +O(2^{-k}2^{-2m+12\delta m}). \qedhere
\]

\end{proof}

\bigskip
{\bf Step 4}: {\it Conclusion}.
We still need to control the error terms coming from Lemma \ref{LemItheta}, which is of form $R$ where
\[ 2^k|\Omega^{\ell}R|\lesssim 2^{-2m+12\delta m} \int_{|r|\approx 2^{k_2},|\rho-r|\approx 2^{k_1}}|
  \what{g_1}(t,(\rho-r)\theta)| \cdot |\what{g_2}(t,(\rho-r)\theta)| \,\mathrm{d}r.
\]
Now we may assume $k_1\geq k_2$; if $k_2\leq -m/2$ the above is bounded by
\[2^{-(2-12\delta)m} 2^{-k_1} 2^{2\delta'm},
\]
which suffices since $k_1\geq -7m/8-10$. If $k_2\geq-m/2$ then we have the bound
\[2^{-(2-12\delta)m}\cdot 2^{m/2}2^{\delta'm}\cdot 2^{\delta'm},
\]
which also suffices.
Moreover, in the integral $J_{++}^{\|}$, if $r\geq \rho$, then by Lemma \ref{LemItheta}, up to acceptable error terms, this integral is of form
\[
\int_{\rho}^{+\infty}e^{2it(\rho-r)}|\rho-r| r \, \varphi_{\leq 5}(r/\rho)  \what{f^1}(t,(\rho-r)\theta) \what{f^1}(t,r\theta\big)\,  \mathrm{d}r,
\]
which decays like $2^{-100m}$ since we can integrate by parts in $r$, using the fact that $\max(j_1,j_2)\leq (1-4\delta)m$
and that there is no boundary term since $\what{f^1}$ is supported on $|\rho-r|\gtrsim 2^{-7m/8}$.
Thanks to this we can restrict the $r$ integral in $J_{++}^{\|}$ to $r\in[0,\rho]$,
and similarly restrict the integral in $J_{+-}^{\|}$ to $r\in(-\infty,0]$ and the integral in $J_{-+}^{\|}$ to $r\in[\rho,+\infty)$.
In conclusion, we see from from the formulas \eqref{Jparpol3}, with \eqref{Itheta}-\eqref{Ftheta},
and using \eqref{Ithetaexp}, and the apriori bound \eqref{aprioriE}, that
\begin{align}
\label{Jparpol3}
\begin{split}
J_{++}^{\parallel} & = \frac{\pi}{it} \int_{0}^\rho \frac{(\rho-r)^2}{\rho} \varphi_{\leq 5}(r/\rho)
  \, F_1^\theta(t,\rho-r) F_2^\theta(t,r)\, \mathrm{d}r + \mathcal{R},
\\
J_{+-}^{\parallel} & = \frac{\pi}{it} \int_{-\infty}^0 \frac{(\rho-r)^2}{\rho} \, F_1^\theta(t,\rho-r) \overline{F_2^\theta(t,-r)}\, \mathrm{d}r	
  + \mathcal{R},
\\
J_{-+}^{\parallel} & = \frac{\pi}{it}  \int_{\rho}^\infty \frac{(\rho-r)^2}{\rho} \, \varphi_{\leq 5}(\rho/r) \,
  \overline{F_1^\theta(t,r-\rho)} F_2^\theta(t,r)\, \mathrm{d}r + \mathcal{R}.
\end{split}
\end{align}
where $\mathcal{R}$ denotes the usual acceptable remainder terms.

We then extend the profile $F^\theta$ in \eqref{Ftheta} to negative arguments by letting
\begin{align}
\label{Gtheta}
\begin{split}
G_{n}^\theta(t,x) & := F_{n}^\theta(t,x), \qquad x>0,
\\
G_{n}^\theta(t,x) & := \overline{F_{n}^\theta(t,-x)}, \qquad x<0, \qquad n=1,2.
\end{split}
\end{align}
With this definition, recalling the formula \eqref{Pr2} and the notation $I_{\kappa_1\kappa_2} = 4(2\pi)^{3/2}I_{\kappa_1\kappa_2}^{m,k,k_1,k_2}$,
and putting together \eqref{I-Jpar} and \eqref{Jparpol3}, we see that
\begin{align}
I_{++}^{m,k,k_1,k_2} - I_{+-}^{m,k,k_1,k_2} - I_{-+}^{m,k,k_1,k_2}
  - \frac{1}{4it\sqrt{2\pi}} \int_{-\infty}^\infty \frac{(\rho-r)^2}{\rho} G^\theta_1(\rho-r) G^\theta_2(r) \, \mathrm{d}r \in \mathcal{R}.
\end{align}
which implies the desired conclusion. \qed

\subsection{Proof of Lemma \ref{LemAux}}\label{secLemAux}
Let $L_{\theta}(t,r)=\varphi_{\leq-10}(r\langle t\rangle^{7/8})F_{\theta}(t,r)$, we decompose
\[\frac{i}{4\sqrt{2}\pi t}\int_\R \frac{(\rho-r)^2}{\rho} F_\theta(t,\rho-r)F_\theta(t,r) \, \mathrm{d}r=\mathcal{N}(t,\rho)+\frac{i}{4\sqrt{2}\pi t}\big(A(t,\rho)+B(t,\rho)+C(t,\rho)\big),
\]
where
\begin{align*}
A(t,\rho) &:=\int_\R \frac{(\rho-r)^2}{\rho} L_\theta(t,\rho-r)G_\theta(t,r) \, \mathrm{d}r,
\\
B(t,\rho) &:=\int_\R \rho F_\theta(t,\rho)L_\theta(t,r) \, \mathrm{d}r,
\\
C(t,\rho) &:=\int_\R \bigg(\frac{(\rho-r)^2}{\rho} F_\theta(t,\rho-r)-\rho F_\theta(t,\rho)\bigg)L_\theta(t,r) \, \mathrm{d}r.
\end{align*}
Let $|\rho|\approx 2^k$, $|\rho-r|\approx 2^{k_1}$ and $|r|\approx 2^{k_2}$; estimating $A$ using Lemma \ref{lemImp}, we get
\[ 2^k2^{20k^+}\sup_{|\alpha|\leq N_1}|\Omega^{\alpha}A(t,\rho)|
  \lesssim \int_{|\rho -r|\lesssim 2^{-7m/8}\lesssim|r|}\varepsilon^22^{m/5}|\rho-r|^2|\rho-r|^{-1}|r|^{-1}\,\mathrm{d}r\lesssim\varepsilon^{2}2^{-m/2},
\]
so $t^{-1}A(t,\rho)\in\mathcal{R}$. By similar estimates, and using \eqref{LemL1.3}, we can show that
\[t^{-1}\bigg(B(t,\rho)+\frac{1}{8(2\pi)^{3/2}}\rho F_\theta(\rho)\cdot\int_{\R} \varphi_{\leq-10}(r\langle t\rangle^{7/8}) H_\theta(t,r)\,\mathrm{d}r\bigg)\in\mathcal{R}.
\]

Finally, for $C(t,\rho)$ we can decompose
\begin{align*}
\begin{split}
t^{-1}C(t,\rho) & = t^{-1} \, \int_\R \bigg(\frac{(\rho-r)^2}{\rho}-\rho\bigg) F_\theta(t,\rho)L_\theta(t,r) \, \mathrm{d}r
\\
& + t^{-1} \, \int_\R \frac{(\rho-r)^2}{\rho}  (F_\theta(t,\rho-r)-F_\theta(t,\rho))L_\theta(t,r) \, \mathrm{d}r.
\end{split}
\end{align*}
The first term can be shown to be in $\mathcal{R}$ in the same way as above.
For the second term, as $|r|\lesssim 2^{-7m/8}$, we may assume $|\rho-r|\gg 2^{-7m/8}$ ,
for otherwise the desired bound follows from the smallness of $(\rho-r)^2$;
then, decomposing $f$ into $Q_{jk}f$ (and $F$ accordingly) as in Section \ref{secsetup},
and using \eqref{aprioriLinfty} and the mean value Theorem, we get that
\[
\sup_{|\alpha|\leq N_1}|\Omega^{\alpha}(F_\theta(t,\rho-r)-F_\theta(t,\rho))|\lesssim \sup_{j}\min(2^{-7m/8+j/2-2k},2^{-j/2-2k})\lesssim 2^{-7m/16-2k},
\]
so the contribution of this part to the $X$ norm of $t^{-1}C(t,\rho)$ is bounded by
\[
2^{-m}2^{2k}\cdot 2^{-7m/16-2k} 2^{\delta'm} \cdot\int_{|r|\lesssim 2^{-7m/8}}\min(|r|^{-1},2^m)\,\mathrm{d}r\lesssim 2^{-4m/3}.
\]
This gives that $C(t,\rho)\in\mathcal{R}$, which completes the proof. \qed

\bigskip
\section{Proof of Theorem \ref{Mainth2}: Nonlinear Asymptotics}\label{SecProof2}
We start by proving an upper bound for the correction term $B_\theta(t)$ defined in \eqref{Mainth13}.
Recall that the vector fields in $\Omega$ are equivalent to $\partial_{\theta}$.
\begin{proposition}\label{b_bound}
We have
\begin{equation} \label{b_bound0}
\sup_{|\alpha|\leq 2N_1}|\Omega^{\alpha}H_\theta(t,\rho)|
  \lesssim \e \varphi_{\leq-10}(\rho\langle t \rangle^{7/8})\rho^{-1}(1+|t|)^{C\e},
  \quad \sup_{|\alpha|\leq N_1} |\Omega^{\alpha}B_\theta(t)|\lesssim \e (1+|t|)^{-1+C\e}.
\end{equation}
\end{proposition}

\begin{proof}
Assume $|t|\gtrsim 1$. Recall that
\begin{align*}
H_\theta(t,\rho) & :=  \varphi_{\leq-10}(\rho\langle t \rangle^{7/8}) \int_0^t \int_\R \int_{\mathbb{S}^2}
  e^{is \rho[1-\theta \cdot \phi]} \big| \widehat{f}(t,r\phi) \big|^2 r^2 \,\mathrm{d}\phi \mathrm{d}r \mathrm{d}s,
\\
B_\theta(t) & := \frac{1}{32\pi^2}
  \mathrm{Re} \Big[ \int_\R \int_{\mathbb{S}^2} %\varphi_{\leq 0}(r\langle t \rangle^{1-\delta})
  e^{itr[1-\theta \cdot \phi]} H_\phi(t,r) \,r \,\mathrm{d}\phi \mathrm{d}r \Big]; %cutoff redundant, just to keep it in mind
\end{align*}
we then have, for $|\alpha|\leq 2N_1$, that
\[\Omega^{\alpha}H_\theta(t,\rho)=\sum_{\alpha_1+\alpha_2=\alpha}  \varphi_{\leq-10}(\rho\langle t \rangle^{7/8}) \int_0^t \int_\R \int_{\mathbb{S}^2}
  e^{is \rho[1-\theta \cdot \phi]} \Omega^{\alpha_1}\widehat{f}(t,r\phi)\cdot\overline{\Omega^{\alpha_2}\widehat{f}(t,r\phi)} r^2 \,\mathrm{d}\phi \mathrm{d}r \mathrm{d}s.
\] 
Let $|t|\approx 2^m$, by using (\ref{Nrg}), Lemma \ref{lemIBP0} and integrating by parts in $\phi$, 
we see that the above integral can be restricted to the region $|\sin\angle(\theta,\phi)|\lesssim (1+s\rho)^{-1/2}2^{C\varepsilon m}$; using also the $L^{\infty}$ bounds (\ref{aprioriLinfty2}), we estimate
\begin{multline*}
|\Omega^{\alpha}H_\theta(t,\rho)|\lesssim \varepsilon\varphi_{\leq-10}(\rho\langle t \rangle^{7/8})2^{C\varepsilon m}\int_0^t(1+s\rho)^{-1}\,\mathrm{d}s\cdot\int_{\mathbb{R}}\langle r\rangle^{-5}r^2\min(r^{-3/2},2^m)^2\,\mathrm{d}r\\\lesssim\varepsilon\varphi_{\leq-10}(\rho\langle t \rangle^{7/8})\rho^{-1}2^{C\varepsilon m}.
\end{multline*} 
Similarly, we have
\[\Omega^\alpha B_\theta(t) = \frac{1}{32\pi^2}
  \mathrm{Re} \Big[ \int_\R \int_{\mathbb{S}^2} %\varphi_{\leq 0}(r\langle t \rangle^{1-\delta})
  e^{itr[1-\theta \cdot \phi]} \Omega^{\alpha}H_\phi(t,r)r\,\mathrm{d}\phi\mathrm{d}r\bigg],
\]
so using the bounds for $H_\phi(t,r)$ just proved, and Lemma \ref{lemIBP0} and integrating by parts in $\phi$, 
we can again restrict the integral to the region $|\sin\angle(\theta,\phi)|\lesssim (1+tr)^{-1/2}2^{C\varepsilon m}$, and hence obtain
\[|\Omega^\alpha B_\theta(t)|\lesssim\varepsilon 2^{C\varepsilon m}\int_{|r|\lesssim 2^{-7m/8}}(1+tr)^{-1}\,\mathrm{d}r\lesssim \varepsilon t^{-1+C\varepsilon},
\] 
which completes the proof.
\end{proof}

Now, using the same arguments as in Proposition \ref{b_bound}, we can show
\begin{equation}
\label{boundh}
\sup_{|\alpha|\leq N_1}|\Omega^{\alpha}h_\theta(t,\rho)|\lesssim \rho^{-1}(1+|t|)^{-1+C\varepsilon};
\qquad \sup_{|\alpha|\leq N_1}|\Omega^{\alpha}C_\theta(t)|\lesssim (1+|t|)^{-1+C\varepsilon},
\end{equation} 
see the definitions \eqref{Mainth12}-\eqref{Mainth13}.
By Theorem \ref{Mainth1}, we have
\begin{align}
\label{asymptpde2}
\begin{split}
\partial_t F_\theta(t,\rho) =-i\rho C_\theta(t) F_\theta(t,\rho)+ \frac{1}{it}\frac{1}{4(2\pi)^{1/2}}
  \int_\R \frac{(\rho-r)^2}{\rho} F_\theta(t,\rho-r)F_\theta(t,r) \, \mathrm{d}r
  \\
  + \frac{1}{2(2\pi)^{3/2}}\varphi_{\leq0}(\rho\langle t \rangle^{7/8})\cdot h_{\theta}(t,\rho)+ \mathcal{R}(t,\xi).
\end{split}
\end{align}
Let $U_\theta=U_\theta(s,q)$ be defined by
\begin{equation}
\label{defU}
\begin{split}
(\mathcal{F}_qU_\theta)(s,\rho) & := \frac{1}{4(2\pi)^{1/2}} e^{-i\rho D_\theta(e^s)}F_\theta(e^s,\rho)
\\
U_\theta(s,q) & := \frac{1}{4(2\pi)^{1/2}} (\mathcal{F}_{\rho}^{-1}F_\theta)(e^s,q+D_\theta(e^s)),
\end{split}
\end{equation}
where
\begin{equation}
\label{defU'}
D_\theta(t) := \int_0^t C_\theta (t')\,\mathrm{d}t'.
\end{equation}
We then calculate, using (\ref{asymptpde2}), that
\begin{equation}
\label{eqnU}
\partial_s\partial_qU_\theta(s,q)=-U_\theta(s,q)\cdot\partial_q^2U_\theta(s,q)+\mathcal{E}_\theta(s,q),
\end{equation} 
where
\begin{equation}
\mathcal{F}_q\mathcal{E}_\theta(s,\rho) = \frac{1}{4(2\pi)^{1/2}} e^se^{-i\rho D_\theta(e^s)}
  \bigg[\frac{1}{2(2\pi)^{3/2}}\varphi_{\leq0}(\rho\langle e^s\rangle^{7/8})\cdot h_{\theta}(e^s,\rho) + \mathcal{R}(e^s,\xi)\bigg].
\end{equation}
Using the definition of $X$ norm, the assumption about the error term $\mathcal{R}$, and the bound (\ref{boundh}) on $h_\theta$, we obtain that
\begin{equation}
\label{bderror}\mathcal{E}=\mathcal{E}_1+\partial_s\mathcal{E}_2,
  \quad \sup_{|\alpha|\leq 15,|\beta|\leq N_1}\|\nabla^{\alpha}\Omega^{\beta}\mathcal{E}_j(s)\|_{L_q^2\cap L_q^\infty}\lesssim \varepsilon^2 e^{-\gamma s}
\end{equation}
for $j\in\{1,2\}$.

Using (\ref{eqnU}), we can obtain the asymptotic behavior of $U$ as follows.
\begin{proposition}\label{finaldata}There exists a function $\widetilde{U}=\widetilde{U}_\theta(s,q)$, satisfying \[\sup_{|\alpha|\leq 14,|\beta|\leq N_1}\|\nabla^{\alpha}\Omega^{\beta}\partial_q\widetilde{U}_\theta(s)\|_{L_q^2\cap L_q^\infty}\lesssim\varepsilon e^{C\varepsilon s},\] and the equation
\[\partial_s\partial_q\widetilde{U}_\theta+\widetilde{U}_\theta\cdot\partial_q^2\widetilde{U}_\theta=0,
\] such that we have
\begin{equation}\label{diffest}\sup_{|\alpha|\leq 14,|\beta|\leq N_1}\|\nabla^{\alpha}\Omega^{\beta}\partial_q(\widetilde{U}_\theta(s,q)-U_\theta(s,q))\|_{L_q^2\cap L_q^\infty(|q|\lesssim e^{\gamma s/10})}\lesssim\varepsilon^2 e^{-\gamma s/20}.
\end{equation}
\end{proposition}

\begin{proof} For simplicity, denote
\[\|G\|_{Y_a}:=\sup_{|\alpha|\leq a,|\beta|\leq N_1}\|\nabla^{\alpha}\Omega^{\beta}G(s)\|_{L_q^2\cap L_q^\infty}.
\]

We use the method of characteristics. Let $V_\theta(s,q)=(\partial_qU_\theta)(s,q)$ and $z_\theta=z_\theta(s,q)$ be defined such that
\[\partial_sz_\theta(s,q)=U_\theta(s,z_\theta(s,q)),\quad z_\theta(0,q)=q,
\] and make the bootstrap assumption
\begin{equation}\|\log\partial_qz_\theta(s,q)\|_{Y_{14}}\leq C\varepsilon s.\label{boot}
\end{equation}For simplicity we will omit the subscript $\theta$ below. Then we calculate that
\begin{multline*}\partial_s[V(s,z(s,q))]=[(\partial_{s}\partial_q+U\cdot\partial_q^2)U](s,z(s,q))=\mathcal{E}(s,z(s,q))\\=\mathcal{E}_1(s,z(s,q))+\partial_s[\mathcal{E}_2(s,z(s,q))]-(\partial_q\mathcal{E}_2)(s,z(s,q))\cdot U(s,z(s,q)).
\end{multline*} Using (\ref{bderror}) and the bounds for $U$ which follow from Lemma \ref{lemImp}, as well as the bootstrap assumption (\ref{boot}), we obtain that
\[
V(s,z(s,q))=V_{\infty}(q)+\mathcal{O}_{Y_{14}}(\varepsilon^2 e^{-\gamma s/6}),
\]
for some function $V_\infty\in Y_{14}$ with $\|V_\infty\|_{Y_{14}}\leq C_0\varepsilon$, where $C_0$ depends only on the initial data,
and $\mathcal{O}_{Y_a}(\varepsilon^2 e^{-\gamma s/6})$ denotes any function that is $O(\varepsilon^2 e^{-\gamma s/6})$ measured in the norm $Y_a$.

Next, we calculate 
\[\partial_s(\log\partial_qz(s,q))=V(s,z(s,q))=V_{\infty}(q)+\mathcal{O}_{Y_{14}}(\varepsilon^2 e^{-\gamma s/6}),\]
upon integrating in $s$ (and choosing $C$ large enough depending on $C_0$), we can recover the bootstrap assumption \eqref{boot}. 
Moreover, we have
\[
\log\partial_qz(s,q) = sV_\infty(q)+\int_0^s\mathcal{O}_{Y_{14}}(\varepsilon^2 e^{-\gamma s/6})\,\mathrm{d}s'.
\]
Define
\[E_\infty(q):=\int_0^{\infty}\mathcal{O}_{Y_{14}}(\varepsilon^2 e^{-\gamma s/6})\,\mathrm{d}s',\]
where the $\mathcal{O}_{Y_{14}}(\cdot)$ term is the same as above, we then obtain that
\[\log\partial_qz(s,q)=sV_\infty(q)+E_\infty(q)+\mathcal{O}_{Y_{14}}(\varepsilon^2 e^{-\gamma s/6}).
\]

We can now define the function $\widetilde{U}(s,q)$.
Let $q=q_0(s)$ be the unique point where $z(s,q_0(s))=0$, we let
\begin{align}
\begin{split}
& \partial_q\widetilde{z}(s,q) = e^{sV_\infty(q)+E_\infty(q)}, \qquad \widetilde{z}(s,q_0(s))=0
\\
& \widetilde{U}(s,\widetilde{z}(s,q))=\partial_s\widetilde{z}(s,q).
\end{split}
\end{align}
By calculations similar to the ones above we can check that
$\partial_s\partial_q\widetilde{U}+\widetilde{U}\cdot\partial_q^2\widetilde{U}=0$,
as claimed in the statement.

Finally, we let $\widetilde{V}:=\partial_q\widetilde{U}$ and need to control $\widetilde{V}(s,q)-V(s,q)$ as in \eqref{diffest}.
For $|q|\lesssim e^{\gamma s/10}$, we may replace $q$ by $z(s,q)$ and reduce to considering $\widetilde{V}(s,z(s,q))-V(s,z(s,q))$  for $|q-q_0(s)|\lesssim e^{\gamma s/9}$.
Moreover, since
\[V(s,z(s,q)) = V_\infty(q)+\mathcal{O}_{Y_{14}}(\varepsilon e^{-\gamma s/6}),\qquad \widetilde{V}(s,\widetilde{z}(s,q)) = V_\infty(q),
\]
we just need to control $\widetilde{V}(s,z(s,q))-\widetilde{V}(s,\widetilde{z}(s,q))$,
which is essentially bounded by $|\partial_q\widetilde{V}(s)|\cdot|z(s,q)-\widetilde{z}(s,q)|$.
Then, since $z(s,q_0(s)) = \widetilde{z}(s,q_0(s)) = 0$, $|q-q_0(s)|\lesssim e^{\gamma s/9}$, and
\begin{align*}
|\partial_qz-\partial_q\widetilde{z}|=e^{sV_\infty(q)}\cdot\big(e^{\mathcal{O}_{Y_{14}}(\varepsilon e^{-\gamma s})}-1\big)=\mathcal{O}_{Y_{14}}(\varepsilon^2e^{-\gamma s/6}),
\end{align*}
we end up obtaining \eqref{diffest}.
\end{proof}

To conclude the proof of Theorem \ref{Mainth2}, we need to transfer the asymptotics of $U$ to the asymptotics of $u$ (and $\partial u$), which is done in the following:

\begin{proposition}\label{transfer0}
Let $x=r\omega\in\mathbb{R}^3$, such that $\omega\in\mathbb{S}^2$ and $|r-t|\lesssim e^{\gamma s/10}$, where $t=e^s$. Then we have
\begin{equation}\label{transfer}
\sup_{|\alpha|\leq 10,|\beta|\leq N_1}\bigg|\nabla^{\alpha}\Omega^{\beta}\partial\bigg(u(t,x)-\frac{\omega_j}{2r}
  \partial_q\mathcal{F}^{-1}F_\omega(t,r-t)\bigg)\bigg|\lesssim \varepsilon t^{-1-\gamma/30}.
\end{equation}
\end{proposition}

\begin{proof}
Let $|t|\approx 2^m$. We consider only the $\partial_j$ component in $\partial$, the others being similar. By definition, we can write
\begin{equation}
\begin{split}
\partial_ju(t,x) &=-\Im\int_{\R^3}\frac{i\xi_j}{|\xi|}e^{-it|\xi|+ix\cdot\xi}\widehat{f}(\xi)\,\mathrm{d}\xi
\\
&=\Re\int_{0}^\infty     e^{-it\rho}\rho^2\,\mathrm{d}\rho\int_{\mathbb{S}^2}\theta_je^{\rho r(\omega\cdot\theta)}F_\theta(t,\rho)\,\mathrm{d}\theta.
\end{split}
\end{equation}In the above integral, with fixed $\rho$, we can restrict to the region $|\sin\angle(\theta,\omega)|\lesssim (\rho r)^{-1/2}2^{C\varepsilon m}$, by integrating by parts in $\theta$ using (\ref{Nrg}) and Lemma \ref{lemIBP0}; moreover when $\theta$ is close to $\pm\omega$ we can replace $F_\theta(t,\rho)$ by $F_{\pm\omega}(t,\rho)$, where the control on the difference is given by
\[\sup_{|\alpha|\leq N_1}|\Omega^{\alpha}(F_{\theta}(t,\rho)-F_{\pm\omega}(t,\rho))|\lesssim (r\rho)^{-1/2}2^{C\varepsilon m}\sup_{|\alpha|\leq N_1+1}\|\Omega^{\alpha}F_\theta(t,\rho)\|_{L^\infty}
\] and Lemma \ref{lemImp}.
By a standard calculation on oscillatory integrals (see for example Lemma \ref{LemItheta}), we then obtain that
\[\int_{\mathbb{S}^2}\theta_je^{\rho r(\omega\cdot\theta)}F_\theta(t,\rho)\,\mathrm{d}\theta=\frac{i}{\rho r}\bigg(\omega_je^{-i\rho r}F_{-\omega}(\rho)+\omega_je^{i\rho r}F_{\omega}(\rho)\bigg)+(\mathrm{error}_1),
\] where the error term satisfies estimates that are at least $2^{-m/10}$ better than the main term, namely
\[|(\mathrm{error}_1)|\lesssim \varepsilon2^{-m/10} \min\big(1,\frac{1}{|\rho r|}\big)\cdot\min(2^m,|\rho|^{-1}),
\]together with the corresponding bounds with vector fields $\Omega$. Therefore we obtain (using also the assumption $F_\theta(t,-\rho)=\overline{F_\theta}(t,\rho)$)
\begin{equation}
\begin{split}\partial_ju(t,x)&=\frac{\omega_j}{r}\Im\int_0^{\infty}\rho e^{i\rho(r-t)}F_\omega(t,\rho)\,\mathrm{d}\rho+\frac{\omega_j}{r}\Im\int_0^\infty\rho e^{i\rho(-r-t)}F_{-\omega}(t,\rho)\,\mathrm{d}\rho+(\mathrm{error}_2)\\
&=\frac{\omega_j}{2r}(\partial_q\mathcal{F}^{-1}F_\omega(t,r-t)+\partial_q\mathcal{F}^{-1}F_{-\omega}(t,-r-t))+(\mathrm{error}_2),
\end{split}
\end{equation} where the error term satisfies
\[|(\mathrm{error}_2)|\lesssim \varepsilon2^{-m-m/10},
\]
and is thus bounded by the right hand side of (\ref{transfer}) (note also that $|r|\approx 2^m$).
The corresponding bounds with several applications of $\nabla$ and $\Omega$ can be proved in the same way.

It remains to bound $|\partial_q\mathcal{F}^{-1}F_{-\omega}(t,-r-t)|$; 
for this we recall that $F_\theta(t,\rho)=\widehat{f}(t,\rho\theta)$, and decompose $f$ into $Q_{jk}f$. 
Since $|-r-t|\approx 2^m$ we may assume $j\geq m-10$; using also the Hausdorff-Young inequality and \eqref{fhatimp}
we may assume $k\geq -m/10$. Then, using the bound \eqref{aprioriLinfty} 
and Hausdorff-Young again we deduce that \[|\partial_q\mathcal{F}^{-1}F_{-\omega}(t,-r-t)|\lesssim 2^{-m/10},\] which gives what we need.
The corresponding bounds with vector fields easily follow as before.
\end{proof}

\begin{proof}[Proof of Theorem \ref{Mainth2}]
We may assume $|t|\gg 1$. %The estimate when $||x|-t|\gtrsim |t|^{\gamma/10}$ follows from Klainerman-Sobolev and the energy bound \eqref{Ali0}.
The asymptotic description \eqref{Mainth24} for $||x|-t|\lesssim |t|^{\gamma/10}$ 
follows from combining Propositions \ref{finaldata} and \ref{transfer0}, and the relationship \eqref{defU}-\eqref{defU'} 
between $F_{\theta}$ and $U_{\theta}$.
\end{proof}

%\appendix
%\section{Supporting material}\label{App}

\end{document}